\documentclass[a4paper,10pt, reqno, english]{amsart}

\usepackage{amsmath,amssymb,latexsym, amsfonts}
\usepackage{verbatim}
\usepackage{newcent}
\usepackage{slashed}
\usepackage{mathbbol}
\usepackage[dvipsnames]{xcolor}
\usepackage[colorlinks=true, allcolors=RawSienna]{hyperref}
\usepackage{mathabx}
\usepackage{bbm}
\usepackage{mathrsfs}
\usepackage{tikz-cd}

\usepackage{mathtools} 

\usepackage{array}

\usepackage{tikz}

\usepackage[all,ps, cmtip]{xy}

\usepackage{color}
\definecolor{lightgrey}{rgb}{0.7, 0.7, 0.7}

\usepackage{stmaryrd}

\usepackage{capt-of}

\usepackage[OT2, T1]{fontenc} 

\makeatletter
\@namedef{subjclassname@2020}{%
  \textup{2020} Mathematics Subject Classification}
\makeatother

\usepackage[textsize=scriptsize]{todonotes}

\usepackage{scalerel,stackengine}

\DeclareRobustCommand{\SkipTocEntry}[5]{}

\newcommand\pig[1]{\scalerel*[5pt]{\big#1}{%
  \ensurestackMath{\addstackgap[1.5pt]{\big#1}}}}

\usepackage[main=english, russian]{babel}

\numberwithin{equation}{section}

\theoremstyle{plain}

\newtheorem{theorem}{Theorem}[section]
\newtheorem*{theorem*}{Theorem}

\newtheorem{prop}[theorem]{Proposition}
\newtheorem{lemma}[theorem]{Lemma}
\newtheorem{lem}[theorem]{Lemma}

\newtheorem{cor}[theorem]{Corollary}

\newtheorem*{conj*}{Conjecture}

\theoremstyle{definition}
\newtheorem{definition}[theorem]{Definition}
\newtheorem{dfn}[theorem]{Definition}

\newtheorem{rem}[theorem]{Remark}

\newcommand{\ehhe}{{\scriptscriptstyle{\mathbb{E}}}}
\newcommand{\effe}{{\scriptscriptstyle{\mathbb{F}}}}
\newcommand{\ggii}{{\scriptscriptstyle{\mathbb{G}}}}
\newcommand{\akka}{{\scriptscriptstyle{\mathbb{H}}}}
\newcommand{\akkajappa}{{\scriptscriptstyle{H}}}

\newcommand{\emme}{{\scriptscriptstyle{M}}}

\newcommand{\uhhu}{{\scriptscriptstyle{U}}}

\newcommand{\ikks}{{\scriptscriptstyle{X}}}

\newcommand{\zett}{{\scriptscriptstyle{Z}}}

 \newcommand{\E}{{\mathbb{E}}}
 \newcommand{\F}{{\mathbb{F}}}
\newcommand{\G}{{\mathbb{G}}}

 \newcommand{\kkk}{{\mathbb{k}}}

\newcommand{\frkg}{{\mathfrak g}}

\newcommand{\gd}{\delta} 
\newcommand{\gD}{\Delta} 
\newcommand{\gve}{\varepsilon} 
  
\newcommand{\gvf}{\varphi}  

\newcommand{\gl}{\lambda}

\newcommand{\gr}{\gvr} 
\newcommand{\gvr}{\varrho} 

\newcommand{\gs}{\sigma} 
\newcommand{\gS}{\Sigma}

\newcommand{\cD}{{\mathcal D}}

\newcommand{\cH}{{\mathcal H}}

\newcommand{\cO}{{\mathcal O}}

\newcommand{\Hom}{\operatorname{Hom}}

\newcommand{\ad}{{\operatorname{ad}}}
\newcommand{\coad}{{\operatorname{{coad}}}}

\newcommand{\Tor}{{\rm Tor}}
\newcommand{\Ext}{{\rm Ext}}
\newcommand{\ext}{\mathscr{E}\hspace*{-1pt}xt}

\newcommand{\Cotor}{\operatorname{Cotor}}

\newcommand{\id}{{\rm id}}

\newcommand{\g}{{\frkg}}                                                        

\newcommand{\pr}{{\rm pr} \,}

\newcommand{\due}[3]{{}_{{#2 }} {#1}_{{ #3}}\,}    

\newcommand{\pl}{\partial}

\newcommand{{\Hl}}{{H^{\ell}}} 
\newcommand{{\mHop}}{{m_{H^{\rm op}}}} 
\newcommand{{\Hop}}{{H^{\rm op}}} 
\newcommand{{\mUop}}{{m_{U^{\rm op}}}} 
\newcommand{{\mUopp}}{{m_{\scriptscriptstyle{U^{\rm op}}}}} 
\newcommand{{\Uop}}{{U^{\rm op}}}
\newcommand{{\mVop}}{{m_{V^{\rm op}}}} 
\newcommand{{\Vop}}{{V^{\rm op}}}  
\newcommand{{\Ae}}{{A^{\rm e}}}
\newcommand{{\Be}}{{B^{\rm e}}}
\newcommand{{\Ue}}{{U^{\rm e}}}
\newcommand{{\He}}{{H^{\rm e}}}
\newcommand{{\Aop}}{{A^{\rm op}}}
\newcommand{{\Aope}}{({A^{\rm op}})^{\rm e}}
\newcommand{{\Aopl}}{{A^{\rm op}_\pl}}

\newcommand{{\Bop}}{{B^{\rm op}}}
\newcommand{{\Bopp}}{{\scriptscriptstyle{{B^{\rm op}}}}}
\newcommand{{\Bope}}{({B^{\rm op}})^{\rm e}}
\newcommand{{\Bpl}}{{B_\pl}}

\newcommand{{\op}}{{{\rm op}}}
\newcommand{{\coop}}{{{\rm coop}}}
\newcommand{{\sop}}{{*^{\rm op}}}
\newcommand{{\co}}{{{\rm co}}}
\newcommand{{\diag}}{{{\rm diag}}}

\newcommand{\amoda}{A^{\rm e}\mbox{-}\mathbf{Mod}}                  %
\newcommand{\umod}{U\mbox{-}\mathbf{Mod}}                     
\newcommand{\yd}{{}^\uhhu_\uhhu\mathbf{YD}}                     %

\newcommand{\hyd}{{}^\akkajappa_\akkajappa\mathbf{YD}}                     %

\newcommand{\ydr}{{}_\uhhu\!\mathbf{YD}{}^\uhhu} 
\newcommand{\hydr}{{}_\akkajappa\!\mathbf{YD}{}^\akkajappa}

\newcommand{\comodu}{\mathbf{Comod}\mbox{-}U}         
\newcommand{\ucomod}{U\mbox{-}\mathbf{Comod}}

\newcommand{\lact}{\smalltriangleright}                  
\newcommand{\ract}{\smalltriangleleft}
\newcommand{\blact}{\blacktriangleright}  
\newcommand{\bract}{\blacktriangleleft}

\newcommand{{\gog}}{{G \rightrightarrows G_0}}

\newcommand{{\rra}}{\rightrightarrows}

\newcommand{{\lra}}{\ \longrightarrow \ }
\newcommand{{\lla}}{\ \longleftarrow \ }
\newcommand{{\lma}}{\ \longmapsto \ }

\newcommand{{\bull}}{{\scriptscriptstyle{\bullet}}}
\newcommand{{\qqquad}}{{\quad\quad\quad}}

\newsavebox{\foobox}

\begin{document}

\title{Brackets and products from centres in extension categories} 

\author{Domenico Fiorenza}
\author{Niels Kowalzig}

\begin{abstract}
Building on Retakh's approach to $\Ext$ groups through categories of extensions, Schwede reobtained the well-known Gerstenhaber algebra structure on $\Ext$ groups over bimodules of associative algebras both from splicing extensions (leading to the cup product) and from a suitable loop in the categories of extensions (leading to the Lie bracket). We show how Schwede's construction admits a vast generalisation to general monoidal categories with coefficients of the $\Ext$ groups taken in (weak) left and right monoidal (or Drinfel'd) centres. In case of the category of left modules over bialgebroids and coefficients given by commuting pairs of braided (co)commutative (co)monoids in these categorical centres, we provide an explicit description of the 
algebraic structure obtained this way, and a complete proof that this leads to a Gerstenhaber algebra is then obtained from an operadic approach. This, in particular, considerably generalises the classical construction given by Gerstenhaber himself.
Conjecturally, the algebraic structure
we describe should produce a Gerstenhaber algebra for an arbitrary monoidal category enriched over abelian groups, but even the bilinearity of the cup product and of the Lie-type bracket defined by the abstract construction in terms of extension categories remain elusive in this general setting.
  \end{abstract}

\address{Dipartimento di Matematica, Universit\`a degli Studi di Roma La
Sapienza, P.le Aldo Moro 5, 00185 Roma, Italy}
\email{fiorenza@mat.uniroma1.it}

\address{Dipartimento di Matematica, Universit\`a di Roma Tor Vergata, Via della Ricerca Scientifica 1, 00133 Roma, Italy}
\email{kowalzig@mat.uniroma2.it}



\vspace*{-.6cm}

\maketitle

\vspace*{-.6cm}

\tableofcontents

\section*{Introduction}

An apparently quite wide-spread theorem roughly states that for a monoidal category $({\mathscr{C}}, \otimes, \mathbb{1})$ fulfilling some mild conditions, the groups $\Ext_{\mathscr{C}}(\mathbb{1},  \mathbb{1})$ define a Gerstenhaber algebra: that is, a graded object equipped with a graded Lie bracket and a graded commutative product subject to a Leibniz rule. However, a detailed proof has appeared only very recently in  \cite{LowVdB:TBISOTDEAOTUIAMC} in a vast generality, using the language of properads.
This builds on
\cite{Sho:HATANFMC, Sho:DGCADC}, see also \cite{Her:MCATGBIHC} for an advanced attempt and
\cite{Volkov-Whiterspoon} for a treatment in terms of coderivations.

The {\em mild} conditions just mentioned refer to the problem that for an arbitrary monoidal category a priori the monoidal structure is not exact, and one hence assumes the existence of a full additive subcategory which has the desired properties and is implicitly dealt with instead of the original monoidal category.
The precise conditions needed, which we prefer not to discuss here in detail, are spelled out in \cite[\S2.1]{Sho:HATANFMC} or, somewhat differently, in \cite[Def.~5.1]{Sho:DGCADC},
or \cite[Lem.~2.1]{Schw:AESIOTLBIHC}
for a discussion adapted to the Hochschild context of associative algebras.

Less known and so far unproven is the following enhancement of the above statement, which we formulate as a conjecture here:

\begin{conj*}
Let ${\mathscr{C}}$ be a monoidal category (fulfilling a few mild conditions),
let $Z$ be a braided commutative monoid in the left weak monoidal centre of ${\mathscr{C}}$, 
and let $X$ be a braided cocommutative comonoid in the right weak monoidal centre of ${\mathscr{C}}$. 
If $(X,Z)$ is a commuting pair, then $\Ext_{\mathscr{C}}(X,Z)$
is a Gerstenhaber~algebra.
\end{conj*}

Here, by left resp.\ weak centre of $({\mathscr{C}}, \otimes, \mathbb{1})$ we mean those (braided monoidal) categories $\mathscr{Z}^\ell({\mathscr{C}})$ resp.\
${\mathscr{Z}}^r({\mathscr{C}})$
whose objects are objects in ${\mathscr{C}}$ together with not necessarily invertible natural maps $Z \otimes M \to M \otimes Z$ resp.\ $M \otimes X \to X \otimes M$ 
that for arbitrary $M \in {\mathscr{C}}$ fulfil certain (hexagon) compatibilities in a cus\-to\-ma\-ry sense, and {\em commuting pair} refers to the case where these two maps coincide on $Z \otimes X$, see \S\ref{centres} for a more precise definition.

That these central objects need further structure such as a multiplication 
resp.\ comultiplication
in order to define a Gerstenhaber algebra on their Ext groups becomes clear, {\em e.g.}, when observing that the usual Yoneda product is not an internal operation on $\Ext_{\mathscr{C}}(X,Z)$.

\par

In \S\ref{opihashi}, we motivate the above conjecture by mimicking Schwede's description \cite{Schw:AESIOTLBIHC}
of the Gerstenhaber algebra structure on the Hochschild cohomology of an associative algebra through Retakh's enhancement of $\Ext$ groups to $\ext$ spaces.
More precisely, if we denote by $\#$ the splicing of two extensions in $\ext$ (that induces the Yoneda product on $\Ext$ groups), and if we assume $X$ to be a comonoid in ${\mathscr{Z}}^r({\mathscr{C}})$ with comultiplication $\gD_\ikks$ and $Z$ a monoid in ${\mathscr{Z}}^\ell({\mathscr{C}})$ with multiplication $\mu$, we can consider both multiplication resp.\ comultiplication as extensions of length zero, and so as objects in $\ext_{\mathscr{C}}^0(X, X \otimes X)$ and $\ext_{\mathscr{C}}^0(Z \otimes Z, Z)$, respectively. Given two extensions $\E \in \ext_{\mathscr{C}}^p(X,Z)$ and $\F \in \ext_{\mathscr{C}}^q(X,Z)$ of length $p$ resp.\ $q$, one can then consider the splicing
$$
\E \cup \F \coloneq   \mu \# (\E \otimes Z)\# (X\otimes \F) \# \gD_\ikks
$$
and its connected component in $\ext_{\mathscr{C}}^{p+q}(X, Z)$, that for convenience will be denoted by the same symbol:
\[
\E\cup\F\in \pi_0\ext_{\mathscr{C}}^{p+q}(X, Z)\simeq \Ext_{\mathscr{C}}^{p+q}(X,Y).
\]
On top, if $\mathbf{Moloch}(\E, \F)$ denotes the truncated tensor product complex of two extensions $\E$ and $\F$,
one can devise the
following loop in $\ext_{\mathscr{C}}^{p+q}(X,Z)$:

\begin{equation*}
\xymatrix@C=0.5em{
  &&
    \mu \# \mathbf{Moloch}(\E,\F )\# \gD_\ikks 
     \ar[lld]
 \ar[rrd]
  &&
  \\
  \hbox to 5em{\hss$
    \E \cup \F=   \mu \# (\E \otimes Z)\# (X\otimes \F) \# \gD_\ikks 
\hss$}
  \ar@{<-}[dd]
  &&
    &&
      \hbox to 5em{\hss$
 \mu \#(Z\otimes\F) \#(\E \otimes X) \# \gD_\ikks 
$\hss}
      \\ \\
            \hbox to 5em{\hss$
               \mu \# (Z\otimes\E) \#(\F \otimes X)  \# \gD_\ikks 
$\hss}
              &&
   &&
         \hbox to 5em{\hss$
           \F\cup \E = \mu \#(\F\otimes Z) \#(X\otimes\E) \# \gD_\ikks 
$\hss}
           \ar@{<-}[uu]
  \\ 
  &&
   \mu \# \mathbf{Moloch}(\F,\E )  \# \gD_\ikks 
\ar[llu]
  \ar[rru]
}
\end{equation*}   
where the vertical arrows use in an essential way that $X$ and $Z$ form a commuting pair in the right and left weak centres of ${\mathscr{C}}$, respectively.

Choosing $\E \cup \F$ as a base point in $\ext_{\mathscr{C}}^{p+q}(X, Z)$, this gives an element 
\[
\{\E,\F\} \in
\pi_1(\ext_{\mathscr{C}}^{p+q}(X, Z),\E \cup \F)\simeq \Ext_{\mathscr{C}}^{p+q-1}(X,Z).
\]
This way, one has defined a degree zero cup product
\[
\cup\colon \Ext_{\mathscr{C}}^{p}(X,Z)\times \Ext_{\mathscr{C}}^{q}(X,Z) \to \Ext_{\mathscr{C}}^{p+q}(X,Z)
\]
and a degree $-1$ bracket
\[
\{-,-\}\colon \Ext_{\mathscr{C}}^{p}(X,Z)\times \Ext_{\mathscr{C}}^{q}(X,Z) \to \Ext_{\mathscr{C}}^{p+q-1}(X,Z).
\]
Unfortunately, even in the case of the category of bimodules over associative algebras, this topological approach leaves one one step before the conclusion: a proof that these operations define a Gerstenhaber algebra or the mere fact these operations are bilinear remain elusive and no simpler argument than those used in \cite{Sho:HATANFMC} seems to be available to establish these facts in a purely topological fashion.

In case of the category of bimodules over associative algebras, however, there is a convenient mixed topological/algebraic approach developed in the already mentioned work \cite{Schw:AESIOTLBIHC} by Schwede, which leads to a transparent proof of the fact that the cup product and the bracket obtained from a similar loop as above indeed do endow the groups $\Ext_\Ae(A,A)$ of an associative algebra $A$ with a  Gerstenhaber algebra structure: it is shown in {\em op.~cit.}\ how to obtain an explicit description in terms of  Hochschild cocycle representatives of the cup and bracket operations (dictated by the geometry of $\ext$ spaces) for which the Gerstenhaber algebra axioms can be directly checked, and which coincide with the cup product and Lie bracket of the classical Gerstenhaber algebra structure on $\Ext_\Ae(A,A)$ as originally introduced in \cite{Ger:TCSOAAR}.

Following the mentioned approach, we will show how to extract a cocycle description for the cup and bracket operations derived from the geometry of $\ext$ spaces of a commuting pair $(X,Z)$.
As even an accurate construction will not allow us here to obtain a proof for a Gerstenhaber algebra structure in full generality, we will be rather informal in this derivation, omitting all the needed technical (mainly categorical) assumptions in order to make the construction completely rigorous. Yet, we will be accurate enough to be able to provide in \S\ref{sec:explicit} an explicit cocycle description for the operations in the particular case that $\mathscr{C}$ is the monoidal category of left $U$-modules for a left bialgebroid $(U,A)$. In terms of this cocycle description we shall explicitly exhibit the Gerstenhaber algebra structure, thus turning the above Conjecture into Theorem \ref{extobroid}:

\begin{theorem*}[Theorem \ref{extobroid}]
  Let $(U,A)$ be a left bialgebroid, $Z$ a braided commutative monoid in the left weak centre of the monoidal category of left $U$-modules, 
and $X$ a braided cocommutative comonoid in its right weak  centre such that $(X,Z)$ constitutes a commuting pair.
  Then the cochain complex $\Hom_U(\mathrm{Bar}_\bullet(U,X),Z)$, computing $\Ext^\bull_U(X,Z)$ 
  for $A$-projective $U$,
  defines an operad with multiplication, which induces a Gerstenhaber algebra structure on the cohomology groups.
  \end{theorem*}

Taking $U=\Ae$ and $X=Y=A$, this recovers the usual Gerstenhaber al\-ge\-bra structure on the standard Hochschild cohomology of an associative algebra $A$ with coefficients in itself as a very particular case.

\vskip .4 cm

\begin{center}
* \ \, * \ \, *
\end{center}

\vskip .05 cm

  Finding higher structures on chain or cochain complexes, or more specifically on cohomology groups such as $\Ext$ groups, has risen some interest in the last two decades.
  For example, in \cite{FarSol:GSOTCOHA} 
  the case for a Hopf algebra over a field $\kkk$ was treated by identifying $\Ext^\bull_H(\kkk,\kkk)$ as a subalgebra of the Hochschild cohomology $\Ext^\bull_{\He}(H,H)$, where $H$ is merely seen as a $\kkk$-algebra. This has been elaborated on in \cite{Men:CMCMIALAM} by establishing a duality relation to the Gerstenhaber structure on $\Cotor$ groups using operadic techniques.
Both approaches have been generalised in \cite{Kow:BVASOCAPB} to bialgebroids using the centre construction, which, in particular, allows to establish that the well-known isomorphism
\begin{equation}
  \label{falegnameria}
  \tag{$\dagger$}
  \Ext^\bull_H(\kkk, \ad(H))
  = \Ext^\bull_{\He}(H,H)
\end{equation}
of $\kkk$-modules 
  is actually one of Gerstenhaber algebras. Here, $\ad(H)$ refers to $H$ itself seen as a left $H$-module with respect to the left adjoint action.

  Independently from our work at hand,
in the context of finite tensor categories,
  a Gerstenhaber algebra structure on $\Ext^\bull_{\mathscr{C}}(X,Z)$ defined through a coproduct on $X$, a product on $Z$, and the respective lifts of these structure to the monoidal center, 
  appeared in Section 3 of the first (arXiv) version of
  \cite{SchweiWoi:TDGVFATDC}, compare \cite{SchweiWoi:HIOBCAATDCFFTC} as well for a separate treatment by the same authors.
This provides
a vast class of examples with applications relevant in quantum topology. In particular, due to the finiteness assumption on the tensor category, one has distinguished choices for $X$ and $Z$, in addition to the monoidal unit, given by the canonical coend and the canonical end, respectively.
Among the applications, see \cite[Ex.~5.10]{SchweiWoi:HIOBCAATDCFFTC}, one recovers \eqref{falegnameria} for a finite dimensional quantum group or Hopf algebra.

However, as the result in \cite{Kow:BVASOCAPB} shows, a finiteness assumption is not needed in order for an isomorphism of Gerstenhaber algebras as in \eqref{falegnameria}
to hold, and hence one obtains a large class of examples 
  from infinite dimensional quantum groups.
  Other interesting infinite dimensional examples include, in the same spirit,
$\Ext^\bull_H(\coad(H), \kkk)$ or, 
from a somewhat different viewpoint, $\Ext^\bull_{\cD(H,G)}(\kkk, \cH(H,G))$ for a generalised Heisenberg double $\cH(H,G)$ over a generalised Drinfel'd double $\cD(H,G)$ for two arbitrary Hopf algebras $H$ and $G$, as detailed in \S\ref{examples}.
Crossing braided commutative monoids in the monoidal centre of the category of $H$-modules with the Hopf algebra $H$ in question leads to another class of examples related to left Hopf algebroids.

\vskip .4 cm

\begin{center}
* \ \, * \ \, *
\end{center}

\vskip .05 cm

Throughout the whole text $\kkk$ will be a commutative ring, of characteristic zero if need be.
Unadorned tensor products are {\em not} to be understood over $\kkk$ but rather denote the product in a monoidal category.

\section{Operations in extension categories}
\label{opihashi}

We begin by recalling, in an informal way,  the basic definitions of the $\ext$ spaces we are going to use in the construction of the cup product and the bracket operation on $\Ext$ groups. Our aim here is not to provide a rigorous construction but rather to fix notation and orient the reader. Details on the formal construction and proofs of the statements can be found, {\em e.g.},  in the original works by Retakh on $\ext$ spaces \cite{Ret:HPOCOE} or by Neeman and Retakh \cite{NeeRet:ECATH}.

We will be working in a fixed monoidal category ${\mathscr{C}}$ and will denote the objects of ${\mathscr{C}}$ by the symbols $X, Y, Z,\dots$. We will also assume ${\mathscr{C}}$ is close enough to an abelian category with enough projectives  so that expressions like ``zero object'', ``short exact sequence'' or ``projective resolution'' make sense in ${\mathscr{C}}$. In particular, in ${\mathscr{C}}$ we will have a notion of an extension of an object $X$ by an object $Y$, meaning by this a short exact sequence of the form
\[
0\to Y\to E\to X\to 0
\]
in ${\mathscr{C}}$. Iterating this construction, one gets the notion of a $p$-fold extension of an object $X$ by an object $Y$, which will be denoted by
\[
\E \colon\quad 0 \to  Y\to E_{p-1}\to\cdots\to E_0\to X \to 0,
\]
so that
\[
\E _i=
\begin{cases}
0&\text{ if } i\leq -2\\
X&\text{ if } i=-1\\
E_i&\text{ if } 0\leq i\leq p-1\\
Y&\text{ if } i=p\\
0&\text{ if } i\geq p+1.
\end{cases}
\]
The maps $\E _i\to \E _{i-1}$ will be denoted by $d_{\E ,i}$ or simply by $d_i$, $d_{\E }$, or still simpler by $d$ when no or not much confusion is likely to arise. 
Occasionally, to make it manifest that $X$ is in place or degree $-1$ and $Y$ in place or degree $p$, we will formally use the degree $j$ shift symbol $[j]$, familiar from the theory of triangulated categories, and will write a $p$-fold extension of $X$ by $Y$ as
\[
\E \colon\quad 0 \to  Y[p]\to E_{p-1}\to\cdots\to E_0\to X[-1] \to 0.
\]
In other words, we will be implicitly assuming that objects without an index are placed in degree zero. More generally, an object denoted, for example, $E_i[j]$, is to be thought as placed in degree $i+j$.

If $\E $ and $\F$ are two $p$-fold extensions of $X$ by $Y$, a {\em morphism}  $f \colon \E \to \F$ {\em of $p$-fold extensions} is defined as a commutative diagram
\begin{equation}
  \label{morfi}
\xymatrix@C=2.1em{
  0 \ar[r]& Y\ar@{=}[d]
  \ar[r]& E_{p-1}\ar[d]^{f_{p-1}}\ar[r]& \cdots \ar[r]& E_1\ar[r]\ar[d]^{f_1}& E_0\ar[d]^{f_0}\ar[r]& X \ar[r]\ar@{=}[d]
  & 0\\
0 \ar[r]& Y\ar[r]& F_{p-1}\ar[r]& \cdots \ar[r]& F_1\ar[r]&F_0\ar[r]& X \ar[r]& 0
}
\end{equation}
that is, where the leftmost and rightmost vertical arrows are the identity.
If $f, g\colon \E  \to \F$ are two morphisms of $p$-fold extensions, then a {\em chain homotopy} between $f$ and $g$ is a degree $+1$ morphism $s$ of graded objects 
from $\E $ to $\F$ such that $f-g=[d,s]$. Explicitly,  remembering that
both $f$ and $g$
are identities
in degree $-1$ and $p$
and that both $\E _i$ and $\F_i$ are zero outside the range $[-1,p]$, this amounts to the datum of 
a family of morphisms
$s_i\colon \E _i \to \F_{i+1}$, for $i = -1, 0, 1, \ldots, p$, such that
\begin{equation*}
  \begin{array}{rcll}
    d_0 \circ s_{-1} &=& 0, &
    \\
    d \circ s_{i} + s_{i-1} \circ d &=& f_i - g_i  & \mbox{for} \ 0 \leq i \leq p-1,
    \\
   s_{p-1} \circ d_{p-1}&=& 0. &
  \end{array}
\end{equation*}
Similarly, one defines homotopies between homotopies, {\em etc}. This way, one defines an $\infty$-category $\Ext^p_{\mathscr{C}}(X,Y)$ whose objects are  $p$-fold extensions of $X$ by $Y$, whose $1$-morphisms are morphisms of $p$-fold extensions, whose $2$-morphisms are homotopies between morphisms, and so on. The space $\ext^p_{{\mathscr{C}}}(X,Y)$ is defined as the topological realisation of the simplicial set defined by this $\infty$-category $\Ext^p_{\mathscr{C}}(X,Y)$, that is, the simplicial set having
the $\kkk$-morphisms in $\Ext^p_{\mathscr{C}}(X,Y)$ as $\kkk$-simplices.
  For $p=0$, the space $\Ext^0_{\mathscr{C}}(X,Y)$ is defined as the set $\Hom_{\mathscr{C}}(X,Y)$ endowed with discrete topology.
The two basic properties of the $\ext$ spaces that we shall need and are going to use are the following relations to $\Ext$ groups in ${\mathscr{C}}$, which run under the name {\em Retakh's isomorphism} \cite{Ret:HPOCOE, NeeRet:ECATH}:
\begin{align*}
  \pi_0\ext^p_{\mathscr{C}}(X,Y)&=\Ext^p_{\mathscr{C}}(X,Y),
  \\
\pi_1(\ext^p_{\mathscr{C}}(X,Y);\E )&=\Ext^{p-1}_{\mathscr{C}}(X,Y),
\end{align*}
where the first line can be taken as a definition of the $\Ext$ group $\Ext^p_{\mathscr{C}}(X,Y)$. When ${\mathscr{C}}$ is an abelian category with enough projectives, this is equivalent to the classical definition of the $\Ext$ groups in ${\mathscr{C}}$.

\subsection{The hash operation}
\label{hashhash}
The {\em hash operation}, also known as {\em splicing}, is a concatenation type operation 
\[
\#\colon \ext_{\mathscr{C}}^p(Y,Z)\times \ext_{\mathscr{C}}^q(X,Y) \to \ext_{\mathscr{C}}^{p+q}(X,Z)
\]
on $\ext$ spaces,
implementing the Yoneda
product 
\[
\circ \colon \Ext_{\mathscr{C}}^p(Y,Z)\otimes \Ext_{\mathscr{C}}^q(X,Y) \to \Ext_{\mathscr{C}}^{p+q}(X,Z)
\]
on $\Ext$ groups
by passing to path connected components. It is defined in slightly different ways depending on whether $p$ or $q$ are zero or not.
When both $p$ and $q$ are zero, $\ext_{\mathscr{C}}^0(Y,Z)$, $\ext_{\mathscr{C}}^0(X,Y)$, and $\ext_{\mathscr{C}}^{0}(X,Z)$ are just the hom sets $\Hom_{\mathscr{C}}(Y,Z)$, $\Hom_{\mathscr{C}}(X,Y)$ and $\Hom_{\mathscr{C}}(X,Z)$, respectively, and the hash operation in this case is the composition of homomorphisms.

When $p>0$ and $q=0$, we are considering a hash of the form $\E \#\F$,
where $\F \colon X\to Y$ is a morphism. It is defined as the top horizontal row of the commutative diagram
\begin{equation}
  \label{pzero}
\xymatrix@C=2.1em{
  0 \ar[r]& Z\ar@{=}[d]
  \ar[r]& E_{p-1}\ar@{=}[d]
  \ar[r]& \cdots \ar[r]& E_1\ar[r]\ar@{=}[d]
  & E_0\times_{Y}X\ar[d]\ar[r]& X \ar[r]\ar[d]^{\F }& 0\\
0 \ar[r]& Z\ar[r]& E_{p-1}\ar[r]& \cdots \ar[r]& E_1\ar[r]&E_0\ar[r]& Y \ar[r]& 0
}
\end{equation}
where the rightmost commutative square is a pullback and the map $E_1\to E_0\times_{Y}X$ is induced by the commutative diagram 
\[
\xymatrix@R=2em@C=2.1em{
E_1\ar[r]^{0}\ar[d]&  X \ar[d]^{\F}\\
E_0\ar[r]& Y 
}
\]
and by the universal property of the pullback.

In case $p=0$ and $q>0$, dually, 
with $\E\colon Y\to Z$ a morphism, $\E \#\F$
is defined as the bottom horizontal row of the commutative diagram
\begin{equation}
  \label{qzero}
\xymatrix@C=2.1em{
  0 \ar[r]& Y\ar[d]^{\E}\ar[r]& F_{q-1}\ar[d]\ar[r]& F_{q-2}\ar@{=}[d]
  \ar[r]& \cdots \ar[r]& F_0\ar@{=}[d]
  \ar[r]& X \ar[r]\ar@{=}[d]
  & 0\\
0 \ar[r]& Z\ar[r]\ar[r]& Z\sqcup_Y F_{q-1}\ar[r]& F_{q-2}\ar[r]& \cdots \ar[r]& F_0\ar[r]& X \ar[r]& 0
}
\end{equation}
where the leftmost commutative square is a pushout.

It remains to be said what happens
when both $p$ and $q$ are greater than zero. In this case, points in the spaces $\ext_{\mathscr{C}}^p(Y,Z)$ and $\ext_{\mathscr{C}}^q(X,Y)$ are iterated extensions
\[
\E \colon
\quad 0 \to  Z\to E_{p-1}\to\cdots\to E_0\to Y \to 0,
\]
and
\[
\F\colon\quad 0 \to Y\to F_{q-1}\to \cdots \to F_0\to X \to 0,
\]
respectively. Their hash $\E \#\F$
is the iterated extension given by the top horizontal row in  the concatenation
\[
\xymatrix@C=1.1em@R=.7em
             {
\E \#\F\colon\quad 0 \ar[r] & Z\ar[r]& E_{p-1}\ar[r] & \cdots \ar[r]& E_0\ar[rr]
  && F_{q-1}\ar[r] &\cdots\ar[r]& F_0\ar[r]& X \ar[r] & 0.
  \\
  &&&&&Y
\ar`[ru][ur]
\ar@{<-}`l[ul][ul]
             }
             \]

             Notice that the hash composition is associative (up to natural isomorphisms)
                            and unital with units given by the compositions $X\xrightarrow{\epsilon_X}\mathbb{1}\xrightarrow{1_Z}Z$, where $\mathbb{1}$ is the unit object of $\mathscr{C}$ and $\epsilon_X$ and $1_Z$ are the counit of $X$ and the unit of $Z$, respectively.

\begin{rem}
               \label{rem:degree-shift}
In the above concatenation, the lower indices no more correspond to the position in the sequence of iterated extensions. This may be quite confusing, so occasionally we will formally use the shift operator $[1]$ of triangulated categories as a symbol to restore the correct correspondence. That is, the iterated extension 
\[
\E \colon\quad 0 \to  Y\to E_{p-1}\to\cdots\to E_0\to X \to 0
\]
will be rewritten as
\[
\E \colon\quad 0 \to  Y[p]\to E_{p-1}\to\cdots\to E_0\to X[-1] \to 0,
\]
to stress that $X$ is in position (or degree) $-1$ in the sequence, and $Y$ in position $p$.
With this notation, the iterated extension $\E \#\F$ will read
\[
0 \to Z[p+q]\to E_{p-1}[q]\to \cdots \to E_0[q]\to F_{q-1}\to\cdots\to F_0\to X[-1] \to 0.
\]
Yet, when no confusion can arise, we will still adopt the less cumbersome
\[
0 \to Z\to E_{p-1}\to \cdots \to E_0\to F_{q-1}\to\cdots\to F_0\to X \to 0.
\]
\end{rem}

\subsection{The tensor product}\label{sec:tensor}
By using the monoidal structure in $\mathscr{C}$ and the hash composition, one has (at least) four natural ways of defining a map
\[
\ext_{\mathscr{C}}^p(X,Z)\times \ext_{\mathscr{C}}^q(X,Z)\to \ext_{\mathscr{C}}^{p+q}(X\otimes X,Z\otimes Z).
\]
Namely, given  $\E $ in  $\ext_{\mathscr{C}}^p(X,Z)$ and $\F$ in  $\ext_{\mathscr{C}}^q(X,Z)$, one can use the hash composition jointly with the tensor product to form
\begin{enumerate}
\item $(\E \otimes Z)\#(X\otimes \F) $,
\item $(Z\otimes \F)\#(\E \otimes X) $,
\item $(\F\otimes Z)\#(X\otimes \E ) $,
\item $(Z\otimes \E )\#(\F\otimes X) $.
\end{enumerate}

Both hashing and tensoring are well defined on path connected components and so the above define four operations 
\[
 \Ext_{\mathscr{C}}^p(X,Z)\times \Ext_{\mathscr{C}}^q(X,Z)\to \Ext_{\mathscr{C}}^{p+q}(X\otimes X,Z\otimes Z).
\]
Quite remarkably, the two operations induced on the $\Ext$ groups by {\it(i)} and {\it(ii)} coincide, as well as the two operations induced by {\it(iii)} and {\it(iv)}.
To see this, we need to exhibit a path in $\ext_{\mathscr{C}}^{p+q}(X\otimes X,Z\otimes Z)$ between $(\E \otimes Z) \# (X \otimes \F)$
and $(Z\otimes \F) \# (\E \otimes X)$.
We show how to do this in case $p,q\geq 1$, leaving to the reader the not particularly compelling task of completing the proof in case $p$ or $q$ are equal to zero. An easy way of realising such a path is as a span

\begin{equation}\label{eq:span}
(\E \otimes Z)\# (X\otimes \F)\xleftarrow{\quad\lambda_{\ehhe,\effe}\quad} \mathbf{Moloch}(\E ,\F)\xrightarrow{\quad\gvr_{\ehhe,\effe}\quad}(Z\otimes \F)\#(\E \otimes X) ,
\end{equation}
where $\mathbf{Moloch}(\E ,\F)$ is the totalisation of the diagram obtained by partially removing the top row and the rightmost column of $\E \otimes\F$ as follows:
\[
\xymatrix{ & & & &
 \hbox to 2em{\hss$
  X\otimes X
$\hss}
  \\
  Z\otimes F_0\ar[r] &E_{p-1}\otimes F_0\ar[r]&\cdots \ar[r]& E_0\otimes F_0\ar[ur] &
  \\
  \stackrel{\vdots}{\phantom{a}} \ar[u] & \stackrel{\vdots}{\phantom{a}}\ar[u] & \stackrel{\vdots}{\phantom{a}} \ar[u]&  \stackrel{\vdots}{\phantom{a}}\ar[u] &
  \\
Z\otimes F_{q-1}\ar[r]\ar[u] &E_{p-1}\otimes F_{q-1}\ar[r]\ar[u] &\cdots \ar[r]\ar[u]& E_0\otimes F_{q-1}\ar[u] &\\
Z\otimes Z\ar[r]\ar[u] &E_{p-1}\otimes Z\ar[r]\ar[u] &\cdots \ar[r]\ar[u]& E_0\otimes Z\ar[u] &
}
\]
One easily checks by explicit computation in the top two degrees and by the K\"unneth isomorphism in all the lower degrees that $\mathbf{Moloch}(\E ,\F)$ is indeed an object in $\ext_{\mathscr{C}}^{p+q}(X\otimes X,Z\otimes Z)$.

\begin{rem}
When $X=\mathbb{1}_{{\mathscr{C}}}$ is the unit object of the monoidal category ${\mathscr{C}}$, the complex $\mathbf{Moloch}(\E ,\F)$ is the tensor product of $\E$ and $\F$ as augmented complexes considered in \cite{Schw:AESIOTLBIHC}.
\end{rem}

To produce the span \eqref{eq:span}, notice that the iterated extension $(\E \otimes Z)\# (X\otimes \F)$ is given by the top horizontal row in 
\begin{small}
\begin{equation}
  \label{schnee2}
\xymatrix@C=1.1em@R=1em{
  Z\otimes Z\ar[r]& E_{p-1} \otimes Z\ar[r] & \cdots \ar[r]&
 \hbox to 3em{\hss$
E_0\otimes Z
  $\hss}
  \ar[rr]
  &&
 \hbox to 4em{\hss$
  X\otimes F_{q-1}
  $\hss}
  \ar[r] &\cdots\ar[r]&X\otimes F_0\ar[r]&
X\otimes X.
  \\
  &&&&
 \hbox to 3em{\hss$
   X\otimes Z
   $\hss}
\ar`[ru][ur]
\ar@{<-}`l[ul][ul]
}
\end{equation}
\end{small}
In the spirit of Remark \ref{rem:degree-shift}, one can use the shift operator to have a better control of positions or degrees and write this as
\begin{small}
\[
 Z[p]\otimes Z[q]\to E_{p-1} \otimes Z[q]\to \cdots \to
  E_0\otimes Z[q]\to
   X\otimes F_{q-1}\to\cdots\to X\otimes F_0\to X\otimes X[-1].
\] 
\end{small}
Similarly, the iterated extension $(Z\otimes \F)\#(\E \otimes X) $ is given by the top horizontal row in 
\begin{small}
\[
\xymatrix@C=1.1em@R=1em{
  Z\otimes Z\ar[r]& Z\otimes F_{q-1}\ar[r]& \cdots \ar[r]&
 \hbox to 3em{\hss$
  Z\otimes F_0
$\hss}
  \ar[rr]
  &&
 \hbox to 4em{\hss$
   E_{p-1}\otimes X
$\hss}
   \ar[r] &\cdots\ar[r]&E_0\otimes X\ar[r]&X\otimes X,
  \\
  &&&&
 \hbox to 3em{\hss$
   Z\otimes X
   $\hss}
\ar`[ru][ur]
\ar@{<-}`l[ul][ul]
}
\]
\end{small}
that is, in the degree-shifted notation,
\begin{small}
\[
 Z[p]\otimes Z[q]\to Z[p]\otimes F_{q-1}\to \cdots \to Z[p]\otimes F_0\to E_{p-1}\otimes X\to\cdots\to E_0\otimes X\to X\otimes X[-1].
\]
\end{small}
These combine with $\mathbf{Moloch}(\E ,\F)$ into the single commutative diagram
\[
\xymatrix{
 &E_{p-1}\otimes X\ar[r] &\cdots \ar[r]& E_0\otimes X\ar[r] &X\otimes X\\
  Z\otimes F_0
\ar`u[ru][ru]
  \ar[r] &E_{p-1}\otimes F_0\ar[r]\ar[u] &\cdots \ar[r]\ar[u]& E_0\otimes F_0\ar[r]\ar[u]\ar[ur] &X\otimes F_0\ar[u]
  \\
   \stackrel{\vdots}{\phantom{a}}\ar[u] & \stackrel{\vdots}{\phantom{a}}\ar[u] & \stackrel{\vdots}{\phantom{a}} \ar[u]&  \stackrel{\vdots}{\phantom{a}}\ar[u] & \stackrel{\vdots}{\phantom{a}}\ar[u]
  \\
Z\otimes F_{q-1}\ar[r]\ar[u] &E_{p-1}\otimes F_{q-1}\ar[r]\ar[u] &\cdots \ar[r]\ar[u]& E_0\otimes F_{q-1}\ar[r]\ar[u] &X\otimes F_{q-1}\ar[u]\\
Z\otimes Z\ar[r]\ar[u] &E_{p-1}\otimes Z\ar[r]\ar[u] &\cdots \ar[r]\ar[u]& E_0\otimes Z
\ar`r[ru][ru]
\ar[u]
&
}
\]
The morphism 
\begin{equation}
  \label{hiddensee1}
\lambda_{\ehhe,\effe}\colon \mathbf{Moloch}(\E ,\F) \to (\E \otimes Z)\#(X\otimes \F) 
\end{equation}
is then given by the natural projection
\[
\mathbf{Moloch}(\E ,\F)_k\xrightarrow{\pi} E_{k-q}\otimes Z[q]
\]
for $q \leq k \leq q+p$, by the composition
\[
\mathbf{Moloch}(\E ,\F)_k\xrightarrow{\pi} E_0\otimes F_k \to X\otimes F_k
\]
for $0\leq k\leq q-1$, and by the identity of $X\otimes X[-1]$ for $k=-1$. The morphism
\begin{equation}
  \label{hiddensee2}
\gvr_{\ehhe,\effe}\colon \mathbf{Moloch}(\E ,\F) \to (Z\otimes \F)\#(\E \otimes X) 
\end{equation}
is defined analogously,
that is,
by the natural projection
\[
\mathbf{Moloch}(\E ,\F)_k \xrightarrow{\pi} (-1)^{pk-p} Z[p] \otimes F_{k-p}
\]
for $p \leq k \leq p+q$, by the composition
\[
\mathbf{Moloch}(\E ,\F)_k\xrightarrow{\pi} E_k\otimes F_0 \to 
E_k \otimes X
\]
for $0\leq k\leq p-1$, and again by the identity of $X\otimes X[-1]$ for $k=-1$.

Exchanging the r\^oles of $\E $ and $\F$, we obtain the span
\[
 (\F\otimes Z)\#(X\otimes \E )\xleftarrow{\quad\lambda_{\effe,\ehhe}\quad} \mathbf{Moloch}(\F,\E )\xrightarrow{\quad\gvr_{\effe,\ehhe}\quad}(Z\otimes \E )\#(\F\otimes X) ,
\]
showing that {\it(iii)} and {\it(iv)} above define the same map at the level of $\Ext$ groups. 

We would also like $(Z\otimes \F)\#(\E \otimes X)$ and $(\F\otimes Z)\#(X\otimes \E )$ to lie in the same path connected component so that all of the four operations would coincide at the level of $\Ext$ groups. As the two expressions we want to compare only differ by the order of the factors in the tensor products, a natural requirement to do in order to have them connected by a morphism in $\ext_{\mathscr{C}}^{p+q}(X\otimes X,Z\otimes Z)$ is to assume that $Z$ and $X$ are in the (left resp.\ right) centre of ${\mathscr{C}}$, that is,
come with natural maps
\begin{align*}
  \sigma_{\zett,-}\colon Z\otimes (-) &\to (-)\otimes Z,
  \\
\tau_{-,\ikks}\colon (-)\otimes X&\to X\otimes (-),
\end{align*}
the (left resp.\ right) {\em braidings}.
With this assumption, we obtain the diagram $(\sigma|\tau)_{\effe,\ehhe}$:
\begin{equation}\label{eq:almost-there}
\xymatrix@C=1.1em{
  Z\otimes Z\ar[ddd]_(.4){\sigma_{Z,Z}}\ar[r]& Z\otimes F_{q-1} \ar[r]\ar[ddd]_(.4){\sigma_{Z,F_{q-1}}} & \cdots \ar[r]& Z\otimes F_0\ar[rr]\ar[ddd]_(.4){\sigma_{Z,F_0}}\ar@/_1pc/[rd] && E_{p-1}\otimes X\ar[r]\ar[ddd]^(.4){\tau_{E_{p-1},X}}&\cdots\ar[r]&E_0\otimes X\ar[r]\ar[ddd]^(.4){\tau_{E_0, X}}&X\otimes X\ar[ddd]^(.4){\tau_{X,X}}
  \\
  &&&&
  Z\otimes X
  \ar@/_1pc/[ru] \ar@/_2pc/[d]_(.65){\sigma_{Z,X}}
  \ar@/^2pc/[d]^(.65){\tau_{Z,X}}
   \\
  &&&& X\otimes Z\ar@/^1pc/[rd]
  \\
Z\otimes Z\ar[r]& F_{q-1}\otimes Z\ar[r]&\cdots \ar[r]& F_0\otimes Z\ar[rr]\ar@/^1pc/[ru] && X\otimes E_{p-1}\ar[r] &\cdots\ar[r]&X\otimes E_0\ar[r]&X\otimes X.
}
\end{equation}
where, thanks to the naturality of the braidings
$\sigma_{\zett,-}$ and
$\tau_{-,\ikks}$,
everything commutes except possibly  the small diagram in the middle with the two 
arrows from $Z\otimes X$ to $X\otimes Z$. In order to have this part commute as well, we require
\begin{equation*}
\sigma_{\zett,\ikks}=\tau_{\zett,\ikks},
\end{equation*}
that is, $(X,Z)$ is a \emph{commuting pair} in the sense of Definition \ref{commpair}. Nevertheless, even when this is fulfilled and so $(\sigma|\tau)_{\effe,\ehhe}$ is a commutative diagram, we are not generally done yet: the diagram $(\sigma|\tau)_{\effe,\ehhe}$ does not define a morphism in $\ext_{\mathscr{C}}^{p+q}(X\otimes X,Z\otimes Z)$, see \eqref{morfi}, as the left- and rightmost vertical ``boundary'' morphisms $\sigma_{\zett,\zett}$ and $\tau_{\ikks,\ikks}$ are generally not the identity morphisms. We are going to show how to circumvent this problem by introducing a variant of the above construction in the following section.

\subsection{The cup product}
Assume now not only that $(X,Z)$ is a commuting pair in ${\mathscr{C}}$ in the sense of Definition \ref{commpair}, but also that $Z$ be a (braided) commutative monoid in the braided monoidal category ${\mathscr{Z}}^\ell({\mathscr{C}})$ and $X$ a (braided) commutative monoid in the braided monoidal category ${\mathscr{Z}}^r({\mathscr{C}})^\op$.
This allows for a cup product
\[
\cup \colon \ext_{\mathscr{C}}^p(X,Z)\times \ext_{\mathscr{C}}^q(X,Z)\to \ext_{\mathscr{C}}^{p+q}(X,Z)
\]
as follows. Let
\[
 \gD_\ikks \colon X \to X \otimes X \qquad \mbox{and} \qquad  \mu \colon Z\otimes Z\to Z
\]
be the morphisms defining the (co)monoid structures on $X$ and $Z$, respectively. Then, for $\E $ in  $\ext_{\mathscr{C}}^p(X,Z)$ and $\F$ in  $\ext_{\mathscr{C}}^q(X,Z)$, the cup product $\E \cup \F$ is defined as
\[
\E \cup \F\coloneq \mu \#(\E \otimes Z)\# (X\otimes \F)\# \gD_\ikks .
\]
Passing to path connected components, this yields a map
\[
\cup\colon \Ext_{\mathscr{C}}^p(X,Z)\times \Ext_{\mathscr{C}}^q(X,Z)\to \Ext_{\mathscr{C}}^{p+q}(X,Z).
\]
We already know from the results in \S\ref{sec:tensor} that $\E \cup \F$ is connected to
\[
 \mu \#(Z\otimes \F)\# (\E \otimes X) \# \gD_\ikks 
\]
via the span through $ \mu \#\mathbf{Moloch}(\E ,\F)\# \gD_\ikks $.
But now there is more: since $Z$ and $X$ are (braided) commutative monoids in ${\mathscr{C}}$ and in ${\mathscr{C}}^\op$, respectively, the commutative diagrams
\[
\raisebox{32pt}{\xymatrix{
X\otimes X\ar[dd]_{\tau_{X,X}}&\\
& X\ar[dl]^{ \gD_\ikks }\ar[ul]_{ \gD_\ikks }\\
X\otimes X&
}}
\qquad 
\qquad\qquad
\raisebox{32pt}{\xymatrix{
&Z\otimes Z\ar[dd]^{\sigma_{Z,Z}}\\
    Z\ar@{<-}[dr]_{ \mu }
    \ar@{<-}[ur]^{ \mu }\\
&Z\otimes Z
}}
\]
precisely cure the issue with the boundary morphisms in \eqref{eq:almost-there}, changing the boundary morphisms $\sigma_{Z,Z}$ and $\tau_{X,X}$ into identity morphisms. Taking the relevant pullbacks and pushouts the diagram \eqref{eq:almost-there} induces a morphism
\[
\widetilde{(\sigma|\tau)}_{\effe,\ehhe}\colon
 \mu \#
(Z\otimes \F) \#(\E \otimes X)
\# \gD_\ikks 
\xrightarrow{\ \, \mu \#(\sigma|\tau)_{\effe,\ehhe}\# \gD_\ikks \ \, }
 \mu \#
(\F\otimes Z)\#(X\otimes \E )
\# \gD_\ikks 
\]
in $\ext_{\mathscr{C}}^{p+q}(X,Z)$.

\subsection{The mystic hexagon}
With the assumptions of the previous section,
we have a loop in $\ext_{\mathscr{C}}^{p+q}(X,Z)$ given by
\begin{equation}
  \label{eq:loop}
\xymatrix@C=0.5em{
  &&
    \mu \# \mathbf{Moloch}(\E,\F )\# \gD_\ikks 
     \ar[lld]_{\tilde{\lambda}_{\ehhe,\effe}}
 \ar[rrd]^{\, \tilde{\gvr}_{\ehhe,\effe}}
  &&
  \\
  \hbox to 5em{\hss$
    \E \cup \F=   \mu \# (\E \otimes Z)\# (X\otimes \F) \# \gD_\ikks 
\hss$}
    \ar@{<-}[dd]_{\widetilde{(\sigma|\tau)}_{\ehhe,\effe}} &&
    &&
      \hbox to 5em{\hss$
 \mu \#(Z\otimes\F) \#(\E \otimes X) \# \gD_\ikks 
$\hss}
      \\ \\
            \hbox to 5em{\hss$
               \mu \# (Z\otimes\E) \#(\F \otimes X)  \# \gD_\ikks 
$\hss}
              &&
   &&
         \hbox to 5em{\hss$
           \F\cup \E = \mu \#(\F\otimes Z) \#(X\otimes\E) \# \gD_\ikks 
$\hss}
           \ar@{<-}[uu]_{\widetilde{(\sigma|\tau)}_{\effe,\ehhe}}
  \\ 
  &&
   \mu \# \mathbf{Moloch}(\F,\E )  \# \gD_\ikks 
\ar[llu]^{\tilde{\gvr}_{\effe,\ehhe}}
  \ar[rru]_{\, \tilde{\lambda}_{\effe,\ehhe}}
}
\end{equation}   
where we wrote
$$\tilde{\lambda}_{\ehhe,\effe}
\coloneq  \mu \#\lambda_{\ehhe,\effe}\# \gD_\ikks 
\qquad
\mbox{and}
\qquad
\tilde{\gvr}_{\ehhe,\effe} \coloneq  \mu \#\gvr_{\ehhe,\effe}\# \gD_\ikks ,
$$
respectively.
%
Choosing $\E \cup \F$ as a base point, 
this defines an element in the based loop space $\Omega_{\E \cup \F}\ext_{\mathscr{C}}^{p+q}(X,Z)$. 
Passing to connected components, we obtain a map
\[
\{\,\hspace*{1.3pt},\,\}\colon \Ext_{\mathscr{C}}^p(X,Z)\times \Ext_{\mathscr{C}}^q(X,Z)\to \Ext_{\mathscr{C}}^{p+q-1}(X,Z).
\]

\subsection{Cocycle representatives}
It is often convenient to represent an element $\E $ in $\ext_{\mathscr{C}}^p(X,Z)$ by a cocycle representative  $\phi$. By this one means the following: fix a projective resolution $\mathbb{P}^X \xrightarrow{\gve} X$ of the object $X$ and for $p=0$ consider the composition 
\begin{equation}
  \label{zerozero}
\phi_0\colon P^X_0\to X \xrightarrow{\E } Z,
\end{equation}
while for $p>0$ one considers a chain map $\phi\colon \mathbb{P}^X \to \E $ which is the identity over $X$, that is, a commutative diagram of the form
\begin{equation}
  \label{fame0}
\xymatrix{
   P^X_{p+1}\ar[d] \ar[r]^-d &P^X_p\ar[r]^-{d} \ar[d]_{\phi_p} & P^X_{p-1} \ar[r]^{d} \ar[d]^{\phi_{p-1}} & \ldots \ar[r]^-{d}
  &  P^X_0 \ar[r]^\gve \ar[d]_{\phi_0} & X \ar@{=}[d]
 \ar[r] & 0
  \\
 0 \ar[r]^d & Z \ar[r]^{i_\ehhe}  & E_{p-1}  \ar[r]^d & \ldots \ar[r]^{d}
 & E_0 \ar[r]^{p_\ehhe} & X  \ar[r] & 0
}
\end{equation}
The chain map $\phi$ exists and is unique up to homotopy by the projectivity of $\mathbb{P}^X$. Moreover, since $\mathbb{P}^X$ is a resolution, the datum of the homotopy class of $\phi$ is equivalent to the datum of the cohomology class of $\phi_p$ in the complex $\Hom_{\mathscr{C}}(\mathbb{P}^X,Z)$. For this reason one usually identifies $p$-cocycles with closed elements in $\Hom_{\mathscr{C}}(P^X_{p},Z)$.
\begin{rem}
To get a more uniform presentation, $0$-cocycle representatives will be realised by commutative diagrams as follows:
\[
\xymatrix{
P^X_{0}\ar[d]_{\phi_0}\ar[r]^{\epsilon}& X\ar[d]^{\E}\\
Z\ar@{=}[r] & Z
}
\]
\end{rem}

\medskip

It will be convenient to consider more generally commutative diagrams of the form 
\begin{equation}
  \label{fame00}
\xymatrix{
   P^X_{p+1}\ar[d] \ar[r]^-d &P^X_p\ar[r]^-{d} \ar[d]_{\phi_p} & P^X_{p-1} \ar[r]^{d} \ar[d]^{\phi_{p-1}} & \ldots \ar[r]^-{d}
  &  P^X_0 \ar[r]^\gve \ar[d]_{\phi_0} & X \ar[d]^{f}
 \ar[r] & 0
  \\
 0 \ar[r]^d & Z \ar[r]^{i_\ehhe}  & E_{p-1}  \ar[r]^d & \ldots \ar[r]^{d}
 & E_0 \ar[r]^{p_\ehhe} & X  \ar[r] & 0
}
\end{equation}
for an arbitrary morphism $f\colon X\to X$. We will call these diagrams \emph{$f$-twisted cocycle representatives} or {\em $f$-twisted chain maps}, or still {\em chain maps over $f$.}
In degree zero, an $f$-twisted $0$-cocycle representative for $\E$ will be defined as a commutative diagram of the form
\begin{equation}\label{eq:tonight}
\xymatrix{
P^X_{0}\ar[d]_{\phi_0}\ar[r]^{\epsilon}& X\ar[d]^{\E\circ f}\\
Z\ar@{=}[r] & Z
}
\end{equation}

\begin{rem}
  \label{rem:astray}
  It is immediate from the definition of twisted cocycle representatives that if $\phi\colon \mathbb{P}^X\to \E$ is an $f$-twisted cocycle representative and $\gamma\colon \mathbb{P}^X\to \mathbb{P}^X$ is a morphism of chain complexes, then $\phi\circ\gamma\colon \mathbb{P}^X \to \E$ is a $( f\circ \gamma_{-1})$-twisted
  cocycle representative,  where $\gamma_{-1}\colon X\to X$ is the degree $-1$ component of $\gamma$; if $\phi\colon \mathbb{P}^X  \to \E$ is an $f$-twisted cocycle representative  and $\gamma\colon \E\to \F$ is a morphism of chain complexes, then $\gamma\circ \phi\colon  \mathbb{P}^X \to \F$ is a $(\gamma_{-1}\circ f)$-twisted
  cocycle representative,  where $\gamma_{-1}\colon X\to X$ is the degree $-1$ component of $\gamma$. Another immediate consequence of the definition (and of the additivity of our category ${\mathscr{C}}$) is that if $\phi,\psi\colon \mathbb{P}^X\to \E$ are an $f$-twisted and a $g$-twisted cocycle representative,  respectively, then $\phi+\psi$ is an $(f+g)$-twisted cocycle representative.
\end{rem}

\subsection{The cup product in terms of cocycle representatives}
\label{sunshine}
If $\phi$ and $\psi$ are cocycle representatives for $\E \in \ext_{\mathscr{C}}^p(Y,Z)$ and $\F\in \ext_{\mathscr{C}}^q(X,Y)$, one can easily write a cocycle representative $\phi\#\psi$ for $\E \#\F$.
As for the $\#$-operation on the $\ext$ spaces, we will need to distinguish four cases. When $p$ and $q$ are zero, $\E $ and $\F$ are morphisms from $Y$ to $Z$ and from $X$ to $Y$, respectively, and their representative $0$-cocycles are the compositions $P^Y_{0}\to Y\xrightarrow{\E } Z$ and ${P^X_{0}\to X\xrightarrow{\F} Y}$, respectively. In this case, $\phi\#\psi$ is simply the composition 
\[
P^X_{0}\to X\xrightarrow{\F } Y\xrightarrow{\E} Z.
\]
 When $p>0$ and $q=0$, one considers the commutative diagram
 \vskip .2 cm
\begin{footnotesize}
\[
\xymatrix{
P^X_{p+1}\ar[r]\ar[d]_{\zeta_{p+1}}&P^X_{p}\ar[r]\ar[d]_{\zeta_p}&P^X_{p-1}\ar[r]\ar[d]_{\zeta_{p-1}}&\cdots\ar[r]& P^X_{1}\ar[r]\ar[d]_{\zeta_1}& P^X_{0}\ar@/^1.4pc/[rr]\ar@{=}[r]\ar[d]_{\zeta_0}&P^X_{0}\ar[r] \ar[d]^{\psi_0}\ar[r]&X\ar[d]^{\F }\ar[r]&0\\
P^Y_{p+1}\ar[r]\ar[d]&P^Y_{p}\ar[r]\ar[d]_{\phi_p} &P^Y_{p-1}\ar[r]\ar[d]_{\phi_{p-1}}&\cdots\ar[r]& P^Y_{1}\ar[r]\ar[d]_{\phi_1}& P^Y_{0}\ar[d]_{\phi_0}\ar[r]&Y\ar@{=}[r]\ar@{=}[d]&Y\ar[r]\ar@{=}[d]&0\\
0\ar[r]&Z\ar[r] &E_{p-1}\ar[r]&\cdots\ar[r]& E_1\ar[r] \ar[r]& E_0\ar[r]\ar@/_1.4pc/[rr]&Y\ar@{=}[r]&Y\ar[r]&0
}
\]
\end{footnotesize}
\vskip .3 cm
\noindent where the top vertical arrows are defined by the projectivity of $\mathbb{P}^X$. By the universal property of the pullback, we obtain a commutative diagram
\[
\xymatrix{
P^X_{p+1}\ar[r]\ar[d]&P^X_{p}\ar[r]\ar[d]_{(\phi\#\psi)_p} &P^X_{p-1}\ar[r]\ar[d]_{(\phi\#\psi)_{p-1}}&\cdots\ar[r]& P^X_{1}\ar[r]\ar[d]_{(\phi\#\psi)_1} & P^X_{0} \ar[r]^{\gve} \ar[d]_{(\phi\#\psi)_0} & X\ar@{=}[d]\ar[r]&0
\\
0\ar[r]\ar[d]&Z\ar[r]\ar@{=}[d] &E_{p-1}\ar[r]\ar@{=}[d]&\cdots\ar[r]& E_{1}\ar[r]^-{(d,0)}\ar@{=}[d]& E_{0}{}^{p_\E}\!\!\!\times^{\effe}_YX\ar[r]\ar[d]&X\ar[d]^{\F }\ar[r]&0
\\
0\ar[r]&Z\ar[r] &E_{p-1}\ar[r]&\cdots\ar[r]& E_1\ar[r]& E_0\ar[r]&Y\ar[r]&0
}
\]
whose top part exhibits a cocycle representative $\phi\#\psi$ for $\E \#\F$. Notice that for every $i>0$ one has $(\phi\#\psi)_i=\phi_i\circ\zeta_i$, while for $i=0$ that $(\phi\#\psi)_0$ is the lift of  $\phi_0\circ\zeta_0$ to the fibre product $E_{0}{}^{p_\E}\!\!\!\times^{\effe}_YX$, that is, $(\phi_0\circ\zeta_0, \gve)$. Similarly, and more directly, when $p=0$ and $q>0$, one considers the commutative diagram
\vskip .2 cm
\[
\xymatrix@C=1.9em{
P^X_{q+1}\ar[r]\ar[d]_{\zeta_1}&P^X_{q}\ar@/^1.4pc/[rr] \ar@{=}[r]\ar[d]_{\zeta_0} &P^X_{q}\ar[r]\ar[d]_{\psi_q}&P^X_{q-1}\ar[d]_{\psi_{q-1}}\ar[r]&\cdots\ar[r]& P^X_{1}\ar[r]\ar[d]_{\psi_1}& P^X_{0}\ar[r]\ar[d]_{\psi_0}&X\ar@{=}[d]\ar[r]&0
\\
P^Y_{1}\ar[d]\ar[r]&P^Y_{0}\ar[r]\ar[d]_{\phi_0}&Y\ar[d]_{\E}\ar[r] &F_{q-1}\ar[r]\ar[d]&\cdots\ar[r]& F_{1}\ar[r]\ar@{=}[d]& F_{0}\ar[r]\ar@{=}[d]&X\ar@{=}[d]\ar[r]&0
\\
0\ar[r]&Z\ar@/_1.4pc/[rr]\ar@{=}[r]&Z\ar[r] &Z\sqcup_Y F_{q-1}\ar[r]&\cdots\ar[r]& F_1\ar[r]& F_0\ar[r]&X\ar[r]&0
}
\]
\vskip .3 cm
\noindent whose composite vertical arrows are a cocycle representative $\phi\#\psi$ for $\E \#\F$.
In degrees from $0$ to $q-2$, one has $(\phi\#\psi)_i=\psi_i$, whereas in degree $q-1$ one sees that $(\phi\#\psi)_{q-1}$ is the class of $(0,\psi_{q-1})$ in the quotient of $Z\oplus F_{q-1}$ by the image of $Y$, and
in the top degree $q$, one has $(\phi\#\psi)_q=\phi_0\circ\zeta_0=\E\circ \psi_q$.

Finally, in case $p,q>0$,
one uses the projectivity of $\mathbb{P}_{X}$ to define the maps $\zeta_i\colon P^X_{i+q} \to P^Y_{i}$, making the diagram
\begin{equation*}
\xymatrix@C=1.7em{
  P^X_{p+q}\ar[d]^{\zeta_{p}}\ar[r]& P^X_{p+q-1}\ar[r]\ar[d]^{\zeta_{p-1}} & \cdots \ar[r]& P^X_{q}\ar@/^1.4pc/[rr] \ar[d]^{\zeta_0}\ar@{=}[r] &P^X_{q}\ar[d]^{\psi_q}\ar[r]& P^X_{q-1}\ar[r]\ar[d]^{\psi_{q-1}} &\cdots\ar[r]& P^X_{0}\ar[d]^{\psi_{0}}\ar[r]& X\ar@{=}[d]
  \\
  P^Y_{p}\ar[d]_{\phi_{p}}\ar[r]& P^Y_{p-1}\ar[r]\ar[d]_{\phi_{p-1}} & \cdots \ar[r]& P^Y_{0}\ar[d]_{\phi_{0}}\ar[r] &Y\ar@{=}[d]\ar[r]& F_{q-1}\ar[r]\ar@{=}[d] &\cdots\ar[r]& F_0\ar@{=}[d]\ar[r]& X\ar@{=}[d]
  \\
  Z\ar[r]& E_{p-1}\ar[r] & \cdots \ar[r]& E_0\ar[r]\ar@/_1.4pc/[rr] &Y\ar[r]& 
   F_{q-1} \ar[r] &\cdots\ar[r]&  F_{0}\ar[r]& X
}
\end{equation*}

\vskip 8pt

\noindent commute. Since the cup product is a composition of $\#$-products, the above rules yield, in particular, an expression for a cocycle representative $\phi\cup\psi$ for the product $\E \cup\F$, given cocycle representatives $\phi$ and $\psi$ for $\E $ and $\F$, respectively.

\begin{rem}
  \label{rem:verbatim}
  By verbatim repeating the above construction of the case $p>0$ and $q=0$, one sees that if $\phi$ is an $f$-twisted cocycle representative  for $\F$ and $\psi_0$ a cocycle representative for a morphism $\F\colon X\to Y$, then 
\[
\xymatrix@C=2.3em{
P^X_{p+1}\ar[r]\ar[d]&P^X_{p}\ar[r]\ar[d]_{(\phi\#\psi)_p} &P^X_{p-1}\ar[r]\ar[d]_{(\phi\#\psi)_{p-1}}&\cdots\ar[r]& P^X_{1}\ar[r]\ar[d]_{(\phi\#\psi)_1} & P^X_{0} \ar[r]^{\gve} \ar[d]_{(\phi\#\psi)_0} & X\ar@{=}[d]\ar[r]&0
\\
0\ar[r]&Z\ar[r]&E_{p-1}\ar[r]&\cdots\ar[r]& E_{1}\ar[r]^-{(d,0)}& E_{0}{}^{p_\E}\!\!\!\times^{f\circ \F}_YX\ar[r]&X\ar[r]&0
}
\]
is a cocycle representative for $\E\#(f\circ\F)$,
where $(\phi\#\psi)_0=(\phi_0\circ\zeta_0,\gve)$ and  $(\phi\#\psi)_i=\phi_i\circ\zeta_i$ for every $i>0$. 
\end{rem}

\subsection{The bracket $\{\,\hspace*{1.3pt},\,\}$ in terms of cocycle representatives}
\label{torloniasantisubito}
Let $\phi$ and $\psi$ be two cocycle representatives for $[\E ]\in \Ext_{\mathscr{C}}^p(X,Z)$ and $[\F]\in \Ext_{\mathscr{C}}^q(X,Z)$, respectively, and as before assume that $(X,Z)$ is a commuting pair in the sense of Definition \ref{commpair}, where $X$ is a (braided) commutative monoid in ${\mathscr{Z}}^r(\umod)^\op$ and $Z$ a (braided) commutative monoid in ${\mathscr{Z}}^\ell(\umod)$. 

In order to associate with $\phi$ and $\psi$ a cocycle representative  $\{\phi,\psi\}$ representing the element
$$
\{[\E ],[\F]\}\in \Ext_{\mathscr{C}}^{p+q-1}(X,Z)=\pi_1 \ext_{\mathscr{C}}^{p+q}(X,Z),
$$
defined by the loop \eqref{eq:loop} based at $\E \cup\F$, one applies the following procedure. For any $n\geq 0$, one considers the (higher) categories $\widetilde{\ext}_{\mathscr{C}}^{\raisebox{-3.5pt}{\scriptsize $n$}}(X,Z)$, whose objects are pairs $(\mathbb{G},\varphi)$, where $\mathbb{G}$ is an object in  $ \ext_{\mathscr{C}}^n(X,Z)$ and $\varphi\colon \mathbb{P}^X\to \mathbb{G}$ is a cocycle representing $\mathbb{G}$. Morphisms between $(\mathbb{G}_1,\varphi_1)$ and $(\mathbb{G}_2,\varphi_2)$ are pairs $(f,h)$ consisting of a morphism $f\colon \mathbb{G}_1\to \mathbb{G}_2$ in $ \ext_{\mathscr{C}}^n(X,Z)$ and of a homotopy $h$ between $f\circ \varphi_1$ and $\varphi_2$, that is, a degree $-1$ map $h\colon \mathbb{P}^X\to \mathbb{G}_2$ such that $[d,h]=\varphi_2-f\circ \varphi_1$. Higher morphisms are defined recursively in a similar way. One has a natural forgetful functor  
\[
\widetilde{\ext}_{\mathscr{C}}^{\raisebox{-3.5pt}{\scriptsize $n$}}(X,Z)\to \ext_{\mathscr{C}}^n(X,Z)
\]
whose fibre at the point $\mathbb{G}$ is the hom space $\Hom_{\ext_{\mathscr{C}}^n(X,Z)}(\mathbb{P}^X,\mathbb{G})$. Since $\mathbb{P}^X$ is a resolution of $X$, it is in particular an acyclic complex and so this hom space is contractible. Therefore, the projection $\widetilde{\ext}_{\mathscr{C}}^{\raisebox{-3.5pt}{\scriptsize $n$}}(X,Z)\to \ext_{\mathscr{C}}^n(X,Z)$ is a homotopy equivalence and one can read the homotopy groups of $\ext_{\mathscr{C}}^n(X,Z)$ in terms of the homotopy groups of $\widetilde{\ext}_{\mathscr{C}}^{\raisebox{-3.5pt}{\scriptsize $n$}}(X,Z)$. Moreover, $\widetilde{\ext}_{\mathscr{C}}^{\raisebox{-3.5pt}{\scriptsize $n$}}(X,Z)\to \ext_{\mathscr{C}}^n(X,Z)$ is a flat fibration: every path in $ \ext_{\mathscr{C}}^n(X,Z)$ has a canonical horizontal lift in $\widetilde{\ext}_{\mathscr{C}}^{\raisebox{-3.5pt}{\scriptsize $n$}}(X,Z)$, once a lift for the starting point is chosen:
if $(\mathbb{G}_1,\varphi_1)$ is a point in the fibre over $\mathbb{G}_1$ and $f\colon \mathbb{G}_1\to \mathbb{G}_2$ a morphism in $ \ext_{\mathscr{C}}^n(X,Z)$, then we have the canonical lift
\[
(f,0)\colon (\mathbb{G}_1,\varphi_1)\to (\mathbb{G}_2,f\circ \varphi_1).
\]
This gives a way to encode a loop $\gamma$ in $ \ext_{\mathscr{C}}^n(X,Z)$ based at a point $\mathbb{G}$ into a homotopy operator, {\em i.e.}, a degree $-1$ map $h_\gamma\colon \mathbb{P}^X\to \mathbb{G}$, once a representative cocycle $\mathbb{P}^X\to \mathbb{G}$ for $\mathbb{G}$ is chosen: one horizontally lifts the loop to a path in $\widetilde{\ext}_{\mathscr{C}}^{\raisebox{-3.5pt}{\scriptsize $n$}}(X,Z)$ with starting point $(\mathbb{G},\varphi)$. The final point of this path will be a point $(\mathbb{G},\tilde{\varphi})$ and, since they both represent the same element $[\mathbb{G}]$ in $\Ext_{\mathscr{C}}^n(X,Z)$, the two cocycle representatives  $\varphi$ and $\tilde{\varphi}$ will be homotopic via a degree $-1$ map $h_\gamma\colon \mathbb{P}^X\to \mathbb{G}$, which is unique up to higher homotopies. The homotopy class $[h_\gamma]$ is the element in $\Ext_{\mathscr{C}}^{n-1}(X,Z)$ representing the loop $\gamma$.

We now apply this abstract nonsense to the loop \eqref{eq:loop}. Starting with two cocycle representatives $\phi$ and $\psi$ for $\E $ and $\F$, respectively, one can form a representative cocycle $\phi\cup \psi$ for $\E \cup\F$ and exhibit a representative cocycle
\[
\phi \, \tilde{\cup}_{\scriptscriptstyle \otimes} \, \psi\colon \mathbb{P}^X\to \mu \# \mathbf{Moloch}(\E,\F )\# \gD_\ikks 
\]
for $ \mu \# \mathbf{Moloch}(\E,\F )\# \gD_\ikks $ such that $\phi\cup\psi=\tilde{\lambda}_{\ehhe,\effe}\circ (\phi\,\tilde{\cup}_{\scriptscriptstyle \otimes}\, \psi)$.
The first step
 in the horizontal lift of the loop consists in lifting $\tilde{\lambda}_{\ehhe,\effe}$. This is immediate: the horizontal lift is
\[
\big( \mu \# \mathbf{Moloch}(\E ,\F)\# \gD_\ikks ,\phi\,\tilde{\cup}_{\scriptscriptstyle \otimes}\, \psi\big)\xrightarrow{(\tilde{\lambda}_{\ehhe,\effe},0)} \big(\E \cup\F, \phi\cup\psi\big).
\]
Also the subsequent horizontal lift of $\widetilde{(\sigma|\tau)}_{\effe,\ehhe}\circ\tilde{\gvr}_{\ehhe,\effe}$ is immediate: it is
\[
\big( \mu \# \mathbf{Moloch}(\E ,\F)\# \gD_\ikks ,\phi\,\tilde{\cup}_{\scriptscriptstyle \otimes}\, \psi\big)
\xrightarrow{\big(\widetilde{(\sigma|\tau)}_{\effe,\ehhe}\circ\tilde{\gvr}_{\ehhe,\effe},0\big)}
\big(\F\cup\E , \widetilde{(\sigma|\tau)}_{\effe,\ehhe}\circ\tilde{\gvr}_{\ehhe,\effe}\circ (\phi\,\tilde{\cup}_{\scriptscriptstyle \otimes}\, \psi)\big).
\]
The next step is less trivial: one has to horizontally lift $\tilde{\lambda}_{\effe,\ehhe}$ in such a way that its endpoint matches the endpoint of $\big(\widetilde{(\sigma|\tau)}_{\effe,\ehhe}\circ\tilde{\gvr}_{\ehhe,\effe},0\big)$.
That is, one has to find a cocycle representative $\tilde{\kappa}_{\phi,\psi}$ for $ \mu \# \mathbf{Moloch}(\F,\E )\# \gD_\ikks $ such that
$$
\tilde{\lambda}_{\effe,\ehhe}\circ \tilde{\kappa}_{\phi,\psi}=\widetilde{(\sigma|\tau)}_{\effe,\ehhe}\circ\tilde{\gvr}_{\ehhe,\effe}\circ (\phi\,\tilde{\cup}_{\scriptscriptstyle \otimes}\, \psi).
$$
Writing $\tilde{\kappa}_{\phi,\psi}=\psi\,\tilde{\cup}_{\scriptscriptstyle \otimes}\, \phi+\tilde{\epsilon}_{\phi,\psi}$, this becomes the equation
\[
\tilde{\lambda}_{\effe,\ehhe}\circ(\psi\,\tilde{\cup}_{\scriptscriptstyle \otimes}\, \phi+\tilde{\epsilon}_{\phi,\psi})=\widetilde{(\sigma|\tau)}_{\effe,\ehhe}\circ \tilde{\gvr}_{\ehhe,\effe}\circ (\phi\,\tilde{\cup}_{\scriptscriptstyle \otimes}\, \psi).
\] 
Assuming one has been able (or lucky) enough to find an $\tilde{\epsilon}$ solving this equation, then one has the immediate horizontal lift
\begin{small}
\[
\big( \mu \# \mathbf{Moloch}(\F,\E )\# \gD_\ikks , \psi\,\tilde{\cup}_{\scriptscriptstyle \otimes}\, \phi+\tilde{\epsilon}_{\phi,\psi}\big)
\xrightarrow{\big(\widetilde{(\sigma|\tau)}_{\ehhe,\effe}\circ\tilde{\gvr}_{\effe,\ehhe},0\big)}
\big(\E \cup\F,\widetilde{(\sigma|\tau)}_{\ehhe,\effe}\circ \tilde{\gvr}_{\effe,\ehhe}(\psi\,\tilde{\cup}_{\scriptscriptstyle \otimes}\, \phi+\tilde{\epsilon}_{\phi,\psi})\big). 
\]
\end{small}
This concludes the procedure of horizontally lifting the loop based at $\E \cup\F$, so the corresponding element $\{\phi,\psi\}$ in $\Ext_{\mathscr{C}}^{p+q-1}(X,Z)$ is represented by a homotopy $\tilde{s}(\phi,\psi)$ between $\widetilde{(\sigma|\tau)}_{\ehhe,\effe}\circ\tilde{\gvr}_{\effe,\ehhe}(\psi\,\tilde{\cup}_{\scriptscriptstyle \otimes}\, \phi+\tilde{\epsilon}_{\phi,\psi})$ and $ \tilde{\lambda}_{\ehhe,\effe}\circ (\phi\,\tilde{\cup}_{\scriptscriptstyle \otimes}\, \psi)$, taken up to higher homotopies.
  Writing
  \[
  \tilde{\eta}(\phi,\psi)=
  \tilde{\lambda}_{\ehhe,\effe} (\phi\,\tilde{\cup}_{\scriptscriptstyle \otimes}\, \psi)
-
\widetilde{(\sigma|\tau)}_{\ehhe,\effe}\circ\tilde{\gvr}_{\effe,\ehhe}(\psi\,\tilde{\cup}_{\scriptscriptstyle \otimes}\, \phi
  +\tilde{\epsilon}_{\phi,\psi}),
\]
one therefore sees that
 the element $\{\phi,\psi\}$ is represented by the homotopy class of a solution $\tilde{s}(\phi,\psi)$ of the equation $\tilde{\eta}(\phi,\psi)=[d,\tilde{s}(\phi,\psi)]$, that is, by the cohomology class of the closed element $\tilde{s}_{p+q-1}(\phi,\psi)$ in $\Hom_{\mathscr{C}}(P^X_{p+q-1},Z)$. 
 \par
To explicitly compute the element $\tilde{s}_{p+q-1}(\phi,\psi)$, one can argue as follows. If $\alpha$ is a cocycle representative for an object $\mathbb{H}$ in $\ext^{p+q}_{\mathscr{C}}(X\otimes X, Z\otimes Z)$, then by the constructions from \S\ref{sunshine} one writes an explicit formula for a cocycle representative $\tilde{\alpha}$ for the object $\mathbb{G}=\mu\#\mathbb{H}\#\Delta_\ikks$ in $\ext_{\mathscr{C}}^{p+q}(X, Z)$, which can be straightforwardly extended to arbitrary collections of maps
 $$
\big\{\beta_k\colon P^{X\otimes X}_k \to H_{k}\big\}_{0\leq k\leq p+q} \qquad \mbox{and} \qquad
\big\{\nu_k\colon P^{X\otimes X}_k \to H_{k+1}\big\}_{0\leq k\leq p+q-1}.
 $$
Doing this, one obtains a graded linear map
$$
  \Phi_\bull\colon \textstyle
      {\bigoplus\limits_{{k=-1}}^r}\Hom_{\mathscr{C}}(P^{X\otimes X}_k ,{H}_{k})
      \to \textstyle
          {\bigoplus\limits_{k=-1}^r}\Hom_{\mathscr{C}}(P^{X}_k  ,{G}_{k}).
$$
Looking at the explicit expression for $\Phi$, it is rather easy to define a graded linear map
$$
    \Psi_\bull\colon \textstyle
        {\bigoplus\limits_{k=-1}^{r-1}}\Hom_{\mathscr{C}}(P^{X\otimes X}_k ,{H}_{k+1})\to \textstyle
        {\bigoplus\limits_{k=-1}^{r-1}}\Hom_{\mathscr{C}}(P^{X}_k  ,{G}_{k+1})
$$
such that
\[
 \Phi\circ [d,-]=[d,-]\circ \Psi.
\]
Writing
 \[
 \eta(\phi, \psi) =  \gl_{\ehhe, \effe}\big((\phi \cup_{\scriptscriptstyle \otimes} \psi) \circ \tau\big)
    - 
(\sigma|\tau)_{\ehhe,\effe} \circ 
    \gvr_{\effe, \ehhe}\big((
    \psi \cup_{\scriptscriptstyle \otimes} \phi +\epsilon_{\phi,\psi}\big),
  \]  
  where $\phi \cup_{\scriptscriptstyle \otimes} \psi$ is a $\tau$-twisted cocycle representative  for $\mathbf{Moloch}(\E,\F)$, the map $\psi \cup_{\scriptscriptstyle \otimes} \phi$ is a $\tau$-twisted cocycle representative  for $\mathbf{Moloch}(\F,\E)$, and $\epsilon(\phi,\psi) = \epsilon_{\phi,\psi}$ is a $0$-twisted cocycle representative  for $\mathbf{Moloch}(\F,\E)$,
  then $\Phi(\eta(\phi, \psi)) = \tilde{\eta}(\phi,\psi)$ and the defining equation $\tilde{\eta}(\phi,\psi)=[d,\tilde{s}(\phi,\psi)]$ for $\tilde{s}(\phi,\psi)$ reduces to
  \[
  \Phi(\eta(\phi, \psi))=[d,\tilde{s}(\phi,\psi)].
  \]
If  ${s}(\phi,\psi)\in {\bigoplus_{k=-1}^{p+q-1}}\Hom_{\mathscr{C}}(P^{X\otimes X}_k , ((\E \otimes Z) \#(X \otimes \F))_{k+1})$
is tailored in such a way that $\Phi([d,{s}(\phi,\psi)])=\Phi(\eta(\phi,\psi))$, then one can take $\tilde{s}(\phi,\psi)= \Psi(s(\phi,\psi))$,
%
and deduces that $\{\phi,\psi\}$ is explicitly represented by the cohomology class of the closed element $\Psi_{p+q-1}(s_{p+q-1}(\phi,\psi))$ in $\Hom_{\mathscr{C}}(P^X_{p+q-1},Z)$.

In the next section, we will explicitly apply what we just developed.


\section{An explicit approach for bialgebroids}\label{sec:explicit}

In this section, we shall describe how the aforementioned theory for extension categories can be made explicit in case of the category of left modules over a left bialgebroid, endowing the corresponding $\Ext$ groups with a Gerstenhaber algebra structure. To this end, a short account on bialgebroids $(U,A)$ and on the respective notions of left and right weak centre
in relation to (the categories of) {\em Yetter-Drinfel'd modules} can be found in Appendix \S\ref{biappendix}.

The main goal in the final part of this section (\S\ref{appassionata}) will be to see that $\Ext_U(X,Z)$ is a Gerstenhaber algebra if
$(X,Z)$ is a commuting pair in the sense of Definition \ref{commpair}.
There, the explicit proof works along standard lines, that is, one shows that the explicit cochain complex $\Hom_U(\mathrm{Bar}_n(U,X), Z)$ computing  $\Ext_U(X,Z)$  is an operad with multiplication, with the bedevilling difference that the operadic composition is somewhat complicated and not defining any endomorphism operad structure.

\subsection{$\Ext$ groups and extension categories over bialgebroid modules}
\label{cadutamassi}

Let $(U, A)$ be a left bialgebroid.
For simplicity, we will always and tacitly assume that $\due U \lact \ract$ is left-right projective over $A$,
that is, projective with respect to both the left $A$-action $\lact$ and the right $A$-action $\ract$,
see \eqref{pergolesi} for notation.

Let $\umod$ be the monoidal category of left $U$-modules.
    In generalising Schwede's approach from the Hochschild case of associative algebras to bialgebroids, that is, in passing from $\amoda$ to $\umod$, 
   we are facing the same kind of problems with respect to the exactness of the monoidal category in question. More precisely, the tensor product $\E \otimes_A \F$
    of two 
    extensions in $\umod$ is in general not exact but it is so if $\E$ and $\F$ are taken in the full subcategory $\underline{\ext}_U$ of the extension category $\ext_U$ that consists of those extensions in which all $U$-modules are left-right projective over $A$. One can then generalise Lemma 2.1 in \cite{Schw:AESIOTLBIHC} in a straightforward way and show that the inclusion of  $\underline{\ext}_U$ into $\ext_U$ induces a homotopy equivalence on classifying spaces. We omit the technical details here but feel therefore entitled to not explicitly distinguish between  $\underline{\ext}_U$ and $\ext_U$
in what follows.

As expounded in Appendix \S\ref{biappendix}, a standard result \cite[Prop.~4.4]{Schau:DADOQGHA}
      establishes an equivalence of braided monoidal categories between the left weak centre 
${\mathscr{Z}}^\ell(\umod)$ and $\yd$, the category
of left-left Yetter-Drinfel'd modules; likewise between ${\mathscr{Z}}^r(\umod)$ and $\ydr$, the category 
of left-right Yetter-Drinfel'd modules. In the former case, this is done by assigning to any $Z \in \yd$ the underlying left module $Z \in \umod$ along with what we might call a {\em left braiding}
\begin{equation}
     \label{sahnejoghurt0}
  \gs =  \gs_{\zett,\emme} \colon Z \otimes_A M \to M \otimes_A Z,
     \quad     z \otimes_A m \mapsto z_{(-1)}m \otimes_A z_{(0)}
\end{equation}
for any $M \in \umod$, 
to form an object $(Z, \sigma)$ in ${\mathscr{Z}}^\ell(\umod)$. Likewise, assign to any $X \in {\mathscr{Z}}^r(\umod)$ its underlying left module $X \in \umod$ along with the {\em right braiding}
 \begin{equation}
     \label{sahnejoghurt0a}
     \tau = \tau_{\emme, \zett} \colon M \otimes_A X \to X \otimes_A M,
     \quad
     m \otimes_A x \mapsto x_{[0]} \otimes_A x_{[1]}m
   \end{equation}
for any $M \in \umod$,
to give an object $(X, \tau)$ in ${\mathscr{Z}}^r(\umod)$, see
Appendix \ref{appendix} for more details and all notation used in what follows.

A cochain complex computing $\Ext^\bull_U(X,Z)$ is given by $\big(C^\bull(U, X, Z), \gd\big)$, where
\begin{equation}
  \label{caffe}
  C^n(U, X, Z) \coloneq \Hom_U(\mathrm{Bar}_n(U,X), Z),
\end{equation}
and
$$
\mathrm{Bar}_n(U,X) \coloneq (\due U \blact \ract)^{\otimes_\Aop {n+1}} \otimes_\Aop \due X {} \ract,
$$
with
$
\mathrm{Bar}_{-1}(U,X) \coloneq X
 $,
is the bar resolution of the left $U$-module $X$, seen as a right $A$-module $\due X {} \ract$ as in \eqref{forgetthis}, with differential $d = \sum^n_{i=0} (-1)^i d_i$ induced by the multiplication in $U$ in the first $n-1$ faces $d_i$ and the left $U$-action
\begin{equation}
  \label{leftaction}
  L\colon \mathrm{Bar}_0(U, X) = \due U {\blact} {} \otimes_\Aop \due X {} \ract \to X =  \mathrm{Bar}_{-1}(U, X),
\quad u \otimes_\Aop x \mapsto ux
\end{equation}
on $X$ that induces the last face map $d_n$. In a standard fashion, we then simply have $\gd \coloneq d^*$.
Observe that the tensor product used in the bar resolution is not the monoidal one \eqref{rain} of $\umod$ (in which factorwise multiplication would not be defined), but rather $\mathrm{Bar}_n(U,X)
\in \umod$ by left multiplication on the first tensor factor.


The Hochschild case for an associative $\kkk$-algebra $A$ and hence Schwede's construction in \cite{Schw:AESIOTLBIHC} is reobtained by setting $U = \Ae$ as well as $X=Z=A$ in what follows.


To connect the context of a left bialgebroid to our general observations in \S\ref{opihashi}, 
let us write down in detail what the concrete setting is here: for two extensions $\E$ in $\ext^p_U(X,Z)$ and $\F$ in $\ext^q_U(X,Z)$, respectively, consider chain maps $\phi\colon \mathrm{Bar}(U, X) \to \E$ and
$\psi\colon \mathrm{Bar}(U, X) \to \F$  over the identity of $X$.
With respect to the bar resolution, Diagram \eqref{fame0} then takes the form
\begin{footnotesize}
\begin{equation}
  \label{fame}
\xymatrix@C=2.7em{
  0 \ar[r]^-d & \mathrm{Bar}_p(U,X) \ar[r]^-{d} \ar[d]_{\phi_p} & \mathrm{Bar}_{p-1}(U,X) \ar[r]^-{d} \ar[d]^{\phi_{p-1}} & \ldots \ar[r]^-{d}
  &  \mathrm{Bar}_0(U,X) \ar[r]^-L \ar[d]_{\phi_0} & X \ar@{=}[d]
 \ar[r] & 0
  \\
 0 \ar[r]^d & Z \ar[r]^{i_\ehhe}  & E_{p-1}  \ar[r]^d & \ldots \ar[r]^{d}
 & E_0 \ar[r]^{p_\ehhe} & X  \ar[r] & 0
}
\end{equation}
\end{footnotesize}
where $L$ is the left $U$-action on $X$ as in \eqref{leftaction}; likewise for $\psi$ and $\F$.

Observe that with respect to the differential $\gd = d^*$ of the cochain complex $C^n(U, X, Z)$ from \eqref{caffe}, the commutativity of this diagram can be expressed as
\begin{equation}
  \label{blumare}
  i_\ehhe \circ \phi_p = \phi_{p-1} \circ d = \gd\phi_{p-1},
  \qquad\qquad
  d \circ \phi_j = \phi_{j-1} \circ d = \gd\phi_{j-1},
  \end{equation}
for $j = 1, \ldots, p-1$, and likewise for $\psi$. In particular, the first square in Diagram \eqref{fame} implies that
\begin{equation}
  \label{famepigreco}
\phi_p \circ d = \gd \phi_p = 0,
\end{equation}
that is, $\phi_p$ is a cocycle (and likewise for $\psi_q$), whereas the last square reads
\begin{equation}
  \label{fame1047}
p_\ehhe\circ \phi_0=L.
\end{equation}

\subsection{The external cup product}
\label{external}
As above, let $(Z, \gs)$ be an object in the left weak centre ${\mathscr{Z}}^\ell(\umod) \simeq \yd$ and $(X, \tau)$ an object in the right weak centre ${\mathscr{Z}}^r(\umod) \simeq \ydr$,
and consider again two extensions
$$
\E\colon 0 \to Z \xrightarrow{\scriptscriptstyle i_\ehhe} E_{p-1}  \xrightarrow{d} \ldots  \xrightarrow{d} E_0 \xrightarrow{\scriptscriptstyle p_\ehhe} X \to 0,
$$
$$
\F\colon 0 \to Z \xrightarrow{\scriptscriptstyle i_\effe} F_{q-1}  \xrightarrow{d} \ldots  \xrightarrow{d} F_0 \xrightarrow{\scriptscriptstyle p_\effe} X \to 0,
$$
of $U$-modules, along with
two
families of morphisms $\phi_j \in C^j(U, X, E_j)$ for $j = 0, \ldots, p$ and $\psi_i \in  C^i(U, X, F_i)$ for $i = 0, \ldots, q$ as in Diagram \ref{fame}.

For these, one can define a sort of $1$-product, or {\em external cup product}   
\begin{equation*}
  \begin{array}{rcl}
    \cup_{\scriptscriptstyle \otimes} \colon C^j(U, X, E_j) \otimes C^i(U, X, F_i) 
    \!\!\!    &\to \!\!\! & C^{j+i}(U, X \otimes_A X, E_j \otimes_A F_i)
\end{array}
\end{equation*}
for $i =1, \ldots, p$ and $j = 1, \ldots, q$,
by defining
  \begin{equation}
\label{cupcup}
  \begin{split}
&   (\phi_j \cup_{\scriptscriptstyle \otimes} \psi_i)(u^0, \ldots, u^{j+i}, m \otimes_A m')
    \\
    &
    \quad
    \coloneq     \phi_j\big(u^0_{(1)}, \ldots, u^j_{(1)}, (u^{j+1}_{(2)} \cdots u^{j+i}_{(2)} m')_{[0]}\big)
    \\
    &
    \qquad\qquad
    \otimes_A \psi_i\big(u^0_{(2)} \cdots u^j_{(2)}
(u^{j+1}_{(2)} \cdots u^{j+i}_{(2)} m')_{[1]}
    , u^{j+1}_{(1)}, \ldots, u^{j+i}_{(1)}, m\big),
\end{split}
  \end{equation}
where the Sweedler notation in square brackets refers to the right $U$-coaction on $X$, see \S\ref{ydydydyd} for notational details.

\begin{rem}\label{rem:annette}  Observe that
$$
C^{j+i}(U, X \otimes_A X, E_j \otimes_A F_i) \subset  C^{j+i}\big(U, X \otimes_A X, \mathbf{Moloch}(\E ,\F)\big)
$$
for $i, j \geq 0$, so we may regard $\cup_{\scriptscriptstyle \otimes}$ as a map into the latter. In other words, given two chain maps $\phi$ and $\psi$ as above, the operation $\cup_{\scriptscriptstyle \otimes}$ defines a collection of linear morphisms
\[
(\phi\cup_{\scriptscriptstyle \otimes}\psi)_k\colon \mathrm{Bar}_k(U,X \otimes_A X)\to \mathbf{Moloch}(\E ,\F)_k, \qquad k\geq 0,
\]
which we extend to degree $-1$ by setting $(\phi\cup_{\scriptscriptstyle \otimes}\psi)_{-1}\coloneq\tau\colon X\otimes_AX\to X\otimes_A X$.
\end{rem}


\begin{rem}\label{rem:theabyss}
By functoriality of the tensor product and the (left-right) YD condition \eqref{yd2}, it is not difficult to see that the map $\tau_\bull\colon  \mathrm{Bar}_\bull(U,X\otimes_A X) \to \mathrm{Bar}_\bull(U,X\otimes_A X)$ is a morphism of complexes, where $\tau_{-1}=\tau$. 
\end{rem}

\begin{lem}
\label{leibnizcup}
  Let $\phi_j \in C^j(U,X,E_j)$ and $\psi_i \in C^i(U, X, Z)$ and as before let $\gd = d^*$ be the pullback of the differential of the bar resolution.
  Then $\gd$ is a derivation of the external cup product \eqref{cupcup}, that is,
  the graded Leibniz rule
  \begin{equation}
    \label{cupleibniz}
    \gd(\phi_j \cup_{\scriptscriptstyle \otimes} \psi_i)
    = \gd \phi_j \cup_{\scriptscriptstyle \otimes} \psi_i
    + (-1)^{j} \phi_j \cup_{\scriptscriptstyle \otimes} \gd \psi_i
    \end{equation}
  holds for any $i, j \geq 0$.
\end{lem}

\begin{proof}
  This is a direct verification using \eqref{cupcup}, the explicit form of the differential in the bar resolution as described below \eqref{caffe}, the fact that the coproduct is a ring morphism, along with the monoidal structure \eqref{rain} in $\umod$ and $X$ as an object in the weak right centre of $\umod$, that is, that the (left-right) YD condition \eqref{yd2} is true.
  Indeed, one has
   \begin{small}
  \begin{equation*}
    \begin{split}
      & \big(\gd(\phi_j \cup_{\scriptscriptstyle \otimes} \psi_i)\big)
      (u^0, \ldots, u^{j+i+1}, m \otimes_A m')
      =
      (\phi_j \cup_{\scriptscriptstyle \otimes} \psi_i) \big(d(u^0, \ldots, u^{j+i+1}, m \otimes_A m')\big)
         \\
      &
      =
\textstyle\sum\limits_{k=0}^{j+i} (-1)^k      (\phi_j \cup_{\scriptscriptstyle \otimes} \psi_i) (u^0, \ldots, u^k u^{k+1}, \ldots, u^{j+i+1}, m \otimes_A m')
\\
&
\quad
+ (-1)^{j+i+1}  (\phi_j \cup_{\scriptscriptstyle \otimes} \psi_i) (u^0, \ldots, u^{j+i}, u^{j+i+1}_{(1)} m \otimes_A u^{j+i+1}_{(2)} m')
       \\
      &
      =
      \textstyle\sum\limits_{k=0}^{j} (-1)^k
     \phi_j\big(u^0_{(1)}, \ldots, u^{k}_{(1)} u^{k+1}_{(1)} , \ldots, u^{j+1}_{(1)}, (u^{j+2}_{(2)} \cdots u^{j+i+1}_{(2)} m')_{[0]}\big)
    \\
    &
    \qquad\qquad
    \otimes_A \psi_i\big(u^0_{(2)} \cdots u^k_{(2)} u^{k+1}_{(2)}  \cdots u^{j+1}_{(2)}
(u^{j+2}_{(2)} \cdots u^{j+i+1}_{(2)} m')_{[1]}
    , u^{j+2}_{(1)}, \ldots, u^{j+i+1}_{(1)}, m\big)
    \\
    &
    \quad
    +
          \textstyle\sum\limits_{k=j+1}^{j+i} (-1)^k
     \phi_j\big(u^0_{(1)},  \ldots, u^j_{(1)}, (u^{j+1}_{(2)} \cdots u^{k}_{(2)} u^{k+1}_{(2)}  \cdots u^{j+i+1}_{(2)} m')_{[0]}\big)
    \\
    &
    \qquad\quad
    \otimes_A \psi_i\big(u^0_{(2)} \cdots u^j_{(2)}
(u^{j+1}_{(2)} \cdots  u^{k}_{(2)} u^{k+1}_{(2)}  \cdots u^{j+i+1}_{(2)} m')_{[1]}
    , u^{j+1}_{(1)}, \ldots,  u^{k}_{(1)} u^{k+1}_{(1)}  \ldots, u^{j+i+1}_{(1)}, m\big)
    \\[2pt]
    &
    \quad
    + (-1)^{j+i+1}
 \phi_j\big(u^0_{(1)}, \ldots, u^j_{(1)}, (u^{j+1}_{(2)} \cdots u^{j+i+1}_{(2)} m')_{[0]}\big)
    \\[2pt]
    &
    \qquad\qquad
    \otimes_A \psi_i\big(u^0_{(2)} \cdots u^j_{(2)}
(u^{j+1}_{(2)} \cdots u^{j+i+1}_{(2)} m')_{[1]}
    , u^{j+1}_{(1)}, \ldots, u^{j+i}_{(1)},  u^{j+i+1}_{(1)} m\big)
     \end{split}
    \end{equation*}
\end{small}
  On the other hand,
  \begin{small}
  \begin{equation*}
    \begin{split}
      & (\gd \phi_j \cup_{\scriptscriptstyle \otimes} \psi_i)
      (u^0, \ldots, u^{j+i+1}, m \otimes_A m')
         \\
      &
      =
 \gd\phi_j\big(u^0_{(1)}, \ldots, u^{j+1}_{(1)}, (u^{j+2}_{(2)} \cdots u^{j+i+1}_{(2)} m')_{[0]}\big)
    \\
    &
    \qquad\qquad
    \otimes_A \psi_i\big(u^0_{(2)} \cdots u^{j+1}_{(2)}
(u^{j+2}_{(2)} \cdots u^{j+i+1}_{(2)} m')_{[1]}
    , u^{j+2}_{(1)}, \ldots, u^{j+i+1}_{(1)}, m\big)
         \\
      &
      =
      \textstyle\sum\limits_{k=0}^{j} (-1)^k \phi_j\big(u^0_{(1)}, \ldots,
 u^{k}_{(1)} u^{k+1}_{(1)}, \ldots,
      u^{j+1}_{(1)}, (u^{j+2}_{(2)} \cdots u^{j+i+1}_{(2)} m')_{[0]}\big)
    \\[-1pt]
    &
    \qquad\qquad
    \otimes_A \psi_i\big(u^0_{(2)} \cdots u^{j+1}_{(2)}
(u^{j+2}_{(2)} \cdots u^{j+i+1}_{(2)} m')_{[1]}
    , u^{j+2}_{(1)}, \ldots, u^{j+i+1}_{(1)}, m\big)
         \\[1pt]
      &
      \quad
      +
 (-1)^{j+1} \phi_j\big(u^0_{(1)}, \ldots,
     u^{j}_{(1)},  u^{j+1}_{(1)}(u^{j+2}_{(2)} \cdots u^{j+i+1}_{(2)} m')_{[0]}\big)
    \\
    &
    \qquad\qquad
    \otimes_A \psi_i\big(u^0_{(2)} \cdots u^{j+1}_{(2)}
(u^{j+2}_{(2)} \cdots u^{j+i+1}_{(2)} m')_{[1]}
    , u^{j+2}_{(1)}, \ldots, u^{j+i+1}_{(1)}, m\big)
       \\
      &
=
      \textstyle\sum\limits_{k=0}^{j} (-1)^k \phi_j\big(u^0_{(1)}, \ldots,
 u^{k}_{(1)} u^{k+1}_{(1)}, \ldots,
      u^{j+1}_{(1)}, (u^{j+2}_{(2)} \cdots u^{j+i+1}_{(2)} m')_{[0]}\big)
    \\[-1pt]
    &
    \qquad\qquad
    \otimes_A \psi_i\big(u^0_{(2)} \cdots  u^{k}_{(2)} u^{k+1}_{(2)} \cdots u^{j+1}_{(2)}
(u^{j+2}_{(2)} \cdots u^{j+i+1}_{(2)} m')_{[1]}
    , u^{j+2}_{(1)}, \ldots, u^{j+i+1}_{(1)}, m\big)
         \\[1pt]
      &
      \quad
      +
 (-1)^{j+1} \phi_j\big(u^0_{(1)}, \ldots,
     u^{j}_{(1)},  (u^{j+1}_{(2)} \cdots u^{j+i+1}_{(2)} m')_{[0]}\big)
    \\
    &
    \qquad\qquad
    \otimes_A \psi_i\big(u^0_{(2)} \cdots u^{j}_{(2)}
(u^{j+1}_{(2)} \cdots u^{j+i+1}_{(2)} m')_{[1]}u^{j+1}_{(1)}
    , u^{j+2}_{(1)}, \ldots, u^{j+i+1}_{(1)}, m\big), 
    \end{split}
    \end{equation*}
\end{small}
  where the Yetter-Drinfel'd condition \eqref{yd2} was used in the last step for the last summand. Moreover,
   \begin{small}
  \begin{equation*}
    \begin{split}
      & (-1)^j(\phi_j \cup_{\scriptscriptstyle \otimes} \gd \psi_i)
      (u^0, \ldots, u^{j+i+1}, m \otimes_A m')
         \\
      &
      =
 (-1)^j \phi_j\big(u^0_{(1)}, \ldots, u^j_{(1)}, (u^{j+1}_{(2)} \cdots u^{j+i+1}_{(2)} m')_{[0]}\big)
    \\
    &
    \qquad\qquad
    \otimes_A \gd \psi_i\big(u^0_{(2)} \cdots u^j_{(2)}
(u^{j+1}_{(2)} \cdots u^{j+i+1}_{(2)} m')_{[1]}
    , u^{j+1}_{(1)}, \ldots, u^{j+i+1}_{(1)}, m\big)
 \\
      &
 =
(-1)^j \phi_j\big(u^0_{(1)}, \ldots, u^j_{(1)}, (u^{j+1}_{(2)} \cdots u^{j+i+1}_{(2)} m')_{[0]}\big)
    \\
    &
    \qquad\qquad
    \otimes_A \psi_i\big(u^0_{(2)} \cdots u^j_{(2)}
(u^{j+1}_{(2)} \cdots u^{j+i+1}_{(2)} m')_{[1]}u^{j+1}_{(1)}, u^{j+2}_{(1)} \ldots, u^{j+i+1}_{(1)}, m\big)
    \\
    &
    \quad
    +
 \textstyle\sum\limits_{k=1}^{i} (-1)^{k+j} \phi_j\big(u^0_{(1)}, \ldots, u^j_{(1)}, (u^{j+1}_{(2)} \cdots u^{j+i+1}_{(2)} m')_{[0]}\big)
    \\[-1pt]
    &
    \qquad\quad
    \otimes_A \psi_i\big(u^0_{(2)} \cdots u^j_{(2)}
(u^{j+1}_{(2)} \cdots u^{j+i+1}_{(2)} m')_{[1]}
    , u^{j+1}_{(1)}, \ldots,  u^{k+j}_{(1)}u^{k+j+1}_{(1)}, \ldots, u^{j+i+1}_{(1)}, m\big)
       \\
    &
    \quad \
    +
 (-1)^{j+i+1} \phi_j\big(u^0_{(1)}, \ldots, u^j_{(1)}, (u^{j+1}_{(2)} \cdots u^{j+i+1}_{(2)} m')_{[0]}\big)
    \\[-1pt]
    &
    \qquad\qquad
    \otimes_A \psi_i\big(u^0_{(2)} \cdots u^j_{(2)}
(u^{j+1}_{(2)} \cdots u^{j+i+1}_{(2)} m')_{[1]}
    , u^{j+1}_{(1)}, \ldots,  u^{k+j+1}_{(1)}, u^{j+i+1}_{(1)}m\big).
    \end{split}
    \end{equation*}
\end{small} 
   By reindexing the second sum right above and subsequently comparing the three explicit expressions just obtained for
    $\gd(\phi_j \cup_{\scriptscriptstyle \otimes} \psi_i)$, as well as
    $\gd \phi_j \cup_{\scriptscriptstyle \otimes} \psi_i$ and $(-1)^{j} \phi_j \cup_{\scriptscriptstyle \otimes} \gd \psi_i$, one confirms that Eq.~\eqref{cupleibniz} is true. 
    \end{proof}

\subsection{A cocycle representative for the cup product}
\label{here}
We now have the technical prerequisites to explicitly illustrate the path indicated in \S\ref{torloniasantisubito}. As a start, we need to prove:

\begin{lemma}
  \label{lem:maywestart}
  If
$\phi\colon \mathrm{Bar}(U, X) \to \E$ and
$\psi\colon \mathrm{Bar}(U, X) \to \F$ 
  are chain maps as above, then the collection of maps $(\phi\cup_{\scriptscriptstyle \otimes} \psi)_k\colon \mathrm{Bar}_k(U,X\otimes_A X)\to \mathbf{Moloch}(\E,\F)_k$
  from Remark \ref{rem:annette}
induced by the external cup product \eqref{cupcup}
  defines  a $\tau$-twisted cocycle representative  $\phi\cup_{\scriptscriptstyle \otimes} \psi\colon \mathrm{Bar}(U,X\otimes_A X)\to \mathbf{Moloch}(\E,\F)$.
\end{lemma}

\begin{proof}
  This is a direct consequence of Lemma \ref{leibnizcup} and the Leibniz rule \eqref{cupleibniz}: to start with, let us prove, for $k = 1, \ldots, p+q$, that the diagram
\begin{equation*}
\xymatrix{
  \mathrm{Bar}_k(U,X \otimes_A X) \ar[r]^-{d} \ar[d]_{\sum\limits_{j+i = k}\!\!\!(\phi_j \, \cup_{\otimes}  \, \psi_i)} & \mathrm{Bar}_{k-1}(U,X \otimes_A X)
 \ar[d]^{\sum\limits_{j+i = k-1} \!\!\!\!\!\!(\phi_j \, \cup_{\otimes}  \, \psi_i)}
\\
\textstyle\bigoplus\limits_{{j+i = k}} \!\!\! E_j \otimes_A F_i \ar[r]^-d & \textstyle\bigoplus\limits_{j+i = k-1} \!\!\!\!\!\!E_j \otimes_A F_i  
}
\end{equation*}
commutes, where $i, j \geq 0$.
Indeed, by means of \eqref{blumare}, one has
\begin{equation*}
  \begin{split}
    d \circ \big(\textstyle\sum\limits_{\mathclap{j+i = k}} \phi_j  \cup_{\otimes}   \psi_i \big)
    & =  {\textstyle\sum\limits_{j+i = k}} \big( (d \circ \phi_j)  \cup_{\otimes}   \psi_i + (-1)^j \phi_j  \cup_{\otimes}   (d \circ \psi_i) \big)
    \\
      & =  {\textstyle\sum\limits_{j+i = k}} \big( \gd \phi_{j-1}  \cup_{\otimes}   \psi_i + (-1)^j \phi_j  \cup_{\otimes}   \gd \psi_{i-1} \big).
    \end{split}
  \end{equation*}
On the other hand, using the Leibniz rule \eqref{cupleibniz}, we compute
\begin{equation*}
  \begin{split}
   \big(\textstyle\sum\limits_{\mathclap{j+i = k-1}} \phi_j  \cup_{\otimes}   \psi_i \big) \circ d
   & =
      \textstyle\sum\limits_{{j+i = k-1}} \gd\big(\phi_j  \cup_{\otimes}   \psi_i \big) 
    \\
    & =  {\textstyle\sum\limits_{j+i = k-1}} \big( \gd \phi_{j}  \cup_{\otimes}   \psi_i + (-1)^j \phi_j  \cup_{\otimes}   \gd \psi_{i} \big)
        \\
      & =  {\textstyle\sum\limits_{j+i = k}} \gd \phi_{j-1}  \cup_{\otimes}   \psi_i + (-1)^j  {\textstyle\sum\limits_{j+i = k}} \phi_j  \cup_{\otimes}   \gd \psi_{i-1},
    \end{split}
  \end{equation*}
which proves the commutativity of the diagram above; for $k > p$ or $k > q$, this also needs the cocycle condition \eqref{famepigreco}. In lowest degree, {\em i.e.}, for $k=0$, we obtain the diagram
\begin{equation*}
\xymatrix{
  \mathrm{Bar}_0(U,X \otimes_A X) \ar[r]^-{L} \ar[d]_{\phi_0 \, \cup_{\otimes}  \, \psi_0} & X \otimes_A X
 \ar[d]^{\tau}
\\
E_0 \otimes_A F_0 \ar[r]^-{p_\ehhe \otimes_A p_\F} & X \otimes_A X 
}
\end{equation*}
where $\tau$ is the braiding as in \eqref{sahnejoghurt0a}. By virtue of \eqref{fame1047}, this reduces to the claim
$
L \cup_{\scriptscriptstyle{\otimes}} L = \tau \circ L.
$
Indeed, with \eqref{cupcup} and the (left-right) YD condition \eqref{yd2}, we directly have
\begin{equation*}
  \begin{split}
(L \cup_{\scriptscriptstyle{\otimes}} L)(u^0, m \otimes_A m') &=
L(u^0_{(1)}, m'_{[0]}) \otimes_A  L(u^0_{(2)} m'_{[1]}, m)
= u^0_{(1)} m'_{[0]} \otimes_A  u^0_{(2)} m'_{[1]} m
\\
&=
(u^0_{(2)} m')_{[0]} \otimes_A  (u^0_{(2)} m')_{[1]} u^0_{(1)} m
\\
&= \tau(u^0_{(1)} m \otimes_A  u^0_{(2)} m')
\\
&
= (\tau \circ L)(u^0, m \otimes_A m'),
  \end{split}
  \end{equation*}
using the monoidal structure \eqref{rain} in the last step.
  \end{proof}

\begin{rem} For the external cup product of two $0$-cocycle representatives  $\E\colon X \to Z$ resp.\ $\F\colon X \to Z$, the statement of Lemma \ref{lem:maywestart} amounts to the commutativity of
\[
\xymatrix{
  \mathrm{Bar}_0(U,X \otimes_A X) \ar[r]^-{L} \ar[d]_{\phi_0 \, \cup_{\otimes}  \, \psi_0} & X \otimes_A X
  \ar[d]^{
    \tau}
\\
Z \otimes_A Z & X \otimes_A X \ar[l]^{\E \otimes_A \F}
}
\]
which again is a consequence of the YD condition \eqref{yd2}, where
$\phi_0 = \E \circ L$ resp.\ $\psi_0 = \F \circ L$ are
as in \eqref{zerozero}. Equivalently, this is the commutativity of
\[
\xymatrix{
  \mathrm{Bar}_0(U,X \otimes_A X) \ar[r]^-{L} \ar[d]_{\phi_0 \, \cup_{\otimes}  \, \psi_0} & X \otimes_A X
  \ar[d]^{(\E \otimes_A \F) \circ
    \tau}
\\
Z \otimes_A Z 
\ar@{=}[r] & Z \otimes_A Z 
}
\]
expressing that $\phi_0  \cup_{\otimes}   \psi_0$ is a $\tau$-twisted $0$-cocycle representative, see Eq.~\eqref{eq:tonight}.

\end{rem}
    
\begin{cor}
  \label{cor:thatsit}
  Let $\lambda_{\ehhe,\effe}\colon  \mathbf{Moloch}(\E,\F)\to (\E  \otimes_A Z) \#(X \otimes_A \F)  $ be the morphism of complexes from Eq.~\eqref{hiddensee1}. Then
  $$
  \lambda_{\ehhe,\effe}(\phi\cup_{\otimes}\psi)\circ \tau_\bull\colon \mathrm{Bar}(U,X\otimes_A X) \to (\E  \otimes_A Z)\# (X \otimes_A \F)  
  $$
  is a $\tau^2$-twisted cocycle representative.
\end{cor}

\begin{proof}
  Immediate from Lemma \ref{lem:maywestart} and Remark \ref{rem:astray}.
 \end{proof}

Assuming from now on that $X$ is a comonoid in ${\mathscr{Z}}^r(\umod)$ with comultiplication  $\gD_\ikks \colon X \to X \otimes_A X$ and $Z$ a monoid in ${\mathscr{Z}}^\ell(\umod)$  with multiplication $  \mu \colon Z \otimes_A Z \to Z$,
our next step consists in giving explicit formul\ae{} for a cocycle representative for the object $\mu\#\mathbb{H}\#\Delta_\ikks$ in $\ext_U^\bull(X,Z)$, starting from a cocycle representative $\xi$ for an object $\mathbb{H}$ in $\ext^\bull_U(X\otimes_A X ,Z\otimes_A Z)$.

\begin{lem}
\label{monteolivetomaggiore}
Let $(Z, \mu)$ be a monoid in ${\mathscr{Z}}^\ell(\umod)$ and $(X, \gD_\ikks)$ a comonoid in ${\mathscr{Z}}^r(\umod)$, 
let $\mathbb{H}\in \ext^r_U(X\!\otimes_A\! X ,Z\!\otimes_A\! Z)$,
and $\xi\colon \mathrm{Bar}(U,X\! \otimes_A \!X)\to \mathbb{H}$ be a co- cycle representative for $\mathbb{H}$.
Then a cocycle representative for~$\mu\#\mathbb{H}\#\Delta_\ikks$~is~given~by
\[
\tilde{\xi}_k \coloneq
\begin{cases}
 \id_\ikks
&
  \text{for}\quad k=-1,
\\
  \big(\xi_0 \circ (\id\otimes_\Aop  \gD_\ikks ),L\big)
&
  \text{for}\quad k=0,
\\
  \xi_{k} \circ (\id^{k+1}\otimes_\Aop  \gD_\ikks )
  & \text{for}\quad 1\leq k\leq r-2,
  \\
  \overline{\big(0,\xi_{r-1} \circ (\id^{r}\otimes_\Aop  \gD_\ikks )\big)} &
  \text{for}\quad k=r-1,
  \\
  \mu  \circ \xi_r \circ (\id^{r+1}\otimes_\Aop  \gD_\ikks) & \text{for}\quad k=r,
\end{cases}
\]
where
$\overline{(\, , \,)}$ indicates the projection onto the quotient $Z\sqcup H_{r-1}\to Z\sqcup_{Z \otimes_A Z} H_{r-1}$, and where we abbreviated $\id^k \coloneq \id_\uhhu^{\otimes_\Aop k}$.
\end{lem}

\begin{proof}
  We start by computing a cocycle representative for $ \mathbb{H} \# \gD_\ikks$,
using the bar resolutions for $X$ and $X\otimes_A X$, respectively. Following the general prescription in \S\ref{sunshine}, we are to consider the commutative diagram  
\vskip .1 cm
\begin{footnotesize}
\[
\xymatrix{
\mathrm{Bar}_{r+1}(U,X)\ar[r]\ar[d]_{\id^{r+2}\otimes_\Aop  \gD_\ikks }
&\mathrm{Bar}_{r}(U,X)\ar[r]\ar[d]_{\id^{r+1}\otimes_\Aop  \gD_\ikks }&\mathrm{Bar}_{r-1}(U,X)\ar[r]\ar[d]_{\id^{r}\otimes_\Aop  \gD_\ikks }&\ldots\\
\mathrm{Bar}_{r+1}(U,X \otimes_A X)\ar[r]\ar[d]&\mathrm{Bar}_{r}(U,X \otimes_A X)\ar[r]\ar[d]_{\xi_{r}} &\mathrm{Bar}_{r-1}(U,X \otimes_A X)\ar[r]\ar[d]_{\xi_{r-1}}&\ldots\\
0\ar[r]&Z\otimes_A Z\ar[r] &H_{r-1}\ar[r]&\ldots
}
\]
\end{footnotesize}

\begin{footnotesize}
\[
\xymatrix{
  \ldots\ar[r]& \mathrm{Bar}_{1}(U,X)\ar[r]\ar[d]_{\id^{2}\otimes_\Aop  \gD_\ikks }& \mathrm{Bar}_{0}(U,X)\ar@/^1.4pc/[rr]\ar@{=}[r]\ar[d]_{\id\otimes_\Aop  \gD_\ikks }&\mathrm{Bar}_{0}(U,X)\ar[r]
  \ar[d]_{L \circ (\id\otimes_\Aop  \gD_\ikks )}\ar[r]&X\ar[d]^{ \gD_\ikks }\ar[r]&0
  \\
  \ldots\ar[r]& \mathrm{Bar}_1(U,X \otimes_A X) \ar[r] \ar[d]_{\xi_1}& \mathrm{Bar}_0(U,X \otimes_A X)
  \ar[d]_{\xi_0}\ar[r] &
         {X\otimes_A X}\ar@{=}[r]\ar@{=}[d] & X\otimes_A X\ar[r]\ar@{=}[d] & 0
         \\
\ldots\ar[r]& H_1\ar[r] \ar[r]& H_0\ar[r]\ar@/_1.4pc/[rr]&X\otimes_A X\ar@{=}[r]&{X\otimes_A X}\ar[r]&0
}
\]
\end{footnotesize}
\vskip .3 cm
\noindent Note that the commutativity of the upper two rows depends on the fact that $X$ is a comonoid in $\umod$, that is, explicitly from \eqref{radicchio4}.
From this diagram we obtain the representative cocycle
\[
\xymatrix@C=3.5em{
  \mathrm{Bar}_{r+1}(U,X)\ar[r]\ar[d]&\mathrm{Bar}_{r}(U,X)\ar[r]
  \ar[d]_{\xi_{r} \circ (\id^{r+1}\otimes_\Aop  \gD_\ikks )} &\mathrm{Bar}_{r-1}(U,X)\ar[r]
  \ar[d]_{\xi_{r-1} \circ (\id^{r}\otimes_\Aop  \gD_\ikks )}&\ldots\\
0\ar[r]&Z\otimes_A Z\ar[r] &H_{r-1}\ar[r]&\ldots
}
\]

\[
\xymatrix@C=3.5em{
\ldots\ar[r]& \mathrm{Bar}_{1}(U,X)\ar[r]\ar[d]_{\xi_{1} \circ (\id^{2}\otimes_\Aop  \gD_\ikks )}& \mathrm{Bar}_{0}(U,X)\ar[r]\ar[d]_{(L,\xi_{0} \circ (\id\otimes_\Aop  \gD_\ikks ))}&X\ar@{=}[d]\ar[r]&0
\\
\ldots\ar[r]& H_{1}\ar[r]_-{(0,d)}& X\times_{X\otimes_A X}H_{0}\ar[r]&X\ar[r]&0.
}
\]
Next, to compute a cocycle representative for $ \mu  \# \mathbb{H}\# \gD_\ikks $, we consider the commutative diagram
\vskip .2 cm
\begin{small}
\[
\xymatrix@C=3.4em{
  \mathrm{Bar}_{r+1}(U,X)\ar[r]\ar[d]_{\zeta_1}&\mathrm{Bar}_{r}(U,X)
  \ar@/^1.4pc/[rr]
  \ar@{=}[r]\ar[d]_{\zeta_0} &\mathrm{Bar}_{r}(U,X)\ar[r]
  \ar[d]_{\xi_{r} \circ (\id^{r+1}\otimes_\Aop  \gD_\ikks )}&\mathrm{Bar}_{r-1}(U,X)
  \ar[d]_{\xi_{r-1} \circ (\id^{r}\otimes_\Aop  \gD_\ikks )}\ar[r] & \ldots
  \\
  P_{1}^{Z\otimes_A Z}\ar[d]\ar[r]&P_{0}^{Z\otimes_A Z}\ar[r]\ar[d]_{} & Z\otimes_A Z\ar[d]_{ \mu }\ar[r]_-{i_\mathbb{H}} & H_{r-1}\ar[r]\ar[d]_{\overline{(0,\id)}}&\ldots
  \\
0\ar[r]&Z\ar@/_1.4pc/[rr]\ar@{=}[r] & Z \ar[r]^-{\overline{(-\id, 0)}} & Z\sqcup_{Z\otimes_A Z} H_{r-1} \ar[r]&\ldots
}
\]
\vskip .4 cm
\[
\xymatrix@C=3em{
\ldots\ar[r]& \mathrm{Bar}_{1}(U,X)\ar[r]\ar[d]_{\xi_{1} \circ (\id^{2}\otimes_\Aop  \gD_\ikks )}& \mathrm{Bar}_{0}(U,X)\ar[r]\ar[d]_{(\xi_{0} \circ (\id\otimes_\Aop  \gD_\ikks ),L)}&X\ar@{=}[d]\ar[r]&0
\\
\ldots\ar[r]& H_{1}\ar[r]\ar@{=}[d]& H_{0}\times_{X\otimes_A X}X\ar[r]\ar@{=}[d]&X\ar@{=}[d]\ar[r]&0
\\
\ldots\ar[r]& H_1\ar[r]& H_{0}\times_{X\otimes_A X}X\ar[r]&X\ar[r]&0,
}
\]
\end{small}
to see that $ \mu  \# \G\# \gD_\ikks $ is represented by the cocycle
\[
\xymatrix@C=4em@R=3em{
  \mathrm{Bar}_{r+1}(U,X)\ar[r]\ar[d]&\mathrm{Bar}_{r}(U,X)\ar[r]
  \ar[d]_{ \mu  \circ \xi_{r} \circ (\id^{r+1}\otimes_\Aop  \gD_\ikks )}&\mathrm{Bar}_{r-1}(U,X)
  \ar[d]_{\overline{(0,\xi_{r-1} \circ (\id^{r}\otimes_\Aop  \gD_\ikks ))}}\ar[r]&\ldots\\
0\ar[r]&Z\ar[r] & Z\sqcup_{Z \otimes_A Z} H_{r-1} \ar[r]&\ldots
}
\]

\[
\xymatrix@C=3em@R=3em{
  \ldots\ar[r]& \mathrm{Bar}_{1}(U,X)\ar[r]
  \ar[d]_{\xi_{1} \circ (\id^{2}\otimes_\Aop  \gD_\ikks )}& \mathrm{Bar}_{0}(U,X)\ar[r]\ar[d]_{(\xi_{0} \circ (\id\otimes_\Aop  \gD_\ikks ),L)}&X\ar@{=}[d]\ar[r]&0
\\
\ldots\ar[r]& H_{1}\ar[r]& H_{0}\times_{X\otimes_A X}X \ar[r]&X\ar[r]&0,
}
\]
from which the explicit expression for $\tilde{\xi}$ can be read off in all degrees.
\end{proof}

\begin{rem}
  \label{rem:casibassi0}
    For $r=2$, the statement of Lemma \ref{monteolivetomaggiore} has to be understood without the middle line, {\em i.e.},
   \begin{equation}
  \label{xitilde2}
\tilde{\xi}_k \coloneq
\begin{cases}
 \id_\ikks
 &
  \text{for}\quad k=-1,
\\
  \big(\xi_0 \circ (\id\otimes_\Aop  \gD_\ikks ),L\big)
&
  \text{for}\quad k=0,
  \\
  \overline{\big(0,\xi_{1} \circ (\id^{2}\otimes_\Aop  \gD_\ikks )\big)} &
  \text{for}\quad k=1,
  \\
  \mu  \circ \xi_2 \circ (\id^{3}\otimes_\Aop  \gD_\ikks) & \text{for}\quad k=2.
\end{cases}
\end{equation}
If $r = 1$, the two central lines of \eqref{xitilde2} are merged and one has
   \begin{equation}
  \label{xitilde3}
\tilde{\xi}_k \coloneq
\begin{cases}
 \id_\ikks
 &
  \text{for}\quad k=-1,
\\
   \overline{\big(0,\xi_0 \circ (\id\otimes_\Aop  \gD_\ikks ),L\big)}
&
  \text{for}\quad k=0,
  \\
  \mu  \circ \xi_1 \circ (\id^{2}\otimes_\Aop  \gD_\ikks) & \text{for}\quad k=1,
\end{cases}
\end{equation}
where $\tilde\xi_0$ is now taking values in  
   $$
Z\sqcup_{Z \otimes_A Z} H_{0} \times_{X\otimes_A X}X. 
   $$
Here we are omitting the parentheses thanks to the natural isomorphism of $U$-modules
$(Z\sqcup_{Z \otimes_A Z} H_{0}) \times_{X\otimes_A X}X\simeq Z\sqcup_{Z \otimes_A Z} (H_{0} \times_{X\otimes_A X}X)$.
Finally, for $r=0$ one removes the central line of \eqref{xitilde3}, obtaining
      \begin{equation*}
\tilde{\xi}_k \coloneq
\begin{cases}
 \mathbb{H}
 &
  \text{for}\quad k=-1,
  \\
  \mu  \circ \xi_0 \circ (\id\otimes_\Aop  \gD_\ikks) & \text{for}\quad k=0,
\end{cases}
\end{equation*}
where $\mathbb{H}\colon X\otimes_A X\to Z\otimes_A Z$ is the element in $\ext^0_U(X\otimes_A X, Z\otimes_A Z)$ of which $\xi$ is a $0$-cocycle representative.
\end{rem}

The above Lemma \ref{monteolivetomaggiore} suggests the following.

\begin{definition}
  \label{def:nepo}
  Let
$X$ and $Z$ be as before and let
  $\mathbb{H}\in \ext^r_U(X\otimes_A X,Z\otimes_A Z)$
  as well as $\G\coloneq\mu\#\mathbb{H}\#\Delta_\ikks$.
  Define the graded linear map
$$
  \Phi_\bull\colon \textstyle
      {\bigoplus\limits_{{k=-1}}^r}\Hom\big(\mathrm{Bar}_k(U, X\otimes_A X),{H}_{k}\big)
      \to \textstyle
          {\bigoplus\limits_{k=-1}^r}\Hom\big(\mathrm{Bar}_k(U,X) ,{G}_{k}\big)
$$
as follows:
\begin{equation}
  \label{phi}
\Phi_k(\xi) \coloneq
\begin{cases}
 0
 &
  \text{for}\quad k=-1,
\\
  \big(\xi_0 \circ (\id\otimes_\Aop  \gD_\ikks ),0\big)
&
  \text{for}\quad k=0,
\\
  \xi_{k} \circ (\id^{k+1}\otimes_\Aop  \gD_\ikks )
  & \text{for}\quad 1\leq k\leq r-2,
  \\
  \overline{\big(0,\xi_{r-1} \circ (\id^{r}\otimes_\Aop  \gD_\ikks )\big)} &
  \text{for}\quad k=r-1,
  \\
  \mu  \circ \xi_r \circ (\id^{r+1}\otimes_\Aop  \gD_\ikks) & \text{for}\quad k=r,
\end{cases}
\end{equation}
for any $\xi \in \textstyle\bigoplus_{k=-1}^r\Hom\big(\mathrm{Bar}_k(U, X\otimes_A X),H_{k}\big)$.
\end{definition}

\begin{rem}
  \label{rem:casibassi}
 Following Remark \ref{rem:casibassi0},
   for $r=2$, this has to be understood without the middle line, {\em i.e.},
   \begin{equation}
  \label{phi2}
\Phi_k(\xi)\coloneq
\begin{cases}
 0
 &
  \text{for}\quad k=-1,
\\
  \big(\xi_0 \circ (\id\otimes_\Aop  \gD_\ikks ),0\big)
&
  \text{for}\quad k=0,
  \\
  \overline{\big(0,\xi_{1} \circ (\id^{2}\otimes_\Aop  \gD_\ikks )\big)} &
  \text{for}\quad k=1,
  \\
  \mu  \circ \xi_2 \circ (\id^{3}\otimes_\Aop  \gD_\ikks) & \text{for}\quad k=2.
\end{cases}
\end{equation}
If $r = 1$, the two central lines of \eqref{phi2} are merged and one has
   \begin{equation}
  \label{phi3}
\Phi_k(\xi) \coloneq
\begin{cases}
 0
 &
  \text{for}\quad k=-1,
\\
   \overline{\big(0,\xi_0 \circ (\id\otimes_\Aop  \gD_\ikks ),0\big)}
&
  \text{for}\quad k=0,
  \\
  \mu  \circ \xi_1 \circ (\id^{2}\otimes_\Aop  \gD_\ikks) & \text{for}\quad k=1,
\end{cases}
\end{equation}
where $\Phi_0(\xi)$ takes values in  
   $
Z\sqcup_{Z \otimes_A Z} H_{0} \times_{X\otimes_A X}X
   $.
Finally, for $r=0$ one removes the central line of \eqref{phi3} to obtain
      \begin{equation*}
\Phi_k(\xi) \coloneq
\begin{cases}
 0
 &
  \text{for}\quad k=-1,
  \\
  \mu  \circ \xi_0 \circ (\id\otimes_\Aop  \gD_\ikks) & \text{for}\quad k=0,
\end{cases}
\end{equation*}
\end{rem}

  \begin{rem}
    \label{rem:Iwillhauntyou}
In terms of the linear map $\Phi$ from Definition \ref{def:nepo}, Lemma \ref{monteolivetomaggiore} can be stated by saying that if $\xi\colon \mathrm{Bar}(U, X\otimes_A X)\to\mathbb{H}$ is a cocycle representative for $\mathbb{H}\in \ext^r_U(X\otimes_A X,Z\otimes_A Z)$, with $r>0$, then 
\[
(0,\dots,0,(L,0),\id_\ikks)+ \Phi(\xi)
\]
is a cocycle representative for $\G = \mu\#\mathbb{H}\#\Delta_\ikks\in \ext^r_U(X,Z)$.
\end{rem}

\begin{lemma}
If $\xi\colon \mathrm{Bar}(U,X\otimes_A X) \to \mathbb{H}$ is an $f$-twisted cocycle
for the extension $\mathbb{H}\in \ext^r_U(X\otimes_A X,Z\otimes_A Z)$, with $r>0$, then $(0,\dots,0,(L,0),\id_\ikks)+ \Phi(\xi)$
is a cocycle representative for $\mu\#\mathbb{H}\#(f\circ \Delta_\ikks)\in \ext^r_U(X,Z)$.
\end{lemma}
\begin{proof}
Immediate from Remarks \ref{rem:verbatim} \& \ref{rem:Iwillhauntyou}.
\end{proof}

\begin{cor}
  \label{cor:theconductor}
Let $\xi\colon \mathrm{Bar}(U,X\otimes_A X) \to \mathbb{H}$ be an $f$-twisted cocycle
for the extension $\mathbb{H}\in \ext^r_U(X\otimes_A X,Z\otimes_A Z)$, with $r>0$, such that $f\circ\Delta_\ikks=\Delta_\ikks$. Then $(0,\dots,0,(L,0),\id_\ikks)+ \Phi(\xi)$ is a cocycle representative for $\mu\#\mathbb{H}\#\Delta_\ikks\in \ext^r_U(X,Z)$.
\end{cor}

If we additionally assume from now on that
$X$ is a {\em braided cocommutative} comonoid in ${\mathscr{Z}}^r(\umod)$,
by summing up we obtain the following proposition, which we give for strictly positive degrees leaving to the reader the easy task of deriving the corresponding statement for $0$-cocycles.

\begin{prop}
  \label{prop:eventually}
  Let $Z$ be as before and $X$ a braided cocommutative co\-mon\-oid in ${\mathscr{Z}}^r(\umod)$,
let $\E\in \ext^p_U(X,Z)$ and $\F\in \ext^q_U(X,Z)$, with $p,q>0$, and let $\phi\colon \mathrm{Bar}(X,U)\to \E$ and $\psi\colon \mathrm{Bar}(X,U)\to \F$ be cocycle representatives for $\E$ and $\F$, respectively. Moreover, let  $\phi\cup_{\scriptscriptstyle \otimes} \psi\colon \mathrm{Bar}(U,X\otimes_A X)\to \mathbf{Moloch}(\E,\F)$ be the $\tau$-twisted cocycle representative  from Lemma \ref{lem:maywestart}, and let $\lambda_{\ehhe,\effe}\colon  \mathbf{Moloch}(\E,\F)\to (X \otimes_A \F) \# (\E  \otimes_A Z)$ be the morphism of complexes from Eq.~\eqref{hiddensee1}. Then
\[
(0,\dots,0,(L,0),\id_\ikks)+\Phi(\lambda_{\ehhe,\effe}(\phi\cup_{\otimes}\psi)\circ \tau_\bull))
\]
is a cocycle representative for $\E \cup \F=   \mu \# (\E\otimes_A Z)\#(X\otimes_A\F)  \# \gD_\ikks$.
\end{prop}    

\begin{proof}
  Corollary \ref{cor:thatsit} states that $\lambda_{\ehhe,\effe}(\phi\cup_{\otimes}\psi)\circ \tau_\bull \colon \mathrm{Bar}(U,X\otimes_A X) \to (\E  \otimes_A Z)\#(X \otimes_A \F)  $ is a $\tau^2$-twisted cocycle representative.
By assumption
$\tau\circ\Delta_\ikks=\Delta_\ikks$ 
as in \eqref{radicchio6a}, therefore $\tau^2\circ\Delta_\ikks=\Delta_\ikks$, and one concludes by Corollary \ref{cor:theconductor}.
\end{proof}

\begin{definition}
\label{first}
  Let $\E\in \ext^p_U(X,Z)$ resp.\ $\F\in \ext^q_U(X,Z)$ be extensions of length $p$ resp.\ $q$, with cocycle representatives $\phi\colon \mathrm{Bar}(X,U)\to \E$ resp.\ $\psi\colon \mathrm{Bar}(X,U)\to \F$. The cocycle representative $\phi\cup\psi$ for $\E\cup\F$ is defined as
\[
\phi\cup\psi
\coloneq
(0,\dots,0,(L,0),\id_\ikks)+\Phi_\bull\big(\lambda_{\ehhe,\effe}(\phi\cup_{\otimes}\psi)\circ \tau\big) 
\]
for any $p, q \geq 0$.
\end{definition}

\begin{rem}
  By expanding the above definition, one sees that the top degree component of $\phi\cup\psi$
  is, for any $p,q \geq 0$, given by
  \begin{equation}
\begin{split}
    \label{namelythis}
    &(\phi\cup\psi)_{p+q}
 (u^0, \ldots, u^{p+q}, m)
    = \Phi_{p+q}(\lambda_{\ehhe,\effe}(\phi\cup_{\otimes}\psi)\circ \tau) (u^0, \ldots, u^{p+q}, m)
    \\
  \quad  &
  = \mu \circ (\phi_p \cup_\otimes \psi_q) \circ \tau_{p+q} \circ
  (\id^{p+q+1} \otimes_\Aop \Delta_\ikks )
  (u^0, \ldots, u^{p+q}, m)
 \\
\quad
&
=     \phi_p\big(u^0_{(1)}, \ldots, u^p_{(1)}, (u^{p+1}_{(2)} \cdots u^{p+q}_{(2)} m_{(2)})_{[0]}\big)
    \\
    &
    \qquad\qquad
    \cdot_\zett \psi_q\big(u^0_{(2)} \cdots u^p_{(2)}
(u^{p+1}_{(2)} \cdots u^{p+q}_{(2)} m_{(2)})_{[1]}
    , u^{p+1}_{(1)}, \ldots, u^{p+q}_{(1)}, m_{(1)}\big),
\end{split}
  \end{equation}
  when inserting the external cup product \eqref{cupcup}, 
  the edge morphism $\lambda_{\ehhe,\effe}$ as in \eqref{hiddensee1}, the top degree of $\Phi$ as in \eqref{phi}, and the braided cocommutativity \eqref{radicchio6a} on $X$.
\end{rem}

\subsection{The cup product as internal operadic composition}
Before finding, in the spirit of \S\ref{torloniasantisubito}, a second cocycle representative for $\E \cup \F$, we would like to understand how to reproduce the cup product from Eq.~\eqref{namelythis} from more operadic ideas. To this end, we need to give both an external as well as an internal operadic vertical composition first.

  \subsubsection{The external operadic composition}
\label{int}
As said, let us introduce an, in a sense, $2$-product, or
({partial})
{external operadic} composition
$$
\circ^{\scriptscriptstyle \otimes}_i\colon C^j(U, X, E_j) \otimes C^q(U, X, Z) \to C^{j+q-1}(U, X \otimes_A X, E_j \otimes_A Z),
$$
where
 $j = 1, \ldots, p$ and $i = 1, \ldots, j$, 
as follows:
\begin{small}
\begin{equation}
  \label{soso2}
  \begin{split}
    &(\phi_j  \circ^{\scriptscriptstyle \otimes}_i \psi_q)(u^0, u^1, \ldots, u^{p+q-1}, m \otimes_A m') \coloneq
    \\
    &
    \quad
    \phi_j\pig(u^0_{(1)}, \ldots, u^{i-1}_{(1)}, \psi_q\big(1, u^i_{(1)}, \ldots, u^{i+q-1}_{(1)},
    \\
    &
\quad
    (u^{i+q}_{(2)} \cdots u^{p+q-1}_{(2)} m')_{[0]} \big)_{(-1)}
    u^{i}_{(2)} \cdots u^{i+q-1}_{(2)} (u^{i+q}_{(2)} \cdots u^{p+q-1}_{(2)} m')_{[1]},
    u^{i+q}_{(1)}, \ldots,
        \\
&
    \quad
     u^{p+q-1}_{(1)} , m \pig)
       \otimes_A
       u^0_{(2)} \cdots u^{i-1}_{(2)}
     \psi_q\big(1, u^i_{(1)}, \ldots, u^{i+q-1}_{(1)}, (u^{i+q}_{(2)} \cdots u^{p+q-1}_{(2)} m')_{[0]} \big)_{(0)}.
    \end{split}
\end{equation}
\end{small}
Denote by
\begin{equation}
  \label{batteriefastalle}
\phi_j \,\bar\circ^{\scriptscriptstyle \otimes}\, \psi_q \coloneq \textstyle\sum\limits^{j}_{i=1} (-1)^{(i-1)(q-1)} \phi_j \circ^{\scriptscriptstyle \otimes}_i \psi_q
\end{equation}
the {\em full} external operadic composition (or {\em  external Gerstenhaber product}).
For a zero cochain, that is, for $j = 0$, put $\phi_0 \, \bar\circ^{\scriptscriptstyle \otimes}\, \psi_q \coloneq  0$.
The commutator with respect to this Gerstenhaber product will be denoted by
 $$
 [\psi_q,\phi_p]^{\bar\circ^{\scriptscriptstyle \otimes}} \coloneq
  \psi_q \, \bar\circ^{\scriptscriptstyle \otimes} \, \phi_p - (-1)^{(p-1)(q-1)} \phi_p
 \, \bar\circ^{\scriptscriptstyle \otimes} \, \psi_q. 
  $$

\begin{rem}
\label{auteur}
  Observe that it is, a priori, not so clear how to generalise \eqref{soso2} to a map
  $$
C^j(U, X, E_j) \otimes C^i(U, X, F_i) \to C^{j+i-1}(U, X \otimes_A X, E_j \otimes_A F_i)
$$
for $0 \leq i <q$
as there is in general no reason why, apart from $F_q = Z$, also the $U$-modules $F_i$ for $i \neq q$ should be $U$-comodules as well.
\end{rem}

\subsubsection{The internal operadic composition}\label{sec:operadic}
Precomposing resp.\ postcomposing the partial  external operadic composition \eqref{soso2} in case $j = p$  with the comultiplication $\Delta_\ikks$ on $X$ and the multiplication $\mu$ on $Z$, respectively, gives a true ({partial}) ({internal}) {operadic composition} on
the family
  \begin{equation}
    \label{jour}
\cO(n) \coloneq C^n(U,X,Z) = \Hom_U(\mathrm{Bar}_n(U, X), Z)
\end{equation}
  of $\kkk$-modules. This results into a map
  $$
\cO(p) \otimes \cO(q) \to \cO(p+q-1)
$$
explicitly given for 
any $\phi_p \in \cO(p)$, $\psi_q \in \cO(q)$, any $u^j \in U$, and $m \in X$ by
\begin{small}
\begin{equation}
  \label{mainsomma}
  \begin{split}
    &(\phi_p  \circ_i \psi_q)(u^0, u^1, \ldots, u^{p+q-1}, m) \coloneq
    \\
    &
        \phi_p\pig(u^0_{(1)}, \ldots, u^{i-1}_{(1)}, \psi_q\big(1, u^i_{(1)}, \ldots, u^{i+q-1}_{(1)},
    \\
    &
    (u^{i+q}_{(2)} \cdots u^{p+q-1}_{(2)} m_{(2)})_{[0]} \big)_{(-1)}
    u^{i}_{(2)} \cdots u^{i+q-1}_{(2)} (u^{i+q}_{(2)} \cdots u^{p+q-1}_{(2)} m_{(2)})_{[1]},
    u^{i+q}_{(1)}, \ldots,
        \\
&
         u^{p+q-1}_{(1)} , m_{(1)} \pig)
     \cdot_\zett \pig( u^0_{(2)} \cdots u^{i-1}_{(2)}
     \psi_q\big(1, u^i_{(1)}, \ldots, u^{i+q-1}_{(1)}, (u^{i+q}_{(2)} \cdots u^{p+q-1}_{(2)} m_{(2)})_{[0]} \big)_{(0)}
    \pig),
  \end{split}
\end{equation}
\end{small}
for any $i = 1, \ldots, p$.
Observe that the full structure of $Z$ and $X$ is used here: being left $U$-modules, being (left resp.\ right) $U$-comodules, and having a multiplication on $Z$ resp.\ a comultiplication on $X$.
That \eqref{mainsomma} indeed defines partial composition maps in an operad in the standard sense ({\em cf.}~Appendix \ref{pamukkale1}) will not be proven before \S\ref{appassionata} under the three assumptions that $X$ is a (braided) cocommutative comonoid in ${\mathscr{Z}}^r(\umod)$ and $Z$ a (braided) commutative monoid in ${\mathscr{Z}}^\ell(\umod)$ such that $(X, Z)$ yields a commuting pair.

Analogously to \eqref{batteriefastalle}, let us introduce
the {\em full} (internal) operadic composition
by defining
\begin{equation}
  \label{museicapitolini}
\phi_j \,\bar\circ \, \psi_q \coloneq \textstyle\sum\limits^{j}_{i=1} (-1)^{(i-1)(q-1)} \phi_j \circ_i \psi_q,
\end{equation}
which, again, might be called the
({\em internal}) {\em Gerstenhaber product}.

For the sake of illustration, we list a couple of special instances: for arbitrary $Z$ and $X = A$, the base algebra itself, one obtains:
\begin{small}
\begin{equation}
  \label{soso}
  \begin{split}
&(\phi_p \circ_i \psi_q)(u^0, u^1, \ldots, u^{p+q-1})
    \\
    &
    = \phi_p\big(u^0_{(1)}, \ldots, u^{i-1}_{(1)},
    \psi_q(1, u^i_{(1)}, \ldots, u^{i+q-1}_{(1)})_{(-1)}
u^i_{(2)}  \cdots u^{i+q-1}_{(2)}, u^{i+q}, \ldots, u^{p+q-1}
\big)
\\
&
\qquad
    \cdot_Z
    \big(
    u^0_{(2)} \cdots u^{i-1}_{(2)}
    \psi_q(1, u^i_{(1)}, \ldots, u^{i+q-1}_{(1)})_{(0)}
    \big).
    \end{split}
\end{equation}
\end{small}
For example, for $p = q =1$, this would be
\begin{equation}
            \label{barricade1}
(\phi_p \circ_1 \psi_q)(u, v)
= \phi_p\big(u_{(1)}, \psi_q(1, v_{(1)})_{(-1)} v_{(2)}\big) \cdot_Z \big(u_{(2)} \psi_q(1, v_{(1)})_{(0)}\big).
\end{equation}
On the other hand, let $Z =A$ and $X$ be arbitrary. Then \eqref{mainsomma} 
reduces to
\begin{small}
\begin{equation*}
  \label{mainsomma2}
  \begin{split}
    &(\phi_p  \circ_i \psi_q)(u^0, u^1, \ldots, u^{p+q-1}, m) =
    \\
    &
    \;
    \phi_p\pig(u^0, \ldots, u^{i-1},
    \psi_q\big(1, u^i_{(1)}, \ldots, u^{i+q-1}_{(1)},
        \\
&
    \;\;
    (u^{i+q}_{(2)} \cdots u^{p+q-1}_{(2)} m_{(2)})_{[0]} \big) \lact
    u^{i}_{(2)} \cdots u^{i+q-1}_{(2)} (u^{i+q}_{(2)} \cdots u^{p+q-1}_{(2)} m_{(2)})_{[1]},
    u^{i+q}_{(1)}, \ldots,
     u^{p+q-1}_{(1)} , m_{(1)} \pig).
  \end{split}
\end{equation*}
\end{small}
%
Finally, if both $X = Z = A$ equal the base algebra itself, that is, the unit object in the monoidal category $\umod$, we obtain
\begin{small}
\begin{equation*}
  \begin{split}
    &(\phi_p  \circ_i \psi_q)(u^0, u^1, \ldots, u^{p+q-1}) =
    \\
    &
    \quad
    \phi_p\pig(u^0, \ldots, u^{i-1}, \psi_q\big(1, u^i_{(1)}, \ldots, u^{i+q-1}_{(1)}\big) \lact
    u^{i}_{(2)} \cdots u^{i+q-1}_{(2)},
 u^{i+q}, \ldots,
     u^{p+q-1} \pig),
  \end{split}
\end{equation*}
\end{small}
which is sort of the nicest and makes it clear why people never bothered about introducing more general coefficients. Again, in case of the bialgebroid $(U, A) = (\Ae, A)$ and $X = Z = A$, employing the explicit bialgebroid structure of $\Ae$ as quoted in \eqref{purple}, this reduces to the classical Gerstenhaber insertion operations as they originally appeared in \cite{Ger:TCSOAAR},
that is,
\begin{small}
\begin{equation*}
  \begin{split}
    &(\phi_p  \circ_i \psi_q)(a^1, \ldots, a^{p+q-1}) =
   \phi_p\big(a^1, \ldots, a^{i-1}, \psi_q(a^i, \ldots, a^{i+q-1}),
 a^{i+q}, \ldots,
     a^{p+q-1} \big),
  \end{split}
\end{equation*}
\end{small}
where now the commata mean a tensor product over $\kkk$ instead of over $\Aop$, and where we identified $\Hom_\Ae\big((\Ae)^{\otimes_\Aop n+1}, A\big) \simeq \Hom_k(A^{\otimes_k n}, A)$.

\subsubsection{The internal cup product}
Introducing a special element $ \mu  \in \cO(2)$, 
  $$
   \mu (u^0, u^1, u^2, m) \coloneq \gve_\ikks(u^0u^1u^2m),
   $$
   the {\em multiplication}, 
which will again serve in \S\ref{appassionata}, 
one can define
the {\em (internal) cup product}
\begin{equation*}
  \begin{array}{rcl}
    \cup\colon \cO(p) \otimes \cO(q)
\!\!\!    &\to \!\!\! & \cO(p+q),
\end{array}
\end{equation*}
  in the customary way, {\em i.e.},
  \begin{equation}
    \label{cupi}
\phi_p \cup \psi_q = ( \mu  \circ_2 \psi_q) \circ_1 \phi_p.
 \end{equation}
For $\phi_p \in \cO(p)$, $\psi_q \in \cO(q)$, any $u^j \in U$ and $m \in X$ this explicitly comes out as 
\begin{equation}
\label{cupcupcup}
  \begin{split}
&   (\phi_p \cup \psi_q)(u^0, \ldots, u^{p+q}, m)
    \\
    &
    \quad
    =     \phi_p \big(u^0_{(1)}, \ldots, u^p_{(1)}, (u^{p+1}_{(2)} \cdots u^{p+q}_{(2)} m_{(2)})_{[0]}\big)
    \\
    &
    \qquad\qquad
    \cdot_\zett \psi_q \big(u^0_{(2)} \cdots u^p_{(2)}
(u^{p+1}_{(2)} \cdots u^{p+q}_{(2)} m_{(2)})_{[1]}
    , u^{p+1}_{(1)}, \ldots, u^{p+q}_{(1)}, m_{(1)}\big).
\end{split}
\end{equation}

\begin{rem}
  \label{winsstr}
Observe that this is now the same formula as the one obtained by a different approach in \eqref{namelythis}, with the striking difference that
here one does not need to start from the assumption of dealing with cocycles; indeed, in Eq.~\eqref{cupcupcup} above, both $\phi_p$ and $\psi_q$ are merely cochains.
\end{rem}

\begin{rem}
  As already hinted at in Remark \ref{auteur}, observe again that
as in \eqref{cupi}
  analogously writing 
$\phi_j \cup_{\scriptscriptstyle \otimes} \psi_i = ( \mu  \circ^{\scriptscriptstyle \otimes}_2 \psi_i) \circ^{\scriptscriptstyle \otimes}_1 \phi_j$ for the external cup product \eqref{cupcup} in terms of the external operadic operation \eqref{soso2}
does not make sense, and therefore \eqref{cupcup} had to be defined, sort of, by hand.
\end{rem}

Let us again discuss some special cases: for arbitrary $Z$ and $X= A$, Eq.~\eqref{cupcupcup} reduces to
   \begin{equation*}
  \begin{split}
&   (\phi_p \cup \psi_q)(u^0, \ldots, u^{p+q})
    =     \phi_p \big(u^0_{(1)}, \ldots, u^p_{(1)}\big)
    \cdot_\zett \psi_q \big(u^0_{(2)} \cdots u^p_{(2)}, u^{p+1}, \ldots, u^{p+q} \big),
\end{split}
\end{equation*} 
   which, from a slightly different viewpoint, precisely reflects the DG coalgebra structure of the bar resolution, whereas for arbitrary $X$ and $Z = A$, Eq.~\eqref{cupcupcup} by replacing the multiplication in $Z$ with the product in $A$ does not really simplify. In case $Z=X=A$, we therefore obtain
      \begin{equation*}
  \begin{split}
&   (\phi_p \cup \psi_q)(u^0, \ldots, u^{p+q})
    =     \phi_p (u^0_{(1)}, \ldots, u^p_{(1)})
    \psi_q (u^0_{(2)} \cdots u^p_{(2)}, u^{p+1}, \ldots, u^{p+q}),
\end{split}
\end{equation*}
      which, again, in case $(U,A) = (\Ae,A)$ reduces to the well-known cup product from \cite{Ger:TCSOAAR}, that is
   \begin{equation*}
  \begin{split}
&   (\phi_p \cup \psi_q)(a^1, \ldots, a^{p+q})
    =     \phi_p(a^1, \ldots, u^p)
    \psi_q (a^{p+1}, \ldots, a^{p+q}),
\end{split}
\end{equation*}
      where, once more, the commata mean a tensor product over $\kkk$ instead of over $\Aop$, and where we identified $\Hom_\Ae\!\big((\Ae)^{\otimes_\Aop n+1}, A\big) \simeq \Hom_k(A^{\otimes_k n}, A)$.

\subsection{A second cocycle representative for $\E\cup\F$}
\label{there}
As in \S\ref{torloniasantisubito}, we are now going to construct a second cocycle representative for $\E \cup \F$, see Proposition \ref{prop:neverwhere}, that in the next subsection will turn out to be homotopic to the one from Definition \ref{first}. For this, in particular, we need to find the, sort of,
correction term $\epsilon$ along with a couple of technical lemmata.

As above, for the whole section, let $\E\in \ext^p_U(X,Z)$ and $\F\in \ext^q_U(X,Z)$, with $p,q>0$, and let $\phi\colon \mathrm{Bar}(X,U)\to \E$ and $\psi\colon \mathrm{Bar}(X,U)\to \F$ be chain maps, {\em i.e.}, cocycle representatives for $\E$ and $\F$, respectively.

By switching the r\^oles of $\E$ and $\F$ in Remark \ref{rem:annette} and in Lemma \ref{lem:maywestart}, we immediately obtain the following.

\begin{lemma}
  \label{lem:maywestart2}
  If $\phi,\psi$ are chain maps
  as above, then the collection of maps $(\psi\cup_{\scriptscriptstyle \otimes} \phi)_k\colon \mathrm{Bar}(U,X\otimes_A X)\to \mathbf{Moloch}(\F,\E)_k$ induced by the external cup product
defines a $\tau$-twisted cocycle representative  $\psi\cup_{\scriptscriptstyle \otimes} \phi\colon \mathrm{Bar}(U,X\otimes_A X)\to \mathbf{Moloch}(\F,\E)$.
\end{lemma}

\noindent  Let us introduce the
  {\em (doubly)
    braided cup commutator}
  \begin{equation}
  \label{gewitter}
[\phi_j,\psi_q]_{\cup_{\otimes}, \gs, \tau} \coloneq (\phi_j \cup_{\scriptscriptstyle \otimes} \psi_q) \circ \tau_{j+q} - (-1)^{jq} \sigma \circ (\psi_q \cup_{\scriptscriptstyle \otimes} \phi_j),
\end{equation}
  where we wrote
  \begin{equation}
    \label{temporaletrapoco}
  \tau_{k} \coloneq (U^{\otimes_\Aop k+1} \otimes_\Aop \tau)
  \end{equation}
  in order to enhance notational beauty. In this sense, $\tau_{-1} \coloneq \tau$.
  
 \begin{lem}
\label{lem:homotopy-formula}
  Let $\phi_j \in C^j(U,X,E_j)$ and $\psi_q \in C^q(U, X, Z)$ and as before let $\gd = d^*$ be the pullback of the differential of the bar resolution.
If $(X, Z)$ is a commuting pair in the sense of Eq.~\eqref{centralascentralcan1}, 
  then the homotopy formula
  \begin{equation}
    \label{sahnejoghurt}
(-1)^{qj}[\phi_j,\psi_q]_{\cup_{\otimes}, \gs, \tau} 
  =
  (-1)^{q} \phi_j \,\bar\circ^{\scriptscriptstyle \otimes}\, \gd \psi_q
  -  (-1)^{q} \gd(\phi_j \,\bar\circ^{\scriptscriptstyle \otimes} \, \psi_q)
  - \gd \phi_j \,\bar\circ^{\scriptscriptstyle \otimes}\, \psi_q
  \end{equation}
  holds,
  where
  $\gs$ and $\tau$ are the left resp.\ right braidings from \eqref{sahnejoghurt0} resp.\ \eqref{sahnejoghurt0a}.
For $j = 0$, Eq.~\eqref{sahnejoghurt} reduces to
\begin{equation}
    \label{sahnejoghurt1}
[\phi_0,\psi_q]_{\cup_{\otimes}, \gs, \tau} 
  = - \gd \phi_0 \circ^{\scriptscriptstyle \otimes}_1 \psi_q.
    \end{equation}
\end{lem}

\begin{proof}
  This essentially works by
writing down and comparing all appearing terms one by one, not dissimilar to the operadic case;
which is a tedious but straightforward computation, we only indicate here a couple of decisive steps. 

  To this end, note if $\gd  = \sum_{k=0}^{j+1} (-1)^{k} \gd_k$ is the decomposition of the differential $\gd = d^*$ into its cofaces $\gd_k$, then one directly checks that
  %
  \begin{small}
  \begin{equation*}
    \begin{array}{rcl}
      &&
((\phi_j \cup_{\scriptscriptstyle \otimes} \psi_q) \circ \tau_{j+q})(u^0, \ldots, u^{j+q}, m \otimes_A m')
\\[1mm]
      &
      \stackrel{\scriptscriptstyle{\eqref{ydbraid2}}}{=}
      &
(\phi_j \cup_{\scriptscriptstyle \otimes} \psi_q)(u^0, \ldots, u^{j+q}, m'_{[0]} \otimes_A m'_{[1]}m)
      \\[1mm]
      &
      \stackrel{\scriptscriptstyle{\eqref{cupcup}}}{=}
      &
      \phi_j\big(u^0_{(1)}, \ldots, u^j_{(1)}, (u^{j+1}_{(2)} \cdots u^{j+q}_{(2)}  m'_{[1]}m )_{[0]}\big)
\\[1mm]
      &&
  \qquad\qquad    
    \otimes_A \psi_q\big(u^0_{(2)} \cdots u^j_{(2)}
(u^{j+1}_{(2)} \cdots u^{j+q}_{(2)}  m'_{[1]}m)_{[1]}
    , u^{j+1}_{(1)}, \ldots, u^{j+q}_{(1)},  m'_{[0]}\big)
    \\[1mm]
          &
      \stackrel{\scriptscriptstyle{}}{=}
      &
      \phi_j\big(u^0_{(1)}, \ldots, u^j_{(1)}, (u^{j+1}_{(2)} \cdots u^{j+q}_{(2)}  m'_{[1]}m )_{[0]}\big)
\\[1mm]
      &&
  \qquad\qquad    
    \otimes_A u^0_{(2)} \cdots u^j_{(2)}
(u^{j+1}_{(2)} \cdots u^{j+q}_{(2)}  m'_{[1]}m)_{[1]} \psi_q\big(1
    , u^{j+1}_{(1)}, \ldots, u^{j+q}_{(1)},  m'_{[0]}\big)
            \\[1mm]
      &
      \stackrel{\scriptscriptstyle{\eqref{centralascentralcan2}}}{=}
      &
  \phi_j\big(u^0_{(1)}, \ldots, u^j_{(1)}, \psi_q\big(1
    , u^{j+1}_{(1)}, \ldots, u^{j+q}_{(1)},  m'_{[0]}\big)_{(-1)} u^{j+1}_{(2)} \cdots u^{j+q}_{(2)}  m'_{[1]}m \big)
\\[1mm]
      &&
  \qquad\qquad    
    \otimes_A u^0_{(2)} \cdots u^j_{(2)}
\psi_q\big(1
    , u^{j+1}_{(1)}, \ldots, u^{j+q}_{(1)},  m'_{[0]}\big)_{(0)} 
      \\[1mm]
      &
      \stackrel{\scriptscriptstyle{}}{=}
      &
   \gd_{j+1} \phi_j\big(u^0_{(1)}, \ldots, u^j_{(1)}, \psi_q\big(1
    , u^{j+1}_{(1)}, \ldots, u^{j+q}_{(1)},  m'_{[0]}\big)_{(-1)} u^{j+1}_{(2)} \cdots u^{j+q}_{(2)}  m'_{[1]}, m \big)
\\[1mm]
      &&
  \qquad\qquad    
    \otimes_A u^0_{(2)} \cdots u^j_{(2)}
\psi_q\big(1
    , u^{j+1}_{(1)}, \ldots, u^{j+q}_{(1)},  m'_{[0]}\big)_{(0)} 
      \\[1mm]
      &
      \stackrel{\scriptscriptstyle{\eqref{soso2}}}{=}
      &
 ((\gd_{j+1} \phi_j) \circ^{\scriptscriptstyle \otimes}_{j+1} \psi_q)(u^0, \ldots, u^{j+q}, m \otimes_A m'),
       \end{array}
   \end{equation*}
  \end{small}
    where the fact that $(X,Z)$ is a commuting pair was used in the fourth step;
  that is,
  \begin{equation}
  \label{bois1}
(\phi_j \cup_{\scriptscriptstyle \otimes} \psi_q) \circ \tau_{j+q} = (\gd_{j+1} \phi_j) \circ^{\scriptscriptstyle \otimes}_{j+1} \psi_q,
  \end{equation}
in analogy to the operadic case.
On the other hand,
 \begin{small}
  \begin{equation*}
    \begin{array}{rcl}
      &&
 \sigma\pig((\psi_q \cup_{\scriptscriptstyle \otimes} \phi_j)(u^0, \ldots, u^{q+j}, m \otimes_A m')\pig)
\\[1mm]
      &
      \stackrel{\scriptscriptstyle{\eqref{cupcup}}}{=}
      &
      \gs \pig(
 \psi_q\big(u^0_{(1)}, \ldots, u^q_{(1)}, (u^{q+1}_{(2)} \cdots u^{q+j}_{(2)} m')_{[0]}\big)
    \\
    &&
    \qquad\qquad
    \otimes_A \phi_j\big(u^0_{(2)} \cdots u^q_{(2)}
(u^{q+1}_{(2)} \cdots u^{q+j}_{(2)} m')_{[1]}
    , u^{q+1}_{(1)}, \ldots, u^{q+j}_{(1)}, m\big)\pig)     
\\[1mm]
    &
      \stackrel{\scriptscriptstyle{\eqref{ydbraid}}}{=}
      &
 \psi_q\big(u^0_{(1)}, \ldots, u^q_{(1)}, (u^{q+1}_{(2)} \cdots u^{q+j}_{(2)} m')_{[0]}\big)_{(-1)} \phi_j\big(u^0_{(2)} \cdots u^q_{(2)}
 (u^{q+1}_{(2)} \cdots u^{q+j}_{(2)} m')_{[1]},
 \\[1mm]
 &&
 \quad
     u^{q+1}_{(1)}, \ldots, u^{q+j}_{(1)}, m\big)
\otimes_A
 \psi_q\big(u^0_{(1)}, \ldots, u^q_{(1)}, (u^{q+1}_{(2)} \cdots u^{q+j}_{(2)} m')_{[0]}\big)_{(0)}
            \\[1mm]
      &
      \stackrel{\scriptscriptstyle{\eqref{yd}}}{=}
      &
\phi_j\pig(u^0_{(1)} \psi_q\big(1, u^1_{(1)}, \ldots, u^{q}_{(1)},
    \\
    &&
\quad
    (u^{q+1}_{(2)} \cdots u^{q+j}_{(2)} m')_{[0]} \big)_{(-1)}
    u^{1}_{(2)} \cdots u^{q}_{(2)} (u^{q+1}_{(2)} \cdots u^{q+j}_{(2)} m')_{[1]},
    u^{q+1}_{(1)}, \ldots,
        \\
&&
    \quad
     u^{q+j}_{(1)} , m \pig)
       \otimes_A
       u^0_{(2)} \psi_q\big(1, u^1_{(1)}, \ldots, u^{q}_{(1)}, (u^{q+1}_{(2)} \cdots u^{q+j}_{(2)} m')_{[0]} \big)_{(0)}
      \\[1mm]
      &
      \stackrel{\scriptscriptstyle{}}{=}
      &
  \gd_0 \phi_j\pig(u^0_{(1)}, \psi_q\big(1, u^1_{(1)}, \ldots, u^{q}_{(1)},
    \\
    &&
\quad
    (u^{q+1}_{(2)} \cdots u^{q+j}_{(2)} m')_{[0]} \big)_{(-1)}
    u^{1}_{(2)} \cdots u^{q}_{(2)} (u^{q+1}_{(2)} \cdots u^{q+j}_{(2)} m')_{[1]},
    u^{q+1}_{(1)}, \ldots,
        \\
&&
    \quad
     u^{q+j}_{(1)} , m \pig)
       \otimes_A
       u^0_{(2)} \psi_q\big(1, u^1_{(1)}, \ldots, u^{q}_{(1)}, (u^{q+1}_{(2)} \cdots u^{q+j}_{(2)} m')_{[0]} \big)_{(0)}
    \\[1mm]
      &
      \stackrel{\scriptscriptstyle{\eqref{soso2}}}{=}
      &
 ((\gd_0 \phi_j) \circ^{\scriptscriptstyle \otimes}_1 \psi_q)(u^0, \ldots, u^{q+j}, m \otimes_A m'),
       \end{array}
   \end{equation*}
  \end{small}
%
%
%
%
%
%
using the $U$-linearity of both $\psi_q$ and $\phi_j$, and where in the third step we used the fact that the braiding $\gs$ is morphism in $\umod$.
Hence,
\begin{equation}
  \label{bois2}
\sigma \circ (\psi_q \cup_{\scriptscriptstyle \otimes} \phi_j) = (\gd_0 \phi_j) \circ^{\scriptscriptstyle \otimes}_1 \psi_q,
\end{equation}
from which also \eqref{sahnejoghurt1} immediately follows.

With the two identities \eqref{bois1} and \eqref{bois2} at hand, verifying the homotopy \eqref{sahnejoghurt} becomes a feasible task and its full verification is left to the reader.
\end{proof}

Again by Remark \ref{auteur}, one cannot write down a version of the right hand side of \eqref{sahnejoghurt} with the positions of $\phi_j$ and $\psi_q$ interchanged (as one could, up to signs, in the standard operadic situation). Also, the left hand side, {\em i.e.}, the braided cup commutator \eqref{gewitter} would not make sense with 
 $\phi_j$ and $\psi_q$ in opposite order.

\begin{lem}
  \label{lem:xi-and-epsilon}
Let $\phi\colon \mathrm{Bar}(U, X) \to \E$ and
$\psi\colon \mathrm{Bar}(U, X) \to \F$ be as above, and let 
$$
\epsilon_k(\phi, \psi) \colon \mathrm{Bar}_k(U, X \otimes_A X) \to \mathbf{Moloch}(\F ,\E)_k
$$
for any $k = -1, 0, \ldots, p+q$ be defined by
\[
\epsilon_k(\phi,\psi)=
\begin{cases}
0 & \mbox{for} \ k < p,
\\[1mm]
-
[\psi_{k-p},\phi_p]_{\cup_{\otimes}, \gs, \tau}
  {}
\\
\
+
(-1)^{pk+k +1
}
(F_{k-p+1} \otimes_A i_\ehhe) \circ (\psi_{k-p+1} \,\bar\circ^{\scriptscriptstyle \otimes}\,\phi_p) &
\mbox{for} \ p \leq k \leq p+q-1,
  \\[1mm]
  -
  [\psi_q,\phi_p]_{\cup_{\otimes}, \gs, \tau}
& \mbox{for} \ k = p+q. 
  \end{cases}
  \]
If $(X, Z)$ is a commuting pair in the sense of Eq.~\eqref{centralascentralcan1},  then $\epsilon(\phi,\psi) \colon \mathrm{Bar}(U, X \otimes_A X) \to \mathbf{Moloch}(\F ,\E)$ is a $0$-twisted cocycle representative.
 \end{lem}

\begin{proof}
  By definition of $0$-twisted cocycle representative,  we have to show that $\epsilon$ is a chain map with $\epsilon_{-1}=0$. The last part of the statement is true by definition of $\epsilon$, so we have only to prove that $\epsilon$ is a chain map. To start with, let $p < k \leq p+q-1$; the cases for $k = p+q$ and $k = p$ will be proven separately.
  Observe first that since $d$ is a $U$-module map,
  we have from \eqref{sahnejoghurt0} and \eqref{sahnejoghurt0a} that
  \begin{equation}
    \label{hub}
    \begin{array}{rcl}
      (d \otimes_A Z) \circ \gs &=& \gs \circ (Z \otimes_A d),
      \\
      (X \otimes_A d) \circ \tau &=& \tau \circ (d \otimes_A X).
\end{array}
  \end{equation}
Using again the notation  $\tau_k = (U^{\otimes_\Aop k} \otimes_\Aop \tau)$ 
as in \eqref{temporaletrapoco}, 
we compute with $d \circ i_\ehhe = 0$, Eqs.~\eqref{cupcup}, \eqref{gewitter}, and \eqref{blumare} as well as \eqref{hub}:
\begin{equation*}
  \begin{split}
& d \circ \epsilon_k(\phi,\psi)
\\
    &= d \circ \pig(
    -[\psi_{k-p},\phi_p]_{\cup_{\otimes}, \gs, \tau}+(-1)^{pk+k +1} (F_{k-p+1} \otimes_A i_\ehhe) \circ (\psi_{k-p+1} \,\bar\circ^{\scriptscriptstyle \otimes}\,\phi_p)\pig)
    \\
    &=
    -(d \circ \psi_{k-p}) \cup_{\scriptscriptstyle \otimes} \phi_p \circ \tau_{k-1}
  +
    (-1)^{p(k+1)} \gs \circ \big(\phi_p \cup_{\scriptscriptstyle \otimes} (d \circ \psi_{k-p})\big)
 \\
    &
    \quad
      + (-1)^{k-p-1} (F_{k-p} \otimes_A i_\ehhe)\circ [\psi_{k-p},\phi_p]_{\cup_{\otimes}, \gs, \tau}
      \\
    &
    \quad
       +(-1)^{pk+k +1} (d \otimes_A i_\ehhe) \circ (\psi_{k-p+1} \,\bar\circ^{\scriptscriptstyle \otimes}\,\phi_p)
      \\
   &=
    -(\gd\psi_{k-p-1}) \cup_{\scriptscriptstyle \otimes} \phi_p \circ \tau_{k-1}
  +
    (-1)^{p(k+1)} \gs \circ \big(\phi_p \cup_{\scriptscriptstyle \otimes} (\gd\psi_{k-p-1})\big)
 \\
    &
    \quad
      + (-1)^{k-p-1} (F_{k-p} \otimes_A i_\ehhe)\circ [\psi_{k-p},\phi_p]_{\cup_{\otimes}, \gs, \tau}
       \\
    &
    \quad
       +(-1)^{pk+k +1} (F_{k-p} \otimes_A i_\ehhe) \circ (\gd\psi_{k-p} \,\bar\circ^{\scriptscriptstyle \otimes}\,\phi_p),
    \end{split}
  \end{equation*}
whereas, on the other hand, using the homotopy formula \eqref{sahnejoghurt} 
and the fact \eqref{cupleibniz} that $\gd$ is a derivation of the cup product, 
along with $\gd\phi_p = 0$, 
\begin{equation*}
  \begin{split}
    & \epsilon_{k-1}(\phi,\psi) \circ d
\\
    &= \pig(
    -[\psi_{k-p-1},\phi_p]^{\cup, \gs,\tau}+(-1)^{p(k-1)+k} (F_{k-p} \otimes_A i_\ehhe) \circ (\psi_{k-p} \,\bar\circ^{\scriptscriptstyle \otimes}\,\phi_p)\pig) \circ d
    \\
    &=
    - (\gd \psi_{k-p-1} \cup_{\scriptscriptstyle \otimes} \phi_p) \circ \tau_{k}
+ (-1)^{pk} \gs \circ \big(( \phi_p \cup_{\scriptscriptstyle \otimes} \psi_{k-p-1}) \circ d\big)
\\
&
\quad
+(-1)^{p(k-1)+k} (F_{k-p} \otimes_A i_\ehhe) \circ \big((\psi_{k-p} \,\bar\circ^{\scriptscriptstyle \otimes}\,\phi_p) \circ d\big)
  \\
    &=
    - (\gd \psi_{k-p-1} \cup_{\scriptscriptstyle \otimes} \phi_p) \circ \tau_{k}
  + (-1)^{p(k+1)} \gs \circ ( \phi_p \cup_{\scriptscriptstyle \otimes} \gd\psi_{k-p-1})
      \\
    &
    \quad
+(-1)^{p(k-1)+k} (F_{k-p} \otimes_A i_\ehhe) \circ \gd (\psi_{k-p} \,\bar\circ^{\scriptscriptstyle \otimes}\,\phi_p),
  \end{split}
\end{equation*}
which by the homotopy relation \eqref{sahnejoghurt} for the braided cup commutator is the same as before, hence $d \circ \epsilon_k(\phi, \psi) =\epsilon_{k-1}(\phi, \psi) \circ d$ for  $p < k \leq p+q-1$. The case $k = p+q$ is proven similar and left to the reader. As for the case $k = p$, one has to show that
$
       d \circ \pig([\psi_{0},\phi_p]_{\cup_{\otimes}, \gs, \tau}+ (F_{1} \otimes_A i_\ehhe) \circ (\psi_{1} \,\bar\circ^{\scriptscriptstyle \otimes}\,\phi_p)\pig) = 0.
$
Indeed, with the help of \eqref{sahnejoghurt1} and the fact that $(d_\effe \otimes_A Z)(F_0 \otimes_A Z) = 0$ in $\mathbf{Moloch}(\F ,\E)$, 
       \begin{equation*}
  \begin{split}
    &
       d \circ \big(
  [\psi_{0},\phi_p]_{\cup_{\otimes}, \gs, \tau}+ (F_{1} \otimes_A i_\ehhe) \circ (\psi_{1} \circ^{\scriptscriptstyle \otimes}_1\phi_p)\big) 
   \\
    &=
  (F_0 \otimes_A i_\ehhe) \circ [\psi_{0},\phi_p]_{\cup_{\otimes}, \gs, \tau}+ (d \otimes_A i_\ehhe) \circ (\psi_{1} \circ^{\scriptscriptstyle \otimes}_1 \phi_p)
   \\
    &=
     (F_0 \otimes_A i_\ehhe)  \circ [\psi_{0},\phi_p]_{\cup_{\otimes}, \gs, \tau} + (F_0 \otimes_A i_\ehhe) \circ (\gd \psi_{0} \circ^{\scriptscriptstyle \otimes}_1 \phi_p)
     \\
    &=
      (F_0 \otimes_A i_\ehhe)  \circ [\psi_{0},\phi_p]_{\cup_{\otimes}, \gs, \tau} - (F_0 \otimes_A i_\ehhe) \circ [\psi_{0},\phi_p]_{\cup_{\otimes}, \gs, \tau}=0,
\end{split}
\end{equation*}
       which ends the proof.
\end{proof}

\begin{cor}\label{cor:wheredowegofromhere}
  Let $\epsilon$ 
  be as in Lemma \ref{lem:xi-and-epsilon}
and $(X,Z)$ a commuting pair subject to the same assumptions as before, and let $\gvr_{\effe,\ehhe}\colon  \mathbf{Moloch}(\F,\E)\to (Z\otimes_A\E)\#(\F \otimes_A X) $ be the morphism of complexes obtained by switching the r\^oles of $\E$ and $\F$ in \eqref{hiddensee2}. Then 
\[
\gvr_{\effe,\ehhe}(\psi\cup_{\scriptscriptstyle \otimes} \phi+\epsilon(\phi,\psi))\colon \mathrm{Bar}(U,X\otimes_A X)\to (Z\otimes_A\E)\#(\F \otimes_A X) 
\]
is a $\tau$-twisted cocycle representative,  which implies that
\[
(\sigma|\tau)_{\ehhe,\effe}\circ \gvr_{\effe,\ehhe}(\psi\cup_{\scriptscriptstyle \otimes} \phi+\epsilon(\phi,\psi))\colon \mathrm{Bar}(U,X\otimes_A X)\to (\E\otimes_A Z)\#(X\otimes_A \F) 
\]
is a $\tau^2$-twisted cocycle representative.
\end{cor}

\begin{proof}
  From Lemmata \ref{lem:maywestart2} \&  \ref{lem:xi-and-epsilon} we have that $\psi\cup_{\scriptscriptstyle \otimes} \phi+\epsilon(\phi,\psi)\colon \mathrm{Bar}(U,X\otimes_A X)\to \mathbf{Moloch}(\F,\E)$ is a $\tau$-twisted cocycle representative. By noticing that $(\gvr_{\effe,\ehhe})_{-1}=\id_{X\otimes_A X}$ and with the help of Remark \ref{rem:astray}, one proves the first statement. The second is then obvious.
 \end{proof}

%

Finally, 
reasoning as in the proof of Proposition \ref{prop:eventually}, we obtain:

\begin{prop}
  \label{prop:neverwhere}
In the same notation as above,
\[
(0,\dots,0,(L,0),\id_\ikks)+\Phi\big((\sigma|\tau)_{\ehhe,\effe}\circ \gvr_{\effe,\ehhe}(\psi\cup_{\scriptscriptstyle \otimes} \phi+\epsilon(\phi,\psi))\big)
\]
is as well a cocycle representative for $\E \cup \F=   \mu \# (\E\otimes_A Z)\# (X\otimes_A\F)  \# \gD_\ikks$, where $\Phi$ is defined in \eqref{phi}.
\end{prop}

\subsection{Brackets from homotopies}
So far, in \S\ref{here} and \S\ref{there}, we have exhibited two representative cocycles for $\E \cup \F$, namely $\phi\cup\psi=(0,\dots,0,(L,0),\id_\ikks)+\Phi(\lambda_{\ehhe,\effe}(\phi\cup_{\otimes}\psi)\circ \tau_\bull))$ as well as $(0,\dots,0,(L,0),\id_\ikks)+\Phi((\sigma|\tau)_{\ehhe,\effe}\circ \gvr_{\effe,\ehhe}(\psi\cup_{\scriptscriptstyle \otimes} \phi+\epsilon(\phi,\psi)))$. We now want to produce an explicit homotopy between these two.
Writing
\begin{equation}
  \label{eq:eta}
 \eta(\phi, \psi) \coloneq  \gl_{\ehhe, \effe}\big((\phi \cup_{\scriptscriptstyle \otimes} \psi) \circ \tau\big)
    - 
(\gs|\tau)_{\ehhe,\effe} \circ 
    \gvr_{\effe, \ehhe}\big((
    \psi \cup_{\scriptscriptstyle \otimes} \phi +\epsilon({\phi,\psi})\big),
  \end{equation}
where $(\gs|\tau)_{\ehhe,\effe}$ is the morphism of complexes defined in \eqref{eq:almost-there},
this amounts to exhibiting a homotopy between $\Phi(\eta(\phi,\psi))$ and zero.


Before we do so, let us give an explicit expression for $\eta$ in \eqref{eq:eta}:

 \begin{lemma}
Under the standing assumptions, we have
\[
(\gs|\tau)_{\effe,\ehhe} \circ \gvr_{\ehhe, \effe}(
\phi \cup_{\scriptscriptstyle \otimes} \psi)
- 
\gl_{\effe, \ehhe}\big((\psi \cup_{\scriptscriptstyle \otimes} \phi )\circ \tau_\bull  + \epsilon(\phi, \psi) \big)
=0,
\]
as well as
\begin{equation}
  \label{stabiloboss}
    \begin{split}
&\eta(\phi, \psi) =
 \gl_{\ehhe, \effe}\big((\phi \cup_{\scriptscriptstyle \otimes} \psi) \circ \tau_\bull\big)
    - 
(\gs|\tau)_{\ehhe,\effe} \circ 
    \gvr_{\effe, \ehhe}\big((
  \psi \cup_{\scriptscriptstyle \otimes} \phi +\epsilon(\phi,\psi)\big)
\\
 &
\quad
=
\begin{cases}
0 & \mbox{for} \ k < q,
\\
{[}\phi_{k-q}, \psi_q]_{\cup_{\otimes}, \gs, \tau} & \mbox{for} \ q \leq k \leq p+q-2,
\\
  {[}\phi_{p-1},\psi_q]_{\cup_{\otimes}, \gs, \tau} &
  \\
  \qquad
  + (-1)^{p+q} \gs \circ (Z \otimes_A i_\ehhe) \circ (\psi_q \,\bar\circ^{\scriptscriptstyle \otimes} \, \phi_p) 
& \mbox{for} \ k = p+q-1,
\\
  {[}\phi_{p},\psi_q]_{\cup_{\otimes}, \gs, \tau}
  + (-1)^{pq} \gs \circ {[}\psi_{q},\phi_p]_{\cup_{\otimes}, \gs, \tau}  & \mbox{for} \ k = p+q,
\end{cases}
\end{split}
  \end{equation}
 for any two chain maps $\phi
  \colon \mathrm{Bar}(U,X) \to \E$ and
$\psi\colon \mathrm{Bar}(U,X) \to \F$ as before.
\end{lemma}

\begin{proof}
  Observe first that explicitly looking at the edge morphisms \eqref{hiddensee1} and \eqref{hiddensee2} yields
  \begin{equation}
    \label{fierydawn1}
  \gvr^{}_{\effe,\ehhe}\epsilon_k(\phi,\psi)=
\begin{cases}
0 & \mbox{for} \ k < p+q-1,
\\
(-1)^{p+q} (Z \otimes_A i_{\ehhe}) \circ (\psi_q \,\bar\circ^{\scriptscriptstyle \otimes}\, \phi_p)
  & \mbox{for} \ k = p+q-1,
\\
-(-1)^{pq}[\psi_q,\phi_p]_{\cup_{\otimes}, \gs, \tau}
& \mbox{for} \ k = p+q, 
\end{cases}
\end{equation}
as well as
\begin{equation}
\label{fierydawn2}
  \lambda^{}_{\effe,\ehhe}\epsilon_k(\phi,\psi)=
\begin{cases}
0 & \mbox{for} \ k < p,
\\
-
[\psi_{k-p},\phi_p]_{\cup_{\otimes}, \gs, \tau} & \mbox{for} \ p \leq k \leq p+q.
\end{cases}
\end{equation}
Slightly less straightforward, one also has
\begin{equation}
  \label{fierydawn3}
\begin{split}
 & \textstyle\sum\limits_{j+i = k}\pig( (\gs|\tau)^{}_{\effe,\ehhe} \circ \gvr^{}_{\ehhe,\effe}(
  \phi_j \cup_{\scriptscriptstyle \otimes} \psi_i) - 
  \gl^{}_{\effe,\ehhe}\big((\psi_i \cup_{\scriptscriptstyle \otimes} \phi_j)\circ \tau_{i+j}\big)\pig)
  \\
&
\qquad\qquad\qquad\qquad
   =
  \begin{cases}
0 & \mbox{for} \ k < p,
\\
{[}\psi_{k-p},\phi_p]_{\cup_{\otimes}, \gs, \tau}
& \mbox{for} \ p \leq k \leq p+q,
  \end{cases}
  \end{split}
  \end{equation}
which is seen as follows:
first, in case $k < p$, one directly computes
\begin{small}
  \begin{equation*}
    \begin{array}{rcl}
      &&
      \displaystyle
      \textstyle\sum\limits_{j+i = k}
 (\gs|\tau)^{}_{\effe,\ehhe} \circ   \gvr^{}_{\ehhe,\effe}
      \pig((\phi_j \cup_{\scriptscriptstyle \otimes} \psi_i)(u^0, \ldots, u^{i+j}, m \otimes_A m')\pig) 
\\[1mm]
      &
      \stackrel{\scriptscriptstyle{\eqref{cupcup}}}{=}
      &
      \displaystyle
\textstyle\sum\limits_{j+i = k}    (\gs|\tau)^{}_{\effe,\ehhe} \circ     \gvr^{}_{\ehhe,\effe}
      \pig(
\phi_j\big(u^0_{(1)}, \ldots, u^j_{(1)}, (u^{j+1}_{(2)} \cdots u^{j+i}_{(2)} m')_{[0]}\big)
    \\
    & &
    \qquad\qquad
    \otimes_A \psi_i\big(u^0_{(2)} \cdots u^j_{(2)}
(u^{j+1}_{(2)} \cdots u^{j+i}_{(2)} m')_{[1]}
    , u^{j+1}_{(1)}, \ldots, u^{j+i}_{(1)}, m\big)
      \pig)
      \\[1mm]
      &
      \stackrel{\scriptscriptstyle{\eqref{hiddensee2}, \eqref{eq:almost-there}}}{=}
      &
 \tau \circ \big(   
\phi_k(u^0_{(1)}, \ldots, u^k_{(1)}, m'_{[0]})
     \otimes_A p^{}_\effe \psi_0(u^0_{(2)} \cdots u^k_{(2)} m'_{[1]},
 m)
      \big)
            \\[1mm]
      &
      \stackrel{\scriptscriptstyle{\eqref{fame1047}}}{=}
      &
 \tau \circ \big(   
\phi_k(u^0_{(1)}, \ldots, u^k_{(1)}, m'_{[0]})
     \otimes_A L(u^0_{(2)} \cdots u^k_{(2)} m'_{[1]},
 m)
      \big)
      \\[1mm]
      &
      \stackrel{\scriptscriptstyle{\eqref{leftaction}, \eqref{ydbraid2}}}{=}
      &
(u^0_{(2)} \cdots u^k_{(2)} m'_{[1]}m)_{[0]}
    \otimes_A
(u^0_{(2)} \cdots u^k_{(2)} m'_{[1]}m)_{[1]}
      \phi_k(u^0_{(1)}, \ldots, u^k_{(1)}, m'_{[0]}),
    \end{array}
   \end{equation*}
\end{small}
whereas
\begin{small}
  \begin{equation*}
    \begin{array}{rcl}
      &&
      \displaystyle
      \textstyle\sum\limits_{j+i = k}
      \gl_{\effe, \ehhe}
      \pig(\big((\psi_i \cup_{\scriptscriptstyle \otimes} \phi_j) \circ \tau_{i+j}\big)(u^0, \ldots, u^{i+j}, m \otimes_A m')\pig)
\\[1mm]
      &
      \stackrel{\scriptscriptstyle{\eqref{ydbraid2}}}{=}
      &
        \displaystyle
      \textstyle\sum\limits_{j+i = k}
      \gl_{\effe, \ehhe}
      \pig((\psi_i \cup_{\scriptscriptstyle \otimes} \phi_j)(u^0, \ldots, u^{i+j}, m'_{[0]} \otimes_A m'_{[1]}m)\pig)
\\[1mm]
      &
      \stackrel{\scriptscriptstyle{\eqref{cupcup}}}{=}
      &   
      \displaystyle
      \textstyle\sum\limits_{j+i = k}  
      \gl_{\effe, \ehhe}
      \pig(
\psi_i\big(u^0_{(1)}, \ldots, u^i_{(1)}, (u^{i+1}_{(2)} \cdots u^{i+j}_{(2)} m'_{[1]}m)_{[0]}\big)
    \\
    & &
    \qquad\qquad
    \otimes_A \phi_j\big(u^0_{(2)} \cdots u^i_{(2)}
(u^{i+1}_{(2)} \cdots u^{i+j}_{(2)} m'_{[1]}m)_{[1]}
    , u^{i+1}_{(1)}, \ldots, u^{i+j}_{(1)}, m'_{[0]}\big)
      \pig)
      \\[1mm]
      &
      \stackrel{\scriptscriptstyle{\eqref{hiddensee1}}}{=}
      &  
 p_\effe \psi_0\big(u^0_{(1)}, (u^{1}_{(2)} \cdots u^{k}_{(2)} m'_{[1]}m)_{[0]}\big)
 \\[1mm]
 && \qquad \qquad 
 \otimes_A \phi_k\big(u^0_{(2)} (u^{1}_{(2)} \cdots u^{k}_{(2)} m'_{[1]}m)_{[1]}
    , u^{1}_{(1)}, \ldots, u^{k}_{(1)}, m'_{[0]}\big)
         \\[1mm]
      &
      \stackrel{\scriptscriptstyle{\eqref{fame}, \eqref{leftaction}}}{=}
      &
 u^0_{(1)}(u^{1}_{(2)} \cdots u^{k}_{(2)} m'_{[1]}m)_{[0]}
 \\[1mm]
 && \qquad \qquad 
 \otimes_A \phi_k\big(u^0_{(2)} (u^{1}_{(2)} \cdots u^{k}_{(2)} m'_{[1]}m)_{[1]}
    , u^{1}_{(1)}, \ldots, u^{k}_{(1)}, m'_{[0]}\big)
      \\[1mm]
      &
      \stackrel{\scriptscriptstyle{\eqref{yd2}}}{=}
      &
 (u^{0}_{(2)} \cdots u^{k}_{(2)} m'_{[1]}m)_{[0]}
 \otimes_A (u^{0}_{(2)} \cdots u^{k}_{(2)} m'_{[1]}m)_{[1]} \phi_k\big(u^0_{(1)}, \ldots, u^{k}_{(1)}, m'_{[0]}\big),
    \end{array}
   \end{equation*}
\end{small}
where in the last step we also used the $U$-linearity of $\phi_k$; hence, the same expression as above and the claim follows. The case in which $p \leq k \leq p+q$ again directly follows from how the edge morphisms are defined, see \eqref{hiddensee1}--\eqref{hiddensee2}, once more.
Finally, one proves in an analogous way that
\begin{equation}
  \label{fierydawn4}
\begin{split}
  & \textstyle\sum\limits_{j+i = k}\pig( 
    \gl^{}_{\ehhe,\effe}\big((\phi_j \cup_{\scriptscriptstyle \otimes} \psi_i) \circ \tau_{j+i} \big)
    - 
(\gs|\tau)^{}_{\ehhe,\effe} \circ 
    \gvr^{}_{\effe,\ehhe}(
  \psi_i \cup_{\scriptscriptstyle \otimes} \phi_j) 
  \pig)
\\
&
\qquad\qquad\qquad\qquad
   =
\begin{cases}
0 & \mbox{for} \ k < q,
\\
{[}\phi_{k-q},\psi_q]_{\cup_{\otimes}, \gs, \tau}
& \mbox{for} \ q \leq k \leq p+q.
  \end{cases}
  \end{split}
  \end{equation}
Gathering Eqs.~\eqref{fierydawn1}--\eqref{fierydawn4} yields the two claims in this lemma.
  \end{proof}

\begin{lem}
  \label{lem:xi}
Let $\phi\colon \mathrm{Bar}(U,X) \to \E$ and
$\psi\colon \mathrm{Bar}(U,X) \to \F$ be as above, and let
$$
\xi_k(\phi,\psi)\colon \mathrm{Bar}_k(U,X \otimes_A X) \to \big( (\E  \otimes_A Z) \#(X \otimes_A \F) \big)_{k} 
  $$
be the map given by
\begin{equation}
  \label{finallyhere}
\xi_k(\phi,\psi)=
\begin{cases}
0 & \mbox{for} \ k < q
\\
{[}\phi_{k-q}, \psi_q]_{\cup_{\otimes}, \gs, \tau}
& \mbox{for} \ q \leq k \leq p+q-2
\\
{[}\phi_{p-1},\psi_q]_{\cup_{\otimes}, \gs, \tau}
&
\\
\qquad
+ (-1)^{p+q} (i_\ehhe \otimes_A Z) \circ  (\psi_q \, \bar\circ^{\scriptscriptstyle \otimes} \, \phi_p)
& \mbox{for} \ k = p+q-1
\\
0 & \mbox{for} \ k = p+q.
\end{cases}
\end{equation}
If $(X, Z)$ is a commuting pair,
then $\xi(\phi,\psi)$ is a homotopically trivial $0$-twisted cocycle
for $(\E  \otimes_A Z) \# (X \otimes_A \F)  $. A homotopy between $\xi(\phi,\psi)$ and zero is given by 
  $$
s_k(\phi,\psi)\colon \mathrm{Bar}_k(U,X \otimes_A X) \to \big((\E \otimes_A Z)\# (X \otimes_A \F)  \big)_{k+1} 
  $$
defined for any $k = -1, 0 , \ldots, p+q-1$ 
by 
\begin{equation}
  \label{museiincomune}
s_k(\phi,\psi)
=
\begin{cases}
0 & \mbox{for} \ k < q,
\\[.7mm]
(-1)^{q(k+1)}\phi_{k-q+1}\,\bar\circ^{\scriptscriptstyle \otimes}\, \psi_q & \mbox{for} \ q \leq k \leq p+q-2,
\\[.7mm]
(-1)^{p+q}[\psi_q,\phi_p]^{\bar\circ^{\scriptscriptstyle \otimes}}
& \mbox{for} \ k = p+q-1. 
\end{cases}
  \end{equation}
\end{lem}

\begin{proof}
The condition $\xi_{-1}=0$ is true by definition of $\xi$, so we only have to show that $\xi$ is a homotopically trivial morphism of complexes $\xi\colon \mathrm{Bar}(U,X \otimes_A X) \to (\E  \otimes_A Z) \#(X \otimes_A \F)  $, with homotopy given by $s(\phi,\psi)$. This means we need to show that $\xi_k = d \circ s_k + s_{k-1} \circ d = [d, s]_k$, where we denote, somewhat ambiguously,
  by $d$ not only the differential on $\mathrm{Bar}(U,X \otimes_A X)$, but also the differential of both extensions $\E$ and $\F$ as well as on $(\E \otimes_A Z) \#(X \otimes_A \F) $, analogously to diagrams \eqref{fame} and \eqref{schnee2}, respectively.

  Note first that 
\begin{equation}
  \label{infoesame}
d \circ (\phi_j \circ^{\scriptscriptstyle \otimes}_i \psi_q) =
\gd\phi_{j-1} \circ^{\scriptscriptstyle \otimes}_i \psi_q
\end{equation}
for any $i = 1, \ldots, j$,
as follows from the precise form of the differential $d$ on the spliced extension
$(\E \otimes_A Z)\#(X \otimes_A \F)  $
as in \eqref{schnee2}, along with
\eqref{blumare}.
Now, for $k<q$, one clearly has $[d,s(\phi,\psi)]_k=0$. 
 In case $k = q$, we have with Eq.~\eqref{infoesame}
$$
[d, s(\phi, \psi)]_q = d \circ s_q(\phi, \psi_q) = - d \circ (\phi_1 \, \bar\circ^{\scriptscriptstyle \otimes}\, \psi_q)
= - \gd \phi_{0} \circ^{\scriptscriptstyle \otimes}_1 \psi_q = [\phi_0,\psi_q]_{\cup_{\otimes}, \gs, \tau},
$$
where the last step follows from \eqref{sahnejoghurt1}.
Using \eqref{infoesame} again, for $q+1 \leq k \leq p+q-2$ along with the homotopy formula \eqref{sahnejoghurt},
one 
immediately obtains
\begin{equation*}
  \begin{split}
     [d,s(\phi,\psi)]_k &= d \circ s_k(\phi, \psi) + s_{k-1}(\phi, \psi) \circ d
     \\
     &=
     (-1)^{q(k+1)} d \circ (\phi_{k-q+1} \,\bar\circ^{\scriptscriptstyle \otimes} \psi_q) + (-1)^{qk} (\phi_{k-q} \,\bar\circ^{\scriptscriptstyle \otimes} \psi_q) \circ d
     \\
     &=
     (-1)^{q(k+1)} \gd \phi_{k-q} \,\bar\circ^{\scriptscriptstyle \otimes} \psi_q + (-1)^{qk} \gd (\phi_{k-q} \,\bar\circ^{\scriptscriptstyle \otimes} \psi_q)
=  [\phi_{k-q},\psi_q]_{\cup_{\otimes}, \gs, \tau},
  \end{split}
\end{equation*}
whereas, for $k=p+q-1$ we have, again with \eqref{infoesame} in the third step,
\begin{align*}
  [d,s(\phi,\psi)]_{p+q-1}&=  d \circ s_{p+q-1}(\phi, \psi) + s_{p+q-2}(\phi, \psi) \circ d
  \\
  &=
  (-1)^{p+q}d \circ [\psi_q,\phi_p]^{\bar\circ^{\scriptscriptstyle \otimes}}
  +(-1)^{pq}\gd(\phi_{p-1} \,\bar\circ^{\scriptscriptstyle \otimes}\, \psi_{q})
  \\
    &=
(-1)^{p+q}d \circ (\psi_q \,\bar\circ^{\scriptscriptstyle \otimes}\, \phi_p) 
  +(-1)^{pq}\big( (-1)^{q-1}\gd \phi_{p-1} \,\bar\circ^{\scriptscriptstyle \otimes}\, \psi_q
  - \gd(\phi_{p-1} \,\bar\circ^{\scriptscriptstyle \otimes}\, \psi_{q})
  \big)
\\
&=(-1)^{p+q}(i_\ehhe \otimes_A Z) \circ (\psi_q \,\bar\circ^{\scriptscriptstyle \otimes}\, \phi_p)
+ [\phi_{p-1},\psi_q]_{\cup_{\otimes}, \gs, \tau},
\end{align*}
where we interchangeably write $d = i_\ehhe \otimes_A Z$ in the highest degree $Z \otimes_A Z$ of the spliced extension $(\E \otimes_A Z)\#(X \otimes_A \F)  $. Finally, for $k = p+q$, we have
$$
s_{p+q-1}(\phi, \psi) \circ d = (-1)^{p+q}\gd [\psi_q,\phi_p]^{\bar\circ^{\scriptscriptstyle \otimes}} = 0,
$$
as $\gd$ is a derivation of $[\psi_q,\phi_p]^{\bar\circ^{\scriptscriptstyle \otimes}}$, which directly follows from \eqref{sahnejoghurt}, and since both $\phi_p$ and $\psi_q$ are cocycles.
To sum up, $\xi =[d,s]$ in each degree, as desired.
\end{proof}

We add one final assumption, {\em i.e.}, from now on not only the comonoid
$X$ 
in ${\mathscr{Z}}^r(\umod)$ is supposed to be braided cocommutative
but also the monoid
$Z$ {\em braided commutative} in ${\mathscr{Z}}^\ell(\umod)$. We can then formulate:

  \begin{lemma}
    \label{lem:krzysztof}
    Let the comonoid $X \in {\mathscr{Z}}^r(\umod)$ be braided cocommutative and let the monoid
$Z \in {\mathscr{Z}}^\ell(\umod)$ be braided commutative such that $(X,Z)$ constitutes a commuting pair. Furthermore, 
let $\xi,\eta\colon \mathrm{Bar}(U,X \otimes_A X) \to (\E  \otimes_A Z) \#(X \otimes_A \F)  $ be the $0$-twisted cocycle representatives  defined by \eqref{eq:eta} and in \eqref{finallyhere}, respectively. Then we have
\[
\Phi(\eta(\phi,\psi))=\Phi(\xi(\phi,\psi)).
\]
\end{lemma}
\begin{proof}
  This is proven by a degree-wise comparison of the images of $\xi(\phi,\psi)$ as in \eqref{finallyhere} resp.\ the explicit expression of $\eta(\phi,\psi)$ obtained in \eqref{stabiloboss} under the map $\Phi$ defined
in Eq.~\eqref{phi}.

  To start with, for the top degree component of the former we have
\begin{equation}
  \label{top}
  \Phi_{p+q}(\xi) =
  \mu  \circ \xi_{p+q} \circ (\id^{p+q+1}\otimes  \gD_\ikks )\colon \mathrm{Bar}_{p+q}(U,X)\to Z,
\end{equation}
which is simply the zero map as seen from \eqref{finallyhere}. In one degree less, that is, in degree $p+q-1$, with the help of \eqref{radicchio6a}, we read off
\begin{equation}
  \begin{split}
  \label{minestra}
\Phi_{p+q-1}(\xi) &=
  \overline{\big(0,\xi_{p+q-1} \circ (\id^{p+q}\otimes  \gD_\ikks )\big)}
  \\
  \quad
  &=
\overline{\big(0,
{[}\phi_{p-1},\psi_q]_{\cup_{\otimes}, \gs} \circ \gD_\ikks + (-1)^p (i_\ehhe \otimes_A Z) \circ  (\psi_q \, \bar\circ^{\scriptscriptstyle \otimes} \, \phi_p) \circ \gD_\ikks\big) },
\\
&=
\overline{\big(0,
  {[}\phi_{p-1},\psi_q]_{\cup_{\otimes}, \gs} \circ \gD_\ikks \big)}
+ (-1)^{p-1} \overline{ \big(  \mu  \circ  (\psi_q \, \bar\circ^{\scriptscriptstyle \otimes} \, \phi_p) \circ \gD_\ikks, 0 \big) },
  \end{split}
\end{equation}
where
 ${[}\phi_{p-1},\psi_q]_{\cup_{\otimes}, \gs}$ denotes the (simply) braided commutator in the sense of \eqref{gewitter}, that is, with no $\tau$ present anymore (or the identity on $X \otimes_A X$). Here, as before, 
$\overline{(\, , \,)}$ indicates the class
\begin{equation}
  \label{pushpush}
\overline{\big(-  \mu (z \otimes_A z'), 0\big)} = \overline{\big(0,(i_\ehhe \otimes_A Z)(z \otimes_A z')\big)} 
\end{equation}
for $z, z' \in Z$ in the quotient $Z\sqcup_{Z \otimes_A Z} (E_{p-1} \otimes_A Z)$.

On the other hand, the same considerations hold for
$\Phi\big(\eta(\phi,\psi)\big)$:
here,
from an analogous expression as the one in \eqref{top} with $\eta_{p+q}$ instead of
  $\xi_{p+q}$, it is clear
by means of \eqref{radicchio2b} and \eqref{radicchio6a} that the top degree is again the zero map. As the lower degrees of $\xi$ and $\eta$ are the same except for (the second summand in) degree $p+q-1$, we are left with comparing this one: 
\begin{equation*}
  \begin{split}
&\overline{\big(0,\eta_{p+q-1} \circ (\id^{p+q}\otimes  \gD_\ikks )\big)}
  \\
  \quad
  &=
\overline{\big(0,
{[}\phi_{p-1},\psi_q]_{\cup_{\otimes}, \gs} \circ \gD_\ikks + (-1)^{p} \gs \circ (Z \otimes_A i_\ehhe) \circ (\psi_q \,\bar\circ^{\scriptscriptstyle \otimes} \, \phi_p) \circ \gD_\ikks \big) },
\\
&=
\overline{\big(0,
  {[}\phi_{p-1},\psi_q]_{\cup_{\otimes}, \gs} \circ \gD_\ikks \big)}
+ (-1)^{p} \overline{ \big(0,  (i_\ehhe \otimes_A Z) \circ \gs  \circ (\psi_q \,\bar\circ^{\scriptscriptstyle \otimes} \, \phi_p) \circ \gD_\ikks \big)
  \big) }
\\
&=
\overline{\big(0,
  {[}\phi_{p-1},\psi_q]_{\cup_{\otimes}, \gs} \circ \gD_\ikks \big)}
+ (-1)^{p-1} \overline{ \big(  \mu  \circ \gs  \circ (\psi_q \,\bar\circ^{\scriptscriptstyle \otimes} \, \phi_p) \circ \gD_\ikks, 0\big) }
\\
&=
\overline{\big(0,
  {[}\phi_{p-1},\psi_q]_{\cup_{\otimes}, \gs} \circ \gD_\ikks \big)}
+ (-1)^{p-1} \overline{ \big(  \mu  \circ  (\psi_q \, \bar\circ^{\scriptscriptstyle \otimes} \, \phi_p) \circ \gD_\ikks, 0 \big) },
  \end{split}
\end{equation*}
by the identification \eqref{pushpush},
the naturality of $\sigma$, that is,
$
(i_\ehhe \otimes_A Z) \circ \gs = \gs \circ (Z \otimes_A i_\ehhe)
$
as in \eqref{hub}, and finally
the braided commutativity \eqref{radicchio2b} again, 
which hence coincides with the expression in \eqref{minestra}.
\end{proof}

  \begin{definition}
    \label{def:nepo1}
    Let $\mathbb{H}\in \ext^r_U(X\otimes_A X,Z\otimes_A Z)$ and let $\G=\mu\#\mathbb{H}\#\Delta_\ikks$.
    Define the graded linear map
$$
    \Psi_\bull\colon \textstyle
        {\bigoplus\limits_{k=-1}^{r-1}}\Hom(\mathrm{Bar}_k(U, X\otimes_A X),{H}_{k+1})\to \textstyle
        {\bigoplus\limits_{k=-1}^{r-1}}\Hom(\mathrm{Bar}_k(U,X) ,{G}_{k+1})
$$
as follows:
\begin{equation}
\label{psi}
\Psi_k(\nu) \coloneq
\begin{cases}
  (\nu_{-1} \circ \gD_\ikks,0 )
&
  \text{for}\quad k=-1,
\\
  \nu_{k} \circ (\id^{k+1}\otimes_\Aop  \gD_\ikks )
  & \text{for}\quad 0\leq k\leq r-3,
  \\
  \overline{(0,\nu_{r-2} \circ (\id^{r-1} \otimes_\Aop  \gD_\ikks ))} &
  \text{for}\quad k=r-2,
  \\
  \mu  \circ \nu_{r-1} \circ (\id^{r} \otimes_\Aop  \gD_\ikks) & \text{for}\quad k=r-1,
\end{cases}
\end{equation}

for any $\nu \in \textstyle\bigoplus_{k=-1}^{r-1}\Hom(\mathrm{Bar}_k(U, X\otimes_A X),{H}_{k+1})$.
\end{definition}

\begin{rem}
For $r\leq 3$, Eq.~\eqref{psi} is to be interpreted as in Remark \ref{rem:casibassi}.
\end{rem}

\begin{lemma}
  \label{lem:lukaspassion}
Under the same assumptions as in Lemma \ref{lem:krzysztof},
  consider two elements
  $
  \xi \in \textstyle\bigoplus_{k=-1}^r\!\Hom(\mathrm{Bar}_k(U, X\!\otimes_A \!X), {H}_{k})
  $
and
  $
  \nu \in  \textstyle\bigoplus_{k=-1}^{r-1}\!\Hom(\mathrm{Bar}_k(U, X\!\otimes_A \!X), {H}_{k+1})
  $
  subject to the property that $\xi =[d,\nu]$.
  %
%
  Then
  $
  \Phi(\xi)=[d,\Psi(\nu)].
  $
\end{lemma}

\begin{proof}
  We will only consider the case $r>3$ here and leave the adaptation for $r \leq 3$ to the reader.
  
By the hypothesis
$\xi =[d,\nu]$ we explicitly mean
$
\xi_k=d_{\akka}\circ \nu_k
+ \nu_{k-1}\circ d_\mathrm{Bar},
$
for any $k=0,\dots,r-1$, along
with
$
\xi_r= \nu_{r-1}\circ d_\mathrm{Bar} 
$
as well as 
$\xi_{-1} =d_{\akka}\circ \nu_{-1}$.
Hence, it is enough to show
$
\Psi_{k-1}(\nu)\circ d^{X}_{\mathrm{Bar}} =
\Phi_k(\nu \circ d^{X \otimes_A X}_{\mathrm{Bar}}) 
$
for $k = 0, \ldots, r$,
along with
$
d_{\ggii} \circ \Psi_{k}(\nu) =
\Phi_k(d_\akka \circ \nu)
$
for $k = -1, 0, \ldots, r-1$.

The first of these statements follows directly, by looking at the explicit definition of $\Phi$ and $\Psi$ in \eqref{phi} resp.\ \eqref{psi}, from the fact that $X$ is a comonoid in $\umod$, that is, from \eqref{radicchio4}, which along with the monoidal structure \eqref{rain} implies
\begin{equation}
  \label{lost}
(\id^{j} \otimes_\Aop \gD_\ikks)
\circ d^{X}_{\mathrm{Bar}} =
d^{X \otimes_A X}_{\mathrm{Bar}} \circ (\id^{j+1} \otimes_\Aop \gD_\ikks),
\end{equation}
for $j=0, \ldots, r$,
as already used in the proof of Lemma \ref{monteolivetomaggiore}. This proves the case for $k = 1, \ldots, r$, where for $k = r-1$ one uses that
$
\overline{(0,\nu_{r-2} \circ (\id^{r-1} \otimes_\Aop  \gD_\ikks ))} \circ d^{X}_{\mathrm{Bar}}
=
\overline{(0,\nu_{r-2} \circ (\id^{r-1} \otimes_\Aop  \gD_\ikks ) \circ d^{X}_{\mathrm{Bar}})};
$
Eq.~\eqref{lost} also already proves $\Phi_r(\xi) = [d, \Psi(\nu)]_r$ in top degree.
In case $k=0$,
\begin{equation*}
  \begin{split}
\Phi_0(\nu \circ d^{X \otimes_A X}_{\mathrm{Bar}}) 
&= \big(0 \circ d^{X \otimes_A X}_{\mathrm{Bar}}, \nu_{-1} \circ d^{X \otimes_A X}_{\mathrm{Bar}} \circ (\id \otimes \Delta_\ikks)\big)  
\\
&= (0, \nu_{-1} \circ \Delta_\ikks \circ d^{X}_{\mathrm{Bar}})
= \Psi_{-1}(\nu) \circ d^{X}_{\mathrm{Bar}}.
  \end{split}
  \end{equation*}
As for the second statement,
first consider the degrees $r-1$ and $r-2$: 
as in \eqref{qzero}, and similar to the proof of Lemma \ref{monteolivetomaggiore}, one has a commutative diagram
\[
\xymatrix@C=3.6em@R=1.7em{
   Z \otimes_A Z \ar[r]^{i_\akka} \ar[d]_\mu & {H}_{r-1}
   \ar[d]^{\overline{(0,\id)}}
\ar[r]^{d_\akka} & {H}_{r-2} \ar[d]^\id
  \\
  Z \ar[r]^-{\overline{(-\id,0)}} & Z\sqcup_{Z \otimes_A Z} {H}_{r-1} \ar[r]^-{ d_\akka \circ \pr_{\!2}} & H_{r-2}
  }
\]
and hence, for $\G = \mu \# \mathbb{H} \# \gD_\ikks$, 
\begin{equation*}
  \begin{split}
d_\ggii \circ \Psi_{r-1}(\nu)
&=
\overline{(-\id, 0)} \circ \big( \mu \circ \nu_{r-1} \circ (\id^r \otimes_\Aop \gD_\ikks)\big)
\\
&=
\overline{(0,\id)} \circ \big(i_\akka \circ \nu_{r-1} \circ (\id^r \otimes_\Aop \gD_\ikks)\big)
\\
&=
\overline{\big(0, (i_\akka \circ \nu_{r-1}) \circ (\id^r \otimes_\Aop \gD_\ikks)\big)}
\\
&=
\Phi_{r-1}(d_\akka \circ \nu),
  \end{split}
  \end{equation*}
considering that $d_\akka = i_\akka$ in highest degree. In degree $r-2$, we compute
\begin{equation*}
  \begin{split}
d_\ggii \circ \Psi_{r-2}(\nu)
&=
(d_\akka  \circ \pr_{\!2}) \circ \overline{\big(0, \nu_{r-2} \circ (\id^{r-1} \otimes_\Aop \gD_\ikks)\big)}
\\
&=
(d_\akka  \circ \nu_{r-2}) \circ (\id^{r-1} \otimes_\Aop \gD_\ikks)
\\
&=
\Phi_{r-2}(d_\akka \circ \nu).
  \end{split}
\end{equation*}
In degrees $1 \leq k \leq r-3$, there is nothing to prove.
As for the cases $k=-1$ and $k=0$,
recall from \eqref{pzero} that the sequence
$
G_1 \xrightarrow{d_\ggii} G_0 \xrightarrow{p_\ggii}  X 
$
is actually given by
$$
\xymatrix@C=3em{
H_1 \ar[r]^-{(d_\akka,0)} &  H_0\times_{X \otimes_A X} X  \ar[r]^-{\pr_{\!2}} & X, 
  }
$$
and therefore
\begin{equation*}
  \begin{split}
d_\ggii \circ \Psi_0(\nu)
&= (d_\akka,0) \circ \big( \nu_0 \circ (\id \otimes_\Aop \gD_\ikks)\big)
\\
&
= \big(d_\akka \circ \nu_0 \circ (\id \otimes_\Aop \gD_\ikks),0\big)
= \Phi_0(d_\akka \circ \nu),
  \end{split}
  \end{equation*}
which proves the case $k=0$. Finally, 
for 
$k=-1$, one has:
\begin{equation*}
  \begin{split}
p_\ggii \circ \Psi_{-1}(\nu)
&= \pr_{\!2} \circ (\nu_{-1} \circ \gD_\ikks,0)
= 0
= \Phi_{-1}(d_\akka \circ \nu),
  \end{split}
  \end{equation*}
which ends the proof.
\end{proof}

 As an immediate consequence we find:.

\begin{prop}
  Let $X, Z$ be as above and as before $\phi$ and $\psi$ be two cocycle representatives for
  $\E \in \ext^p_U(X, Z)$ and  $\F \in \ext^q_U(X, Z)$, respectively. Let
$$
  s_\bull(\phi,\psi)\in  \textstyle
  {\bigoplus\limits_{k=-1}^{p+q-1}} \Hom\pig(\mathrm{Bar}_k(U,X \otimes_A X), \big((\E \otimes_A Z)\# (X \otimes_A \F)  \big)_{k+1}\pig)
$$
be the map defined by Eq.~\eqref{museiincomune}.
Then 
\[
\Phi(\eta(\phi,\psi))=\big[d,\Psi(s(\phi,\psi))\big].
%
\]
\end{prop}
\begin{proof}
By Lemma \ref{lem:lukaspassion} and Lemma \ref{lem:xi}, we have $[d,\Psi(s(\phi,\psi))]=\Phi(\xi(\phi,\psi))$. The conclusion then immediately follows from Lemma \ref{lem:krzysztof}.
\end{proof}

  As motivated at the end of \S\ref{torloniasantisubito}, the following definition now makes sense:

\begin{definition}
  The {\em Gerstenhaber bracket} of two cocycle representatives $\phi$ and $\psi$ as above is, up to a sign, defined
  as the top degree component of the homotopy $\Psi_{p+q-1}(s(\phi,\psi))$.
  More precisely,
\begin{equation*}
\{\phi,\psi\} \coloneq (-1)^{pq} \Psi_{p+q-1}(s(\phi,\psi))
\end{equation*}
\end{definition}

This definition of the Gerstenhaber bracket can be completely made explicit:

\begin{cor}
One has
\begin{equation}
  \label{greifenhagener}
\{\phi_p,\psi_q\}
=
%
\phi_p \, \bar\circ \, \psi_q - (-1)^{(p-1)(q-1)} \psi_q
\, \bar\circ \, \phi_p, 
\end{equation}
where $\bar\circ$ is the full (internal) operadic composition defined in \eqref{museicapitolini}.
\end{cor}
\begin{proof}
Analogously to what we saw in the proof of Lemma \ref{monteolivetomaggiore},
the top degree $\Psi_{p+q-1}(s(\phi,\psi))$
is obtained by precomposing with $\id^{p+q} \otimes_\Aop  \Delta_\ikks$ and postcomposing with $\mu$, that is,
$\Psi_{p+q-1}(s(\phi,\psi)) =\mu\circ s_{p+q-1}(\phi,\psi)  \circ (\id^{p+q}\otimes_\Aop  \gD_\ikks)$. With the explicit form of $s$ given in \eqref{museiincomune} along with the external and internal operadic compositions of Eqs.~\eqref{soso2} and \eqref{mainsomma}, respectively, the claim is straightforward.
\end{proof}

\begin{rem}
Similarly to what was said in Remark \ref{winsstr}, Eq.~\eqref{greifenhagener} now makes sense even if one drops the assumption that $\phi_p$ and $\psi_q$ are cocycles. In fact, the internal operadic composition \eqref{museicapitolini} only asks for cochains. 
\end{rem}

 However, that this {\em is} a Gerstenhaber bracket indeed, that is, a graded Lie bracket fulfilling a graded Leibniz rule with respect to the (internal) cup product \eqref{cupcupcup}, will explicitly only follow from the results in the subsequent \S\ref{appassionata}.

\subsection{The $\Ext$ groups as a Gerstenhaber algebra}
\label{appassionata}
The explicit Gerstenhaber algebra structure on $\Ext_U^\bull(X,Z)$ is in a direct way obtained by adding a multiplication to the operad $\cO=\Hom_U(\mathrm{Bar}_\bullet(U, X), Z)$ from Section \ref{sec:operadic}, along with an identity and a unit. Define $ \mu  \in \cO(2), \mathbb{1} \in \cO(1)$, and $e \in \cO(0)$ as
\begin{equation}
  \begin{array}{rcl}
  \label{vibrant}
   \mu (u^0, u^1, u^2, m) &\coloneq& \gve_\ikks(u^0u^1u^2m) \lact 1_\zett,
  \\
  \mathbb{1}(u^0, u^1, m) &\coloneq& \gve_\ikks(u^0u^1m) \lact 1_\zett,
  \\
  e(u^0, m) &\coloneq& \gve_\ikks(u^0m) \lact 1_\zett.
  \end{array}
  \end{equation}
The $U$-linearity of these maps follows directly from
Eqs.~\eqref{radicchio4a} and \eqref{radicchio1a}, and it is a straightforward check that
$ \mu  \circ_1  \mu  =  \mu  \circ_2  \mu $ as well as $ \mu  \circ_1 e = \mathbb{1} =  \mu  \circ_2 e$. We then have:

\begin{theorem}
  \label{extobroid}
  Let $(U,A)$ be a left bialgebroid, $Z \in {\mathscr{Z}}^\ell(\umod)$ a braided commutative monoid and $X\in  {\mathscr{Z}}^r(\umod)$ a braided cocommutative comonoid such that $(X,Z)$ constitutes a commuting pair in the sense of Definition \ref{commpair}. 
  Then
$(\cO,  \mu , \mathbb{1}, e)$ defines an operad with multiplication, 
and its cohomology with respect to the differential $\gd\phi\coloneq (-1)^{\phi +1} \{ \mu , \phi\}$, which coincides with $\Ext^\bull_U(X,Z)$ if
$U_\ract$ is projective over $A$,
  becomes a Gerstenhaber algebra.
  \end{theorem}

\begin{proof}
  The proof is a cumbersome but straightforward explicit computation that consists in verifying the associativity \eqref{danton} for the partial compositions on $\cO$ defined in \eqref{mainsomma}.

  Let $\phi_p \in \cO(p)$, $\psi_q \in \cO(q)$,
and $\chi_r \in \cO(r)$.  We provide a detailed proof for the case $j < i$ only, the other cases being completely analogous and left to the reader. We have to show that
\begin{equation}
  \label{thisequaltothat}
\begin{split}
  & ((\phi_p  \circ_i \psi_q) \circ_j \chi_r)(u^0,\dots, u^{p+q+r-2}, m)
\\
  = \quad & ((\phi_p  \circ_j \chi_r) \circ_{i+r-1} \psi_q)(u^0,\dots, u^{p+q+r-2}, m).
\end{split}
\end{equation}
For better readability in later computations, let us introduce the following shorthand notation:
  \[ 
u^{[i,j]} \coloneqq (u^i,\dots,u^j), \qquad u^{(i,j)} \coloneqq u^i \cdots u^j,
\]
for $i \leq j$. 
We will be dealing with the following intervals:
\[
\begin{aligned}
I_1 &\coloneqq[0,j-1]
\\
I_2 &\coloneqq[j,j+r-1]
\\
I_3 &\coloneqq [j+r,i+r-2]
\\
I_4 &\coloneqq [i+r-1,i+q+r-2]
\\
I_5 &\coloneqq [i+q+r-1, p+q+r-2].
\end{aligned}
\]
To keep notation even shorter, let us moreover set
 \[ 
u^{I_1} \coloneqq u^{[0,j-1]}, \qquad u^{(I_1)} \coloneqq u^{(0,j-1)}, 
\]
and similar for the intervals $I_2$ to $I_5$. 
%
Finally, 
we will write $I_{i,j}$ for $I_i\cup I_j$, and so on.
As a first step, we fully expand the left hand side of \eqref{thisequaltothat}:
 \begin{align*}
  &((\phi_p  \circ_i \psi_q) \circ_j \chi_r)(u^0, u^1, \ldots, u^{p+q+r-2}, m)
  \\[2mm]
&=
      \big(\phi_p \circ_i \psi_q\big)\pig(u^{I_1}_{(1)}, \chi_r\big(1, u^{I_2}_{(1)},    (u^{(I_{3,4,5})}_{(2)} m_{(2)})_{[0]} \big)_{(-1)}
      u^{(I_2)}_{(2)} (u^{(I_{3,4,5})}_{(2)} m_{(2)})_{[1]},
    u^{I_{3,4,5}}_{(1)} , m_{(1)} \pig)
        \\
&\qquad
    \cdot_\zett \pig( u^{(I_1)}_{(2)} 
     \chi_r\big(1, u^{I_2}_{(1)},    (u^{I_{3,4,5}}_{(2)} m_{(2)})_{[0]} \big)_{(0)}
     \pig)
       \\[2mm]
  &
  =    \phi_p\pig(u^{I_1}_{(1)}, \chi_r\big(1, u^{I_2}_{(1)},
    (u^{(I_{3,4,5})}_{(3)}  m_{(3)})_{[0]} \big)_{(-2)}
    u^{(I_2)}_{(2)} (u^{(I_{3,4,5})}_{(3)} m_{(3)})_{[1]},
    u^{I_3}_{(1)}, 
     \\
&\qquad
    \psi_q\big(1, u^{I_4}_{(1)},
    (u^{(I_5)}_{(2)}  m_{(2)})_{[0]} \big)_{(-1)}
    u^{(I_4)}_{(2)}  (u^{(I_5)}_{(2)}  m_{(2)})_{[1]},
     u^{I_5}_{(1)}, m_{(1)} \pig)
   \\
    &\qquad
   \cdot_\zett
   \pig( u^{(I_1)}_{(2)} 
\chi_r\big(1, u^{I_2}_{(1)}, 
(u^{(I_{3,4,5})}_{(3)}  m_{(3)})_{[0]} \big)_{(-1)}
     u^{(I_2)}_{(3)} (u^{(I_{3,4,5})}_{(3)}  m_{(3)})_{[2]}
     u^{(I_3)}_{(2)} 
   \psi_q\big(1, u^{I_4}_{(1)},
      (u^{(I_5)}_{(2)} m_{(2)})_{[0]} \big)_{(0)}
        \pig)
  \\
&\qquad
  \cdot_\zett
  \pig( u^{(I_1)}_{(3)}
     \chi_r\big(1, u^{I_2}_{(1)}, (u^{(I_{3,4,5})}_{(3)} m_{(3)})_{[0]} \big)_{(0)}
     \pig).
 \end{align*}    
 Now, we manipulate this expression using the identity
 \begin{equation*}
   \begin{split}
   &
   \phi_p\big(u_{(1)},  
  z_{(-2)} v^1, \ldots, v^{p}, m\big) \cdot_\zett (u_{(2)} z_{(-1)} z') \cdot_\zett (u_{(3)} z_{(0)})
 \\
 = \quad &
 \phi_p\big(u_{(1)},  
   z_{(-1)} v^1, \ldots, v^{p}, m\big) \cdot_\zett (u_{(2)} z_{(0)}) \cdot_\zett (u_{(3)} z')
   \end{split}
   \end{equation*}
 for any $u, v^0, \ldots, v^{p-1} \in U$, $m \in X$, and $z, z' \in Z$,
 that follows as an immediate
 consequence of  \eqref{radicchio3a}, 
to obtain:

\begin{align*}
 &\phi_p\pig(\overbracket{u^{I_1}_{(1)}}^{u_{(1)}}, 
\overbracket{\chi_r\big(1, u^{I_2}_{(1)},
    (u^{(I_{3,4,5})}_{(3)}  m_{(3)})_{[0]} \big)_{(-2)}}^{z_{(-2)}}
    u^{(I_2)}_{(2)} (u^{(I_{3,4,5})}_{(3)} m_{(3)})_{[1]},
    u^{I_3}_{(1)}, 
     \\
&\qquad
    \psi_q\big(1, u^{I_4}_{(1)},
    (u^{(I_5)}_{(2)}  m_{(2)})_{[0]} \big)_{(-1)}
    u^{(I_4)}_{(2)}  (u^{(I_5)}_{(2)}  m_{(2)})_{[1]},
     u^{I_5}_{(1)}, m_{(1)} \pig)
   \\
    &\qquad
    \cdot_\zett \pig( \overbracket{u^{(I_1)}_{(2)}}^{u_{(2)}} 
\overbracket{\chi_r\big(1, u^{I_2}_{(1)}, 
(u^{(I_{3,4,5})}_{(3)}  m_{(3)})_{[0]} \big)_{(-1)}}^{z_{(-1)}}
     \, \overbracket{u^{(I_2)}_{(3)} (u^{(I_{3,4,5})}_{(3)}  m_{(3)})_{[2]}
     u^{(I_3)}_{(2)} 
     \psi_q\big(1, u^{I_4}_{(1)},
      (u^{(I_5)}_{(2)} m_{(2)})_{[0]} \big)_{(0)}}^{z'}
        \pig)
  \\
&\qquad
        \cdot_\zett \pig( \overbracket{u^{(I_1)}_{(3)}}^{u_{(3)}}
     \overbracket{\chi_r\big(1, u^{I_2}_{(1)}, (u^{(I_{3,4,5})}_{(3)} m_{(3)})_{[0]} \big)_{(0)}}^{z_{(0)}}
     \pig)
     \\[2mm]
     = \ &
       \phi_p\pig(\overbracket{u^{I_1}_{(1)}}^{u_{(1)}}, 
\overbracket{\chi_r\big(1, u^{I_2}_{(1)},
    (u^{(I_{3,4,5})}_{(3)}  m_{(3)})_{[0]} \big)_{(-1)}}^{z_{(-1)}}
    u^{(I_2)}_{(2)} (u^{(I_{3,4,5})}_{(3)} m_{(3)})_{[1]},
    u^{I_3}_{(1)}, 
     \\
&\qquad
    \psi_q\big(1, u^{I_4}_{(1)},
    (u^{(I_5)}_{(2)}  m_{(2)})_{[0]} \big)_{(-1)}
    u^{(I_4)}_{(2)}  (u^{(I_5)}_{(2)}  m_{(2)})_{[1]},
     u^{I_5}_{(1)}, m_{(1)} \pig)
   \\
    &\qquad
   \cdot_\zett \pig( \overbracket{u^{(I_1)}_{(2)}}^{u_{(2)}} 
\overbracket{\chi_r\big(1, u^{I_2}_{(1)}, 
(u^{(I_{3,4,5})}_{(3)}  m_{(3)})_{[0]} \big)_{(0)}}^{z_{(0)}}\pig)    
  \\
&\qquad
       \cdot_\zett \pig( \overbracket{u^{(I_1)}_{(3)}}^{u_{(3)}}\, \overbracket{u^{(I_2)}_{(3)} (u^{(I_{3,4,5})}_{(3)}  m_{(3)})_{[2]}
     u^{(I_3)}_{(2)} 
     \psi_q\big(1, u^{I_4}_{(1)},
      (u^{(I_5)}_{(2)} m_{(2)})_{[0]} \big)_{(0)}}^{z'}
     \pig). 
 \end{align*} 
Next, we use that the the YD condition \eqref{yd2} implies 
\begin{equation}\label{eq:dallapiccola}
  \begin{split}
&  (u_{(2)} m)_{[0]} \otimes_A (u_{(2)} m)_{[1]} \otimes_A (u_{(2)} m)_{[2]} u_{(1)}
\\
  = \quad & 
  (u_{(1)} m_{[0]})_{[0]} \otimes_A (u_{(1)} m_{[0]})_{[1]} \otimes_A u_{(2)} m_{[1]} 
\end{split}
  \end{equation}
for any $u \in U$ and $m \in X$, to obtain: 
\begin{align*}
      & \phi_p\pig(u^{I_1}_{(1)}, 
\chi_r\big(1, u^{I_2}_{(1)},
    \overbracket{(u^{(I_3)}_{(3)} u^{(I_{4,5})}_{(3)}  
    m_{(3)})_{[0]}}^{(u_{(2)}m)_{[0]}} \big)_{(-1)}
     u^{(I_2)}_{(2)} \overbracket{(u^{(I_{3})}_{(3)} u^{(I_{4,5})}_{(3)} m_{(3)})_{[1]}}^{(u_{(2)}m)_{[1]}},
    u^{I_3}_{(1)}, 
     \\
&\qquad
    \psi_q\big(1, u^{I_4}_{(1)},
    (u^{(I_5)}_{(2)}  m_{(2)})_{[0]} \big)_{(-1)}
    u^{(I_4)}_{(2)}  (u^{(I_5)}_{(2)}  m_{(2)})_{[1]},
     u^{I_5}_{(1)}, m_{(1)} \pig)
   \\
    &\qquad
    \cdot_\zett \pig( u^{(I_1)}_{(2)} 
\chi_r\big(1, u^{I_2}_{(1)}, 
\overbracket{(u^{(I_{3})}_{(3)} u^{(I_{4,5})}_{(3)}  m_{(3)})_{[0]}}^{(u_{(2)} m)_{[0]}} \big)_{(0)}\pig)     
  \\
&\qquad
        \cdot_\zett \pig( u^{(I_{1,2})}_{(3)} \overbracket{(u^{(I_{3})}_{(3)}u^{(I_{4,5})}_{(3)}  m_{(3)})_{[2]}
     u^{(I_3)}_{(2)}}^{(u_{(2)}m)_{[2]}u_{(1)}} 
     \psi_q\big(1, u^{I_4}_{(1)},
      (u^{(I_5)}_{(2)} m_{(2)})_{[0]} \big)_{(0)}
     \pig)\\
   = \ & 
     \phi_p\Big(u^{I_1}_{(1)}, \chi_r\pig(1, u^{I_2}_{(1)},
     \overbracket{ \big(
    u^{(I_3)}_{(2)}  (u^{(I_{4,5})}_{(3)} m_{(3)})_{[0]}\big)_{[0]}}^{(u_{(1)}m_{[0]})_{[0]}}
     \pig)_{(-1)}  u^{(I_2)}_{(2)} 
\overbracket{\big(u^{(I_3)}_{(2)} (u^{(I_{4,5})}_{(3)} m_{(3)})_{[0]}\big)_{[1]}}^{(u_{(1)}m_{[0]})_{[1]}}, u^{I_3}_{(1)}, 
    \\
    &\qquad
    \psi_q\big(1, u^{I_4}_{(1)},
    (u^{(I_5)}_{(2)}  m_{(2)})_{[0]} \big)_{(-1)} 
    u^{(I_4)}_{(2)}  (u^{(I_5)}_{(2)}  m_{(2)})_{[1]},
     u^{I_5}_{(1)},  m_{(1)} \Big)
  \\
    &\qquad
  \cdot_\zett
  \Big( u^{(I_1)}_{(2)}
        \chi_r\pig(1, u^{I_2}_{(1)},\overbracket{\big(u^{(I_3)}_{(2)} (u^{(I_{4,5})}_{(3)} m_{(3)})_{[0]}\big)_{[0]}}^{(u_{(1)}m_{[0]})_{[0]}}
        \pig)_{(0)}
     \Big)
    \\
    &\qquad
    \cdot_\zett \pig( u^{(I_{1,2})}_{(3)} \overbracket{u^{(I_{3})}_{(3)}
        (u^{(I_{4,5})}_{(3)}  m_{(3)})_{[1]}}^{u_{(2)}m_{[1]}}
    \psi_q\big(1, u^{I_4}_{(1)},
      (u^{(I_5)}_{(2)} m_{(2)})_{[0]} \big)_{(0)}
    \pig).     
 \end{align*} 
Finally, we use that for all
$x \in X$ and
$z \in Z$ one has
\begin{equation*}
 x_{[0]}  \otimes_A x_{[1]}z = 
  z_{(-1)} x \otimes_A z_{(0)} 
\end{equation*}
in $X \otimes_A  Z$,
as follows from the explicit expressions for the braidings \eqref{ydbraid2} and \eqref{ydbraid}, respectively. This implies
\begin{equation*}
 z_{(-1)}\otimes_A  x_{[0]}  \otimes_A x_{[1]}z_{(0)} = 
  z_{(-2)}\otimes_A z_{(-1)} x \otimes_A z_{(0)}
\end{equation*}
as elements in $U \otimes_A  X \otimes_A  Z$,
and so we obtain:
\begin{align*}
     &\phi_p\Big(u^{I_1}_{(1)}, \chi_r\pig(1, u^{I_2}_{(1)},
    \big(
     u^{(I_3)}_{(2)}
       \overbracket{(u^{(I_{4,5})}_{(3)} m_{(3)})_{[0]}}^{x_{[0]}}\big)_{[0]}
       \pig)_{(-1)}  u^{(I_2)}_{(2)}
\big(u^{(I_3)}_{(2)}
  \overbracket{(u^{(I_{4,5})}_{(3)} m_{(3)})_{[0]}}^{x_{[0]}}\big)_{[1]}, 
    u^{I_3}_{(1)}, 
\\
  &\qquad
    {\overbracket{\psi_q\big(1, u^{I_4}_{(1)},
    (u^{(I_5)}_{(2)}  m_{(2)})_{[0]} \big)_{(-1)}}^{z_{(-1)}}} 
    u^{(I_4)}_{(2)}  (u^{(I_5)}_{(2)}  m_{(2)})_{[1]},
     u^{I_5}_{(1)},  m_{(1)} \Big)
  \\
    &\qquad
  \cdot_\zett
  \Big( u^{(I_1)}_{(2)}
        \chi_r\pig(1, u^{I_2}_{(1)},\big(u^{(I_3)}_{(2)} \overbracket{(u^{(I_{4,5})}_{(3)} m_{(3)})_{[0]}}^{x_{[0]}}\big)_{[0]}
        \pig)_{(0)}
     \Big)
    \\
    &\qquad
    \cdot_\zett \pig( u^{(I_{1,2,3})}_{(3)} 
        \overbracket{(u^{(I_{4,5})}_{(3)}  m_{(3)})_{[1]}}^{x_{[1]}}
    {\overbracket{\psi_q\big(1, u^{I_4}_{(1)},
      (u^{(I_5)}_{(2)} m_{(2)})_{[0]} \big)_{(0)}}^{z_{(0)}}}
    \pig)    
    \\[2mm]
    = \ &
         \phi_p\pig(u^{I_1}_{(1)}, \chi_r\big(1, u^{I_2}_{(1)},     \big(u^{(I_3)}_{(2)} 
 \overbracket{\psi_q\big(1, u^{I_4}_{(1)},  (u^{(I_5)}_{(2)} m_{(2)})_{[0]} \big)_{(-1)}}^{z_{(-1)}}
     \overbracket{u^{(I_{4,5})}_{(3)} m_{(3)}}^{x}\big)_{[0]}
     \big)_{(-1)}
    u^{(I_2)}_{(2)}\cdot
    \\
    &\qquad 
    \big(u^{(I_3)}_{(2)}  \overbracket{\psi_q\big(1, u^{I_4}_{(1)},
(u^{(I_5)}_{(2)}  m_{(2)})_{[0]} \big)_{(-1)}}^{z_{(-1)}}
\overbracket{u^{(I_{4,5})}_{(3)}  m_{(3)}}^{x}\big)_{[1]},
    \\
&\qquad
u^{I_3}_{(1)}, 
     \overbracket{\psi_q\big(1, u^{I_4}_{(1)},
    (u^{(I_5)}_{(2)}  m_{(2)})_{[0]} \big)_{(-2)}}^{z_{(-2)}}
    u^{(I_4)}_{(2)}  (u^{(I_5)}_{(2)}  m_{(2)})_{[1]},
     u^{I_5}_{(1)}, m_{(1)} \pig)
  \\
    &\qquad
        \cdot_\zett \pig( u^{(I_1)}_{(2)}
\chi_r\big(1, u^{I_2}_{(1)},
     \big(u^{(I_3)}_{(2)} 
 \overbracket{\psi_q\big(1, u^{I_4}_{(1)},
 (u^{(I_5)}_{(2)}  m_{(2)})_{[0]} \big)_{(-1)}}^{z_{(-1)}}
     \overbracket{u^{(I_{4,5})}_{(3)}  m_{(3)}}^{x}\big)_{[0]}
     \big)_{(0)}
     \pig)
    \\
&\qquad
    \cdot_\zett \pig( u^{(I_{1,2,3})}_{(3)} 
    \overbracket{\psi_q\big(1, u^{I_4}_{(1)},
      (u^{(I_5)}_{(2)}  m_{(2)})_{[0]} \big)_{(0)}}^{z_{(0)}}
    \pig). 
 \end{align*} 
For the right hand side of \eqref{thisequaltothat} we perform a similar computation. First, let us expand
\begin{align*}
  &((\phi_p  \circ_j \chi_r) \circ_{i+r-1} \psi_q)(u^0, u^1, \ldots, u^{p+q+r-2}, m)
  \\[2mm]
  & =
          \big(\phi_p \circ_j \chi_r\big)\pig(u^{I_{1,2,3}}_{(1)}, \psi_q\big(1, u^{I_4}_{(1)},    (u^{(I_5)}_{(2)}  m_{(2)})_{[0]} \big)_{(-1)}
      u^{(I_4)}_{(2)}  (u^{(I_5)}_{(2)}  m_{(2)})_{[1]},
    u^{I_5}_{(1)} , m_{(1)} \pig)
        \\
&
 \qquad    \cdot_\zett \pig( u^{(I_{1,2,3})}_{(2)} 
     \psi_q\big(1, u^{I_4}_{(1)}, (u^{(I_5)}_{(2)}  m_{(2)})_{[0]} \big)_{(0)}
     \pig)
       \\[2mm]
  &=
  \phi_p\pig(u^{I_1}_{(1)},
  \chi_r\pig(1,  u^{I_2}_{(1)},
  \big( u^{(I_3)}_{(2)}  \psi_q\big(1, u^{I_4}_{(1)},
  (u^{(I_5)}_{(3)}  m_{(3)})_{[0]} \big)_{(-1)}
  u^{(I_4)}_{(3)} (u^{(I_5)}_{(3)}  m_{(3)})_{[2]}u^{(I_5)}_{(2)}  m_{(2)}
 \big)_{[0]}
  \pig)_{(-1)}
  \\
  &\qquad
   u^{(I_2)}_{(2)}
 \big(u^{(I_3)}_{(2)} \psi_q\big(1, u^{I_4}_{(1)},
 (u^{(I_5)}_{(3)}  m_{(3)})_{[0]} \big)_{(-1)}
 u^{(I_4)}_{(3)} (u^{(I_5)}_{(3)}  m_{(3)})_{[2]}
 u^{(I_5)}_{(2)} 
 m_{(2)}\big)_{[1]},
 \\
 &\qquad
 u^{I_3}_{(1)},
 \psi_q\big(1, u^{I_4}_{(1)},
 (u^{(I_5)}_{(3)}  m_{(3)})_{[0]} \big)_{(-2)}
 u^{(I_4)}_{(2)} (u^{(I_5)}_{(3)}  m_{(3)})_{[1]},
 u^{I_5}_{(1)}, m_{(1)}\pig)\\
 &\qquad
  \cdot_\zett
  \pig(
 u^{(I_1)}_{(2)} 
  \chi_r\pig(1,  u^{I_2}_{(1)},
  \big( u^{(I_3)}_{(2)}  \psi_q\big(1, u^{I_4}_{(1)},
   (u^{(I_5)}_{(3)}  m_{(3)})_{[0]} \big)_{(-1)}
  u^{(I_4)}_{(3)} (u^{(I_5)}_{(3)}  m_{(3)})_{[2]}
  u^{(I_5)}_{(2)}  m_{(2)}
 \big)_{[0]} 
  \pig)_{(0)}
  \pig)
  \\
  &\qquad
 \cdot_\zett \pig( u^{(I_{1,2,3})}_{(3)} 
  \psi_q\big(1, u^{I_4}_{(1)},    (u^{(I_5)}_{(3)}  m_{(3)})_{[0]} \big)_{(0)}
     \pig).
\end{align*}
Using \eqref{eq:dallapiccola} again, we have

\begin{align*}
&  \phi_p\pig(u^{I_1}_{(1)},
  \chi_r\pig(1,  u^{I_2}_{(1)},
  \big( u^{(I_3)}_{(2)}  \psi_q\big(1, u^{I_4}_{(1)},
  \overbracket{(u^{(I_5)}_{(3)}m_{(3)})_{[0]}}^{(u_{(2)}m)_{[0]}} \big)_{(-1)}
  u^{(I_4)}_{(3)} \overbracket{(u^{(I_5)}_{(3)}  m_{(3)})_{[2]}u^{(I_5)}_{(2)}}^{(u_{(2)}m)_{[2]}u_{(1)}}  m_{(2)}
 \big)_{[0]}
  \pig)_{(-1)}
  \\
  &\qquad
   u^{(I_2)}_{(2)}
 \big(u^{(I_3)}_{(2)} \psi_q\big(1, u^{I_4}_{(1)},
 \overbracket{(u^{(I_5)}_{(3)} m_{(3)})_{[0]}}^{(u_{(2)}m)_{[0]}} \big)_{(-1)}
 u^{(I_4)}_{(3)}  \overbracket{(u^{(I_5)}_{(3)} m_{(3)})_{[2]}
  u^{(I_5)}_{(2)}}^{(u_{(2)}m)_{[2]}u_{(1)}}
 m_{(2)}\big)_{[1]}, u^{I_3}_{(1)},
 \\
 &\qquad
 \psi_q\big(1, u^{I_4}_{(1)},
 \overbracket{(u^{(I_5)}_{(3)} m_{(3)})_{[0]}}^{(u_{(2)}m)_{[0]}} \big)_{(-2)}
 u^{(I_4)}_{(2)} \overbracket{(u^{(I_5)}_{(3)} m_{(3)})_{[1]}}^{(u_{(2)}m)_{[1]}},
 u^{I_5}_{(1)}, m_{(1)}\pig)\\
 &\qquad
  \cdot_\zett
  \pig(
 u^{(I_1)}_{(2)} 
  \chi_r\pig(1,  u^{I_2}_{(1)},
  \big( u^{(I_3)}_{(2)}  \psi_q\big(1, u^{I_4}_{(1)},
   \overbracket{(u^{(I_5)}_{(3)} m_{(3)})_{[0]} }^{(u_{(2)}m)_{[0]}} \big)_{(-1)}
  u^{(I_4)}_{(3)} \overbracket{(u^{(I_5)}_{(3)} m_{(3)})_{[2]}
  u^{(I_5)}_{(2)}}^{(u_{(2)}m)_{[2]}u_{(1)}}  m_{(2)}
 \big)_{[0]} 
  \pig)_{(0)}
  \pig)
  \\
  &\qquad
 \cdot_\zett \pig( u^{(I_{1,2,3})}_{(3)} 
  \psi_q\big(1, u^{I_4}_{(1)},    \overbracket{(u^{(I_5)}_{(3)} m_{(3)})_{[0]} }^{(u_{(2)}m)_{[0]}} \big)_{(0)}
     \pig)
       \\[2mm]
  = \ &
   \phi_p\pig(u^{I_1}_{(1)},
   \chi_r\pig(1,  u^{I_2}_{(1)},
  \big( u^{(I_3)}_{(2)}  \psi_q\big(1, u^{I_4}_{(1)},
  \overbracket{(u^{(I_5)}_{(2)} m_{(3)[0]})_{[0]}}^{(u_{(1)}m_{[0]})_{[0]}} \big)_{(-1)}
  u^{(I_{4})}_{(3)}  \overbracket{u^{(I_{5})}_{(3)}m_{(3)[1]}}^{u_{(2)}m_{[1]}} m_{(2)}
 \big)_{[0]}
  \pig)_{(-1)}
  \\
  &\qquad
  u^{(I_2)}_{(2)}
 \big(u^{(I_3)}_{(2)}  \psi_q\big(1, u^{I_4}_{(1)},
 \overbracket{(u^{(I_5)}_{(2)} m_{(3)[0]})_{[0]}}^{(u_{(1)}m_{[0]})_{[0]}} \big)_{(-1)}
 u^{(I_{4})}_{(3)} \overbracket{u^{(I_{5})}_{(3)}m_{(3)[1]}}^{u_{(2)}m_{[1]}} m_{(2)}\big)_{[1]}, u^{I_3}_{(1)},
 \\
 &\qquad
 \psi_q\big(1, u^{I_4}_{(1)},
 \overbracket{(u^{(I_5)}_{(2)} m_{(3)[0]})_{[0]}}^{(u_{(1)}m_{[0]})_{[0]}} \big)_{(-2)}
 u^{(I_4)}_{(2)} \overbracket{(u^{(I_5)}_{(2)}m_{(3)[0]})_{[1]}}^{(u_{(1)}m_{[0]})_{[1]}},
 u^{I_5}_{(1)}, m_{(1)}\pig)
 \\
 &\qquad
 \cdot_\zett
  \pig(
 u^{(I_1)}_{(2)} 
 \chi_r\pig(1,  u^{I_2}_{(1)},
 \big( u^{(I_3)}_{(2)}  \psi_q\big(1, u^{I_4}_{(1)},
  \overbracket{(u^{(I_5)}_{(2)} m_{(3)[0]})_{[0]} }^{(u_{(1)}m_{[0]})_{[0]}} \big)_{(-1)}
  u^{(I_{4})}_{(3)} \overbracket{u^{(I_{5})}_{(3)}m_{(3)[1]}}^{u_{(2)}m_{[1]}} m_{(2)}
 \big)_{[0]}
  \pig)_{(0)}
  \pig)
  \\
  &\qquad
 \cdot_\zett \pig( u^{(I_{1,2,3})}_{(3)} 
  \psi_q\big(1, u^{I_4}_{(1)},    \overbracket{(u^{(I_5)}_{(2)} m_{(3)[0]})_{[0]}}^{(u_{(1)}m_{[0]})_{[0]}} \big)_{(0)}
     \pig).
  \end{align*}
Finally, using
$$
x_{(2)[0]} \otimes_A x_{(2)[1]} x_{(1)} = x_{(1)} \otimes_A x_{(2)},
$$
that is, the braided cocommutativity of $X$
from \eqref{radicchio6}, 
 we obtain:
\begin{align*}
&   \phi_p\pig(u^{I_1}_{(1)},
   \chi_r\pig(1,  u^{I_2}_{(1)},
  \big( u^{(I_3)}_{(2)}  \psi_q\big(1, u^{I_4}_{(1)},
  (u^{(I_5)}_{(2)}  \overbracket{m_{(3)[0]}}^{x_{(2)[0]}})_{[0]} \big)_{(-1)}
  u^{(I_{4,5})}_{(3)}  \overbracket{m_{(3)[1]} m_{(2)}}^{x_{(2)[1]}x_{(1)}}
 \big)_{[0]}
  \pig)_{(-1)}
  \\
  &\qquad
  u^{(I_2)}_{(2)}
 \big(u^{(I_3)}_{(2)}  \psi_q\big(1, u^{I_4}_{(1)},
 (u^{(I_5)}_{(2)}  \overbracket{m_{(3)[0]}}^{x_{(2)[0]}})_{[0]} \big)_{(-1)}
 u^{(I_{4,5})}_{(3)} \overbracket{m_{(3)[1]} m_{(2)}}^{x_{(2)[1]}x_{(1)}}\big)_{[1]}, u^{I_3}_{(1)},
 \\
 &\qquad 
 \psi_q\big(1, u^{I_4}_{(1)},
 (u^{(I_5)}_{(2)}  \overbracket{m_{(3)[0]}}^{x_{(2)[0]}})_{[0]} \big)_{(-2)}
 u^{(I_4)}_{(2)} (u^{(I_5)}_{(2)}  \overbracket{m_{(3)[0]}}^{x_{(2)[0]}})_{[1]},
 u^{I_5}_{(1)}, m_{(1)}\pig)
 \\
 &\qquad
 \cdot_\zett
  \pig(
 u^{(I_1)}_{(2)} 
 \chi_r\pig(1,  u^{I_2}_{(1)},
 \big( u^{(I_3)}_{(2)}  \psi_q\big(1, u^{I_4}_{(1)},
  (u^{(I_5)}_{(2)}  \overbracket{m_{(3)[0]}}^{x_{(2)[0]}})_{[0]} \big)_{(-1)}
  u^{(I_{4,5})}_{(3)}  \overbracket{m_{(3)[1]} m_{(2)}}^{x_{(2)[1]}x_{(1)}}
 \big)_{[0]}
  \pig)_{(0)}
  \pig)
  \\
  &\qquad
 \cdot_\zett \pig( u^{(I_{1,2,3})}_{(3)} 
  \psi_q\big(1, u^{I_4}_{(1)},    (u^{(I_5)}_{(2)}  \overbracket{m_{(3)[0]}}^{x_{(2)[0]}})_{[0]} \big)_{(0)}
     \pig)
     \\[2mm]
\end{align*}
\begin{align*}
     =  \ &
   \phi_p\pig(u^{I_1}_{(1)},
   \chi_r\pig(1,  u^{I_2}_{(1)},
  \big( u^{(I_3)}_{(2)}  \psi_q\big(1, u^{I_4}_{(1)},
  (u^{(I_5)}_{(2)}  \overbracket{m_{(2)}}^{x_{(1)}})_{[0]} \big)_{(-1)}
  u^{(I_{4,5})}_{(3)}  \overbracket{m_{(3)}}^{x_{(2)}}
 \big)_{[0]}
  \pig)_{(-1)}
  \\
  &\qquad
  u^{(I_2)}_{(2)}
 \big(u^{(I_3)}_{(2)}  \psi_q\big(1, u^{I_4}_{(1)},
 (u^{(I_5)}_{(2)}  \overbracket{m_{(2)}}^{x_{(1)}})_{[0]} \big)_{(-1)}
 u^{(I_{4,5})}_{(3)}  \overbracket{m_{(3)}}^{x_{(2)}}\big)_{[1]},
 u^{I_3}_{(1)},
  \\
 &\qquad
 \psi_q\big(1, u^{I_4}_{(1)},
 (u^{(I_5)}_{(2)}  \overbracket{m_{(2)}}^{x_{(1)}})_{[0]} \big)_{(-2)}
 u^{(I_4)}_{(2)} (u^{(I_5)}_{(2)}  \overbracket{m_{(2)}}^{x_{(1)}})_{[1]},
 u^{I_5}_{(1)}, m_{(1)}\pig)
 \\
 &\qquad
 \cdot_\zett
  \pig(
 u^{(I_1)}_{(2)} 
 \chi_r\pig(1,  u^{I_2}_{(1)},
 \big( u^{(I_3)}_{(2)}  \psi_q\big(1, u^{I_4}_{(1)},
  (u^{(I_5)}_{(2)}  \overbracket{m_{(2)}}^{x_{(1)}})_{[0]} \big)_{(-1)}
  u^{(I_{4,5})}_{(3)}  \overbracket{m_{(3)} }^{x_{(2)}}
 \big)_{[0]}
  \pig)_{(0)}
  \pig)
  \\
  &\qquad
 \cdot_\zett \pig( u^{(I_{1,2,3})}_{(3)} 
  \psi_q\big(1, u^{I_4}_{(1)},    (u^{(I_5)}_{(2)}  \overbracket{m_{(2)}}^{x_{(1)}})_{[0]} \big)_{(0)}
     \pig),
\end{align*}
which is the same as the
left hand side of \eqref{thisequaltothat} computed a moment or two ago
(and if it is not, it is merely a typo); hence, this
proves the first line of the partial operadic associativity \eqref{danton}.
Verifying the second line in \eqref{danton} is left to the reader, and the third line is actually redundant as it says the same as the first (but the associativity properties are easier to read or memorise written this way). This
concludes the proof that $\cO$ is an operad with multiplication.

That the differential $\gd$ on the complex \eqref{jour} can be expressed as $\gd \phi = (-1)^{\phi+1} \{ \mu , \phi\}$ additionally uses the property \eqref{radicchio5a}. That an operad with multiplication induces a Gerstenhaber structure on cohomology follows from a well-known (by now) classical result \cite{GerSch:ABQGAAD}, {\em cf.}~Appendix \S\ref{pamukkale1}.
\end{proof}

  \section{Examples}
  \label{examples}

  \subsection{Hopf algebra cohomology and adjoint (co)\-rep\-re\-sen\-ta\-tions}
It is a standard result that any Hopf algebra $H$ can be seen as a braided commutative monoid in $\hyd$ when using the left adjoint action
$H \otimes_{\, \kkk} H \to H, \ h \otimes_{\, \kkk} g \mapsto h_{(1)} g S(h_{(2)})$ 
  and the coproduct in $H$ as coaction, a situation we denote by $\ad(H)$. Hence, Theorem \ref{extobroid} implies that the cohomology groups
  $ \Ext^\bull_H(k, \ad(H)) $ form a Gerstenhaber algebra, a fact already noted in \cite[Ex.~3.5]{Kow:BVASOCAPB}. What is more, in {\em loc.~cit.}~it has been shown that the customary $\kkk$-module isomorphism
  \begin{equation*}
  \Ext^\bull_H(\kkk, \ad(H))
  = \Ext^\bull_{\He}(H,H)
\end{equation*}
  that follows by applying \cite[Thm.~VIII.3.1]{CarEil:HA} is not only an isomorphism on the level of chains, but remarkably enough also an isomorphism of Gerstenhaber algebras, where the right hand side as the Hochschild cohomology of $H$ seen as a $\kkk$-algebra is a Gerstenhaber algebra by the primordial construction in \cite{Ger:TCSOAAR}.

  If the antipode of $H$ is invertible, then an equally standard result says that $H$ is a braided cocommutative comonoid in $\hydr$ by means of the right adjoint coaction $H \to H \otimes_{\, \kkk} H, \ h \mapsto h_{(2)}  \otimes_{\, \kkk} h_{(3)} S^{-1}(h_{(1)})$ and by the action given by the multiplication in $H$, a situation we denote by $\coad(H)$. Hence,
  $
  \Ext^\bull_H(\coad(H), \kkk)
  $
by Theorem \ref{extobroid}  becomes a Gerstenhaber algebra as well.  It would be interesting to see whether these coadjoint action examples are related by any means to Hochschild cohomology (with nontrivial coefficients) as happens in the case of the adjoint action above.

  As a final remark, putting both the adjoint and the coadjoint action together, one might be tempted to say that the groups
  $
  \Ext^\bull_H(\coad(H), \ad(H))
  $
  form a Gerstenhaber algebra. Yet, Theorem \ref{extobroid} requires  $(\coad(H), \ad(H))$ to be a commuting pair for this to be true, and this only happens if $H$ is simultaneously commutative and cocommutative. But then the adjoint and the coadjoint actions are trivial and so is $ \Ext^\bull_H(\coad(H), \ad(H))$, which means that $ \Ext^\bull_H(\coad(H), \ad(H))$ is indeed a Gerstenhaber algebra, but a trivial one.

\subsection{Generalised Drinfel'd and Heisenberg doubles}
By means of a (not necessarily nondegenerate) Hopf pairing
$
\gvf\colon H \otimes_{\, \kkk} G \to \kkk
$
between arbitrary Hopf algebras $H$ and $G$, one can construct a generalised Drinfel'd double $\cD(H, G)$
that can be seen as a bicrossed product $H \bowtie G$
on the underlying vector space $H \otimes_{\, \kkk} G$, see \cite[\S X.2]{Kas:QG} or \cite[Thm.~3.2]{KasRosTur:QGAKI} for full details, we only give a few hints here in a slightly different convention.

The Hopf pairing allows for a right $H$-action on $G$, resp.\ for a left $G$-action on $H$, given by
\begin{equation}
  \label{waterman}
\begin{array}{rl}
  G \otimes_{\, \kkk} H \to G, & g \otimes_{\, \kkk} h \mapsto g \leftslice h \coloneq \gvf(h, g_{(2)}) g_{(1)},
  \\
G \otimes_{\, \kkk} H \to H, & g \otimes_{\, \kkk} h \mapsto g \rightslice h \coloneq \gvf(S^{-1}(h_{(1)}), g) h_{(2)},
\end{array}
\end{equation}
respectively. The generalised Drinfel'd double $\cD(H, G) = H \bowtie G$ is then a Hopf algebra when equipped with the coproduct, resp.\ counit, given by the standard coproduct, resp.\ counit, on the tensor product of $H$ and $G$,
and the product
$$
(h \otimes_{\, \kkk} g) \cdot (h' \otimes_{\, \kkk} g')
\coloneq h(g_{(1)} \rightslice h'_{(1)})  \otimes_{\, \kkk} (g_{(2)} \leftslice h'_{(2)}) g',  
$$
making the two Hopf algebras $H$ and $G$ Hopf subalgebras by means of the canonical injection. 
In contrast to the classical construction that uses $H = (G^\op)^*$ if $G$ is finite dimensional, $\cD(H, G)$ is not necessarily a braided Hopf algebra but is so if both $H, G$ are finite dimensional and the braiding $\gvf$ is nondegenerate (which is, of course, the case for the canonical pairing between $G$ and its dual $G^*$ but, for example, not so if one uses the trivial pairing induced by the two counits of two arbitrary Hopf algebras $H$ and $G$).


In the same spirit, one can define a generalised Heisenberg double $\cH(H,G)$ as a generalisation of both the Heisenberg algebra and the Heisenberg double for $H = G^*$ in case of finite dimensions. This is essentially the smash product $H \# G$ with respect to the left action from \eqref{waterman}, that is, the tensor product
$H \otimes_{\, \kkk} G$
with product
$$
(h \otimes_{\, \kkk} g) \cdot (h' \otimes_{\, \kkk} g')
\coloneq h(g_{(1)} \rightslice h')  \otimes_{\, \kkk} g_{(2)} g'.  
$$
Along the lines of the customary case for $H = G^*$, as for example spelled out in \cite{Sem:HDHBSAABCYDMAOTDD},
one can define a left action, resp.\ a left coaction, on the Heisenberg double $H \# G$ of, resp.\ over, the Drinfel'd double $H \bowtie G$ that turn $H \# G$ into a braided commutative monoid in the category of left-left Yetter-Drinfel'd modules over $H \bowtie G$, and hence the cohomology groups
$$
\Ext^\bull_{\cD(H,G)}(\kkk, \cH(H,G)) = \Ext^\bull_{H \bowtie G}(\kkk, H \# G) 
$$
carry the structure of a  Gerstenhaber algebra. 

Infinite dimensional versions of the Heisenberg double appear, for example,
in \cite{MelSkoSto:LATNCPSAHA} in the study of noncommutative phase spaces of Lie type as being isomorphic to $U(\g) \# U(\g)^* \simeq U(\g) \# \hat S(\g^*)$, where on the right one uses a suitable completion of the symmetric algebra.
Other applications include the extension of the Hall algebra formalism to derived categories of an abelian category as in \cite{Kap:HDADC}.

  \subsection{Crossed product bialgebroids}
  If $H$ is a $\kkk$-bialgebra and $Z$ a braided commutative monoid in $\hydr$, the smash product $Z \# H$ is a left {\em bialgebroid} over $Z$ as shown in \cite[Thm.~4.1]{BrzMil:BBAD}.
  As $Z$ is the monoidal unit in the category of left $Z \# H$-modules, this hence implies that 
  $\Ext^\bull_{Z \# H}(Z,Z)$ is a Gerstenhaber algebra as any base algebra of a bialgebroid $U$ is trivially a braided commutative monoid in $\yd$.

  On the other hand, if $H$ is in addition a Hopf algebra with invertible antipode $S$, by turning the right $H$-coaction on $Z$ into a left one by using $S$ and the left $H$-action into a left $H^\op$-action by means of $S^{-1}$, the opposite algebra $Z^\op$ becomes a braided commutative monoid in ${}^{\scriptscriptstyle{H^\op}}_{\scriptscriptstyle{H^\op}}\mathbf{YD}$, and hence   
  $\Ext^\bull_{H^\op}(\kkk,Z^\op)$ is a Gerstenhaber algebra as well by Theorem \ref{extobroid}.

  Moreover, one has a $\kkk$-module isomorphism between the cochain complexes computing   $\Ext^\bull_{Z \# H}(Z,Z)$ and $\Ext^\bull_{H^\op}(\kkk,Z^\op)$,
  which we indicate here for shortness in exposition and better readability in degree two only:
  \begin{eqnarray*}
\eta\colon    \Hom_{Z^\op}\big( (Z \# H) \otimes_{V^\op} (Z \# H), Z\big) & \stackrel{\simeq}{\to} &    \Hom_{\, \kkk}( H \otimes_{\, \kkk} H, Z^\op),
    \\
    f & \mapsto & \big\{\, h \otimes_{\, \kkk} h' \mapsto f\big((1 \# h) \otimes_{V^\op} (1 \# h') \big)  \, \big\},
    \end{eqnarray*}
  with inverse
  \begin{equation*}
\begin{split}
\tilde f  \mapsto \big\{ \, & (z \# h) \otimes_{V^\op} (z' \# h') 
\\
& \quad \mapsto 
\tilde f\pig(S\big(z_{[1]}(h_{(1)}z'_{[0]})_{[1]}\big) h_{(2)} z'_{[1]} \otimes_{\, \kkk} S(z'_{[2]})h' \pig)\cdot_\zett z_{[0]} \cdot_\zett (h_{(1)}z'_{[0]})_{[0]}  \, \big\}, 
\end{split}
\end{equation*}
  where we used the definition of the bialgebroid cochain complex in \eqref{jour} along with the target map on the left $Z$-bialgebroid $Z \# H$  given by the right $H$-coaction $z \mapsto z_{[0]} \otimes_{\, \kkk} z_{[1]}$, as detailed in \cite[Thm.~4.1]{BrzMil:BBAD}.
  
  Moreover, one can show that $\eta$ and its inverse are chain maps (easy for $\eta$, naturally more involved for its inverse) and hence one has
  \begin{equation}
    \label{lastequationIwanttowrite}
\Ext^\bull_{Z \# H}(Z,Z) \simeq \Ext^\bull_{H^\op}(\kkk,Z^\op).
  \end{equation}
  Finally, the induced isomorphism \eqref{lastequationIwanttowrite} can be, with some effort, shown to be an isomorphism
  of Gerstenhaber algebras, the detailed exposition of which being, however, not within the scope of this article. 

\appendix
\addtocontents{toc}{\protect\setcounter{tocdepth}{1}}

\section{Centres, bialgebroids, and operads}
\label{appendix}

\subsection{Centres in monoidal categories}
\label{centrappendix}

The following standard definition can be found, {\em e.g.}, in \cite[Def.~4.3]{Schau:DADOQGHA} or \cite[Def.~7.13.1]{EtiGelNikOst:TC}:

\begin{dfn}
\label{centres}
The {\em left weak centre} ${\mathscr{Z}}^\ell({\mathscr{C}})$ of a monoidal category $({\mathscr{C}}, \otimes, \mathbb{1})$ is the category whose objects are pairs $(Z, \gs_{\zett,-})$, where $Z \in {\mathscr{C}}$, such that 
$$
\gs_{\zett,\emme}\colon Z \otimes M \to M \otimes Z 
$$
is natural in $M \in {\mathscr{C}}$ and such that
$
\gs_{\zett,\emme {\scriptscriptstyle{\otimes}} \emme'} = (M \otimes \gs_{\zett,\emme'})(\gs_{\zett,\emme} \otimes M')
$
holds
for all $M, M' \in {\mathscr{C}}$
(which amounts to a hexagon axiom if involving associators), along with $\gs_{\zett, \mathbb{1}} = \id_\zett$.
\end{dfn}

The left weak centre is a monoidal category again, in particular a braided one.
At times, we only write $(Z, \gs)$ instead of $(Z, \gs_{\zett,-})$ for objects in ${\mathscr{Z}}^\ell({\mathscr{C}})$. 

Likewise, one defines the {\em right weak centre} ${\mathscr{Z}}^r({\mathscr{C}})$ of a monoidal category ${\mathscr{C}}$ as the category with objects  $(X, \tau_{-,\ikks})$, where this time
$$
\tau_{\emme,\ikks}\colon M \otimes X \to X \otimes M 
$$
is natural in $M \in {\mathscr{C}}$, 
subject to analogous conditions as above.

\begin{dfn}
  \label{commpair}
  A {\em commuting pair} in a monoidal category ${\mathscr{C}}$ is a pair $(X, Z) \in {\mathscr{Z}}^r({\mathscr{C}}) \times {\mathscr{Z}}^\ell({\mathscr{C}})$ such that
  $$
\gs_{\zett,\ikks} = \tau_{\zett,\ikks}
  $$
  holds as maps $Z \otimes X \to X \otimes Z$.
  \end{dfn}

\begin{rem}
  Since we wanted to deal with left and right centres and hence involve two different braided monoidal categories, the above definition is a slight variation of the following more standard one:
  in a braided monoidal cat\-e\-go\-ry $({\mathscr{Z}}, \gs)$, a pair  $(X, Z)$ of its objects is called a {\em commuting pair} or a pair of {\em commuting objects} if $\sigma_{\zett,\ikks} \circ \sigma_{\ikks,\zett} = \id_{\ikks {\scriptscriptstyle \otimes} \zett}$. For example, 
the {\em braided tensor product algebra} (see \cite[Lem.~2]{Bae:HHIABTC})
of two braided commutative monoids $X, Z$ in $({\mathscr{Z}}, \sigma)$ is not necessarily braided commutative again but is so if $(X, Z)$ is a commuting pair, see {\em op.~cit.}, Lemma 3. See Remark \ref{bridge} for further discussion.
  \end{rem}

\subsection{Left bialgebroids, their (co)modules, and centres}
\label{biappendix}
In this subsection, we gather all the necessary material on bialgebroids, their modules, their comodules, the left and right weak centres in question and how explicitly commuting pairs translate to this context. 

\subsubsection{Bialgebroids}
\label{bialgebroids1}
A {\em left bialgebroid} $(U, A, \gD, \gve, s, t)$, or $(U,A)$ for short, is a generalisation of a $\kkk$-bialgebra to a bialgebra object over a noncommutative base ring $A$ (typically a $\kkk$-algebra), and consists of a $\kkk$-algebra $U$ which is simultaneously an $\Ae$-ring and an $A$-coring that are compatible in mostly the standard way, at least formally. In particular, one has a ring homomorphism resp.\ antihomomorphism $s,t \colon A \to U$ ({\em source} resp.\ {\em target}) inducing four commuting $A$-module structures on $U$
\begin{equation}
  \label{pergolesi}
  a \blact b \lact u \ract c \bract d \coloneq t(c)s(b)us(d)t(a)
\end{equation}
  for $u \in U, \, a,b,c,d \in A$, abbreviated by symbols like
$
\due U {\blact \lact} {\ract \bract}
$
or similar, depending on the relevant action(s) in a specific situation. Also introduce the $\Ae$-ring
$$
U \times_{\scriptscriptstyle A} U   \coloneq
   \big\{ {\textstyle \sum_i} u_i \otimes  v_i \in U_{\!\ract}  \otimes_A  \!  \due U \lact {} \mid {\textstyle \sum_i} a \blact u_i \otimes v_i = {\textstyle \sum_i} u_i \otimes v_i \bract a,  \ \forall a \in A  \big\}
$$
the {\em Sweedler-Takeuchi product}, which allows to define on $U$ a 
comultiplication
$$
 \gD \colon U \to U \times_A  U \subset \due U {} \ract \otimes_A  \due U \lact {}, \ u \mapsto u_{(1)} \otimes_A  u_{(2)},
$$
along with a counit $\gve \colon U \to A$ subject to certain identities that
 at some points differ from those in the bialgebra case, see \cite{Tak:GOAOAA} for the original construction.

 \subsubsection{Example}
   The simplest example of a (noncommutative noncocommutative) left bialgebroid is given by $(\Ae, A)$, where $\Ae \coloneq A \otimes_{\hskip .7pt\kkk} \Aop$ is the enveloping algebra of an associative $\kkk$-algebra $A$ (which, in particular, allows to consider Schwede's approach in \cite{Schw:AESIOTLBIHC} as a special case of our results in the main text). Its structure maps are,
for any $a, b \in A$,
   given by
 \begin{equation}
   \label{purple}
   \begin{array}{rclcrcl}
     \gD(a \otimes_{\hskip .7pt \kkk}
     b) &=& (a \otimes_{\hskip .7pt \kkk} 1) \otimes_A (1 \otimes_{\hskip .7pt \kkk} b),
     & \gve(a \otimes_{\hskip .7pt \kkk} b) &=& ab,
     \\
     s(a) &=&  a \otimes_{\hskip .7pt \kkk} 1, &      t(a) &=&  1 \otimes_{\hskip .7pt \kkk} a,
        \end{array}
   \end{equation}
 along with the factorwise multiplication on $A \otimes_{\hskip .7pt \kkk} \Aop$ for the product.

\subsubsection{Bialgebroid modules and comodules}
 
A {\em left $U$-module} $M$ over a left bialgebroid $(U,A)$ is a left module over the underlying ring $U$. The category of left $U$-modules will be denoted by $\umod$ and we usually write all $U$-actions just by juxtaposition.
The forgetful functor $\umod
\to \amoda$ is induced by
\begin{equation}
  \label{forgetthis}
a \lact m \ract b \coloneq s(a)t(b) m, \qquad \forall\, a,b \in A, \, m \in M,
\end{equation}
with respect to which one forms the tensor product $M \otimes_A M'$ of left $U$-modules, and which, similar to the bialgebra case, is a left $U$-module again by
\begin{equation}
\label{rain}
u(m \otimes_A m') \coloneq \gD(u)(m \otimes_A m') = u_{(1)} m \otimes_A u_{(2)}m',
\end{equation}
for $u \in U, \, m \in M, \, m' \in M'$. In other words, $\umod$ is a (strict) monoidal category with $\otimes_A$ as monoidal product and $A$ as its unit object.

A {\em left} (and analogously {\em right}) {\em comodule} over a left bialgebroid $(U,A)$ is a comodule over the underlying $A$-coring \cite[\S3]{BrzWis:CAC}: a left $A$-module $M$ along with a coassociative and counital coaction
$\gl \colon M \to U_\ract \otimes_A  M, \ m \mapsto m_{(-1)} \otimes_A  m_{(0)}$,
which, by defining the right $A$-action $ ma\coloneq \gve(m_{(-1)} \bract a) m_{(0)}$ for all $a \in A$ on $M$, effectively
corestricts to a map
$
\gl: M \to U_\ract \times_A  M,
$
where 
\begin{equation*}
    \label{takcom2}
   U_\ract \times_{\scriptscriptstyle A} M   \coloneq
   \big\{ {\textstyle \sum_i} u_i \otimes  m_i \in U_\ract  \otimes_A  M
    \mid {\textstyle \sum_i} a \blact u_i \otimes m_i = {\textstyle \sum_i} u_i \otimes m_ia,  \ \forall a \in A  \big\}
\end{equation*}
is a subspace in $ U_\ract \otimes_A  M$.
The left coaction is $A$-bilinear in the sense of 
$
\gl(amb) =  a \lact m_{(-1)} \bract b \otimes_A  m_{(0)} 
$
for all $m \in M$, $a, b \in A$. Again, the category $\ucomod$ of left $U$-comodules (resp.~the category $\comodu$ of right $U$-comodules) is a (strict) monoidal category with monoidal product $\otimes_A$ and unit object $A$.

\subsubsection{} 
\label{ydydydyd}
\hspace*{-.2cm}
A {\em left-left Yetter Drinfel'd (YD) module} $Z$ over a left bialgebroid $(U,A)$ is
simultaneously a left $U$-module (with action denoted by juxtaposition) and a left $U$-comodule with coaction
$\gl\colon Z \to U_\ract \otimes_A Z, \ z \mapsto z_{(-1)} \otimes_A z_{(0)}$ such that the two forgetful functors $\umod \to \amoda$ and $\ucomod \to \amoda$ induce the same $A$-bimodule structure on $Z$, and such that 
\begin{equation}
  \label{yd}
u_{(1)} z_{(-1)} \otimes_A u_{(2)} z_{(0)} = (u_{(1)} z)_{(-1)}  u_{(2)} \otimes_A  (u_{(1)} z)_{(0)}  
\end{equation}
holds for all $u \in U$ and $z \in Z$, the corresponding category of which is braided  monoidal (see below) with respect to $\otimes_A$, unit object given by the base algebra $A$, and denoted by $\yd$.

\subsubsection{} 
\label{ydydydydr}
\hspace*{-.2cm}
Likewise, 
a {\em left-right Yetter Drinfel'd (YD) module} $X$ over a left bialgebroid $(U,A)$ is
simultaneously a left $U$-module (with action denoted by juxtaposition) and a right $U$-comodule with coaction $\gr\colon X \to X \otimes_A \due U \lact {}, \ x \mapsto x_{[0]} \otimes_A x_{[1]}$ such that the two forgetful functors $\umod \to \amoda$ and $\comodu \to \amoda$ induce the same $A$-bimodule structure on $X$, and such that 
\begin{equation}
  \label{yd2}
u_{(1)} x_{[0]} \otimes_A u_{(2)} x_{[1]} = (u_{(2)} x)_{[0]}  \otimes_A  (u_{(2)} x)_{[1]}u_{(1)}  
\end{equation}
holds for all $u \in U$ and $x \in X$, the corresponding category of which is braided  monoidal (see below) with respect to $\otimes_A$, unit object given by the base algebra $A$, and denoted by $\ydr$.

\subsubsection{} 
\hspace*{-.2cm}
Observe that, in contrast to bialgebras, there are {\em no} right-left or right-right YD modules over left bialgebroids.

\subsubsection{Weak centres and commuting pairs for left bialgebroids} 
\label{resaid}
As mentioned at the beginning of \S\ref{cadutamassi}, a well-known fact \cite[Prop.~4.4]{Schau:DADOQGHA} establishes an equivalence of braided monoidal categories between the left weak centre 
${\mathscr{Z}}^\ell(\umod)$
of the monoidal category $\umod$ in the sense of Definition \ref{centres},
and the category $\yd$, and analogously between the right weak centre ${\mathscr{Z}}^r(\umod)$ and $\ydr$: as for the first case, assign to any $Z \in \yd$ its underlying left module $Z \in \umod$ along with the {\em left braiding}
\begin{equation}
     \label{ydbraid}
 \gs =    \gs_{\zett,\emme}\colon Z \otimes_A M \to M \otimes_A Z,
     \quad     z \otimes_A m \mapsto z_{(-1)}m \otimes_A z_{(0)}
\end{equation}
for any $M \in \umod$, 
to form an object $(Z, \sigma)$ in ${\mathscr{Z}}^\ell(\umod)$. Much the same way, assign to any $X \in \ydr$ its underlying left module $X \in \umod$ along with the {\em right braiding}
 \begin{equation}
     \label{ydbraid2}
   \tau =   \tau_{\emme,\ikks}\colon M \otimes_A X \to X \otimes_A M,
     \quad
     m \otimes_A x \mapsto x_{[0]} \otimes_A x_{[1]}m
   \end{equation}
for any $M \in \umod$,
to give an object $(X, \tau)$ in ${\mathscr{Z}}^r(\umod)$.

A pair $(X,Z)$ of objects in
 ${\mathscr{Z}}^r(\umod) \times {\mathscr{Z}}^\ell(\umod)$ resp.\ $\ydr \times \yd$ is then called a {\em commuting pair} in the sense of Definition \ref{commpair} if
\begin{equation}
  \label{centralascentralcan1}
  \tau_{\zett,\ikks} = \sigma_{\zett,\ikks}
\end{equation}
holds, or, if for all
$x \in X$ and
$z \in Z$
\begin{equation}
  \label{centralascentralcan2}
x_{[0]}  \otimes_A x_{[1]}z = 
  z_{(-1)} x \otimes_A z_{(0)} 
\end{equation}
is true,
using explicitly the braidings \eqref{ydbraid2} and \eqref{ydbraid}, respectively.

\begin{rem}
  \label{bridge}
  If the braiding \eqref{ydbraid} is invertible, which happens if more structure is present as, for example, if a bialgebra is a Hopf algebra with invertible antipode, one sets $\tau_{\emme,\zett} = \gs^{-1}_{\emme, \zett}$, and left and right weak centres are equivalent as braided monoidal categories, which we then refer to simply as {\em the} centre, also called the {\em Drinfel'd} centre and denoted here by ${\mathscr{Z}}_1(\umod)$.

  As a braided monoidal category this has a centre again (the centre of a centre, so to speak), sometimes referred to as {\em M\"uger} or {\em Rehren} centre, defined by those objects $Z$ in ${\mathscr{Z}}_1(\umod)$ for which
$\gs_{\zett, \emme} = \gs^{-1}_{\emme, \zett}$ for all $M \in \umod$, and denoted  ${\mathscr{Z}}_2(\umod)$, see \cite[Def.~5.12]{Mue:FSTCATIITQDOTCAS} or \cite{Reh:BGSATSR} for more information. The category  ${\mathscr{Z}}_2(\umod)$ is obviously symmetric monoidal (and can be defined for any monoidal category ${\mathscr{C}}$).
  By mere definition, each pair of objects in ${\mathscr{Z}}_2(\umod)$ yields a  commuting pair in the sense of \eqref{centralascentralcan1} but at least for ${\mathscr{C}} = \umod$ not very interesting ones: these are essentially $A$-modules (or direct sums thereof) as is the case if $U$ were a Hopf algebra over $A = k$.

 However, the notion of commuting pairs is more general: for example, both $(M, A)$ and $(A,N)$ for arbitrary $M, N \in {\mathscr{Z}}_1(\umod)$ obviously are commuting pairs, but not necessarily in ${\mathscr{Z}}_2(\umod)$. Slightly less trivial, for a {\em cocommutative} left bialgebroid, each $M \in \umod$ can be made into a (left-left) YD module by means of the trivial left coaction and hence each pair $(M,N)$ of left $U$-modules
  yields a pair of commuting objects.
\end{rem}

\subsubsection{Braided commutative monoids in $\yd$}
\label{arancia}
A braided commutative monoid in $\yd$ is, roughly speaking, an object $Z$ that has four properties: it is a (left-left) YD module, it is a left $U$-module algebra, a left $U$-comodule algebra, and finally it is braided commutative.

More in detail, writing the action and coaction of an object $Z \in \yd$ as in \S\ref{ydydydyd}, assume that
$(Z,   \mu , 1_\zett)$
with multiplication $z \cdot_\zett z' \coloneq   \mu (z,z')$ and unit $1_\zett$ is an $A$-ring. Being a left $U$-module algebra then explicitly means
\begin{equation}
  \label{radicchio1}
u(z \cdot_\zett z') = (u_{(1)} z) \cdot_\zett (u_{(2)}z'), 
\end{equation}
along with
\begin{equation}
  \label{radicchio1a}
u1_\zett = \gve(u) \lact 1_\zett.
\end{equation}
On the other hand,
$(Z,   \mu , 1_\zett)$
also being a left $U$-comodule algebra implies
\begin{equation*}
  \label{radicchio2}
\gl_\zett(z \cdot_\zett z') = \gl_\zett(z) \gl_\zett(z') =   z_{(-1)} z'_{(-1)} \otimes_A z_{(0)} \cdot_\zett z'_{(0)},
  \end{equation*}
along with
$ 
  \label{radicchio2a}
\gl_\zett(1_\zett) = 1_\uhhu \otimes_A 1_\zett.
$ 
Moreover, 
\begin{equation}
  \label{radicchio2b}
  \mu  \circ \gs =   \mu, 
\end{equation}
which says that $(Z,   \mu , 1_\zett)$ is supposed to be braided commutative; explicitly,
\begin{equation}
  \label{radicchio3}
z \cdot_\zett z' = (z_{(-1)} z') \cdot_\zett z_{(0)}.
\end{equation}
Observe that it is precisely (and tautologically) the YD condition \eqref{yd} on $Z$ that guarantees \eqref{radicchio3} to be well-defined, even when thinking of $U$-modules:
\begin{equation}
  \label{radicchio3a}
\begin{split}
  u(z \cdot_\zett z')
  &
  \stackrel{\phantom{\scriptscriptstyle \eqref{radicchio3}}}{=}
  u_{(1)} z \cdot_\zett u_{(2)}z'
  \stackrel{\scriptscriptstyle \eqref{radicchio3}}{=}
  (u_{(1)} z)_{(-1)} u_{(2)} z' \cdot_\zett (u_{(1)} z)_{(0)}
  \\
  &
  \stackrel{\scriptscriptstyle \eqref{yd}}{=}
  u_{(1)} z_{(-1)} z' \cdot_\zett u_{(2)} z_{(0)}
  \stackrel{\phantom{\scriptscriptstyle \eqref{radicchio3}}}{=}
  u \big((z_{(-1)} z') \cdot_\zett z_{(0)}\big).
  \end{split}
    \end{equation}

Note that there is no relation whatsoever to whether the multiplication $  \mu $ in $Z$ itself is commutative or not: for example, the base algebra $A$ in a left bialgebroid itself is a braided commutative monoid in $\yd$ but not necessarily commutative as a $\kkk$-algebra.

\subsubsection{Braided cocommutative comonoids in $\ydr$}
\label{mandarino}
In much the same way as in \S\ref{arancia}, a braided cocommutative comonoid in $\ydr$ (equivalently, a braided commutative monoid in $(\ydr)^\op$) is again an object equipped with four properties: it is a (left-right) YD module, it is a left $U$-module coalgebra as well as a right $U$-comodule coalgebra, and finally it is braided cocommutative.

More precisely, writing the action and coaction on an object $X \in \ydr$ as in \S\ref{ydydydydr},
assume that 
$(X, \gD_\ikks, \gve_\ikks)$
is an $A$-coring
with coproduct $\gD_\ikks \colon X \to X \otimes_A X, \  x \mapsto x_{(1)} \otimes_A x_{(2)}$ and counit $\gve_\ikks\colon X \to A$.
Being a left $U$-module coalgebra then explicitly means
\begin{equation}
  \label{radicchio4}
\gD_\ikks(ux) = \gD(u) \gD_\ikks(x) = u_{(1)} x_{(1)} \otimes_A u_{(2)} x_{(2)}, 
\end{equation}
along with
\begin{equation}
  \label{radicchio4a}
\gve_\ikks(ux) = \gve(u \bract \gve_\ikks(x)).
\end{equation}
On the other hand, $(X, \gD_\ikks, \gve_\ikks)$ is also 
a right $U$-comodule coalgebra, which
means
$
(\gD_\ikks \otimes_A U) \circ \gvr_\ikks = \gvr_{\scriptscriptstyle{X \otimes_A X}} \circ \gD_\ikks, 
$
and which amounts to 
\begin{equation*}
  \label{radicchio5}
x_{[0](1)} \otimes_A x_{[0](2)} \otimes_A x_{[1]} = x_{(1)[0]} \otimes_A x_{(2)[0]} \otimes_A x_{(2)[1]} x_{(1)[1]},
\end{equation*}
for any $x \in X$, along with
\begin{equation}
  \label{radicchio5a}
t\big( \gve_\ikks(x)\big) = \gve_\ikks(x_{[0]}) \lact x_{[1]}. 
\end{equation}
Finally, braided cocommutativity manifests as
\begin{equation}
  \label{radicchio6a}
\tau \circ \gD_\ikks = \gD_\ikks,
\end{equation}
or, explicitly,
\begin{equation}
  \label{radicchio6}
x_{(2)[0]} \otimes_A x_{(2)[1]} x_{(1)} = x_{(1)} \otimes_A x_{(2)}.
  \end{equation}

\subsection{Operads and Gerstenhaber algebras}
\label{pamukkale1}
A ({\em non-$\gS$}) {\em operad} $\cO$ in the category 
of $\kkk$-modules is a family $\{\cO(n)\}_{n \geq 0}$ of $\kkk$-modules 
with $\kkk$-bilinear operations

\noindent $\circ_i\colon \cO(p) \otimes \cO(q) \to \cO({p+q-1})$, $i = 1, \ldots, p$,
subject to 
\begin{eqnarray}
\label{danton}
\nonumber
\gvf \circ_i \psi &=& 0 \qquad \qquad \qquad \qquad \qquad \! \mbox{if} \ p < i \quad \mbox{or} \quad p = 0, \\
(\varphi \circ_i \psi) \circ_j \chi &=& 
\begin{cases}
(\varphi \circ_j \chi) \circ_{i+r-1} \psi \qquad \mbox{if} \  \, j < i, \\
\varphi \circ_i (\psi \circ_{j-i +1} \chi) \qquad \hspace*{1pt} \mbox{if} \ \, i \leq j < q + i, \\
(\varphi \circ_{j-q+1} \chi) \circ_{i} \psi \qquad \mbox{if} \ \, j \geq q + i.
\end{cases}
\end{eqnarray}

In the main text, especially in \S\ref{int}, these operations are referred to as {\em internal} operadic composition, in contrast to the {\em external} one dealt with in \S\ref{external}.
Call an operad {\em unital} if there is an {\em identity} $\mathbb{1} \in \cO(1)$ such that 
$
\gvf \circ_i \mathbb{1} = \mathbb{1} \circ_1 \gvf = \gvf
$ 
for all $\gvf \in \cO(p)$ and $i \leq p$, and call it {\em with multiplication} if there exists a {\em multiplication}  $ \mu  \in \cO(2)$ and a {\em unit} $e \in \cO(0)$ such that $ \mu  \circ_1  \mu  =  \mu  \circ_2  \mu $ as well as
$ \mu  \circ_1 e =  \mu  \circ_2 e = \mathbb{1}$. Such an object will be denoted by the triple $(\cO,  \mu , e)$. It is then a standard result (see \cite{GerSch:ABQGAAD}) that  $(\cO,  \mu , e)$ defines a cocyclic $\kkk$-module the cohomology $H^\bull(\cO)$ of which yields a Gerstenhaber algebra.


\addtocontents{toc}{\SkipTocEntry}
\section*{Acknowledgements}

The authors wish to thank the referee for careful reading and precious
suggestions. This paper is part of the activities of the MIUR Excellence Department
Projects CUP:B83C23001390001 and MatMod@TOV CUP:E83C23000330006.
D.F.\ has been partially supported by the PRIN 2017 {\em Moduli and Lie
Theory}, Ref.\ 2017YRA3LK.
N.K.\ has been partially supported by the PRIN 2017 {\em Real and Complex
Manifolds: Topology, Geometry, and Holo\-morphic Dynamics}, Ref.\
2017JZ2SW5.
Both authors are members of the {\em Grup\-po Nazionale per le
Strutture Algebriche, Geometriche e le loro Applicazioni} (GNSAGA-INdAM).

\providecommand{\bysame}{\leavevmode\hbox to3em{\hrulefill}\thinspace}
\providecommand{\MR}{\relax\ifhmode\unskip\space\fi M`R }
\providecommand{\MRhref}[2]{%
  \href{http://www.ams.org/mathscinet-getitem?mr=#1}{#2}}
\providecommand{\href}[2]{#2}

\end{document}